\documentclass[11pt,reqno]{amsart}
\usepackage{graphicx}
\usepackage{color}
\usepackage{framed}
\usepackage{extarrows}
\usepackage{amsmath,amssymb,amsthm,amsfonts,
mathrsfs,stmaryrd}
\usepackage{cases,indentfirst}

\usepackage{soul}

\usepackage{multicol}
\usepackage[colorlinks=true, linkcolor=black,filecolor=black,
    urlcolor=black,citecolor=black]{hyperref}
\makeatletter

\newcommand{\Rmnum}[1]{\expandafter\@slowromancap\romannumeral #1@}
\makeatother
\newtheorem{theorem}{Theorem}[section]

\newtheorem{notation}{Notation}[section]
\newtheorem{lemma}{Lemma}[section]

\newtheorem{remark}{Remark}[section]
\usepackage[numbers,sort&compress]{natbib}
\theoremstyle{definition}

\providecommand{\norm}[1]{\left\Vert#1\right\Vert}

\providecommand{\sd}[1]{\mathcal{D}_{#1}}
\providecommand{\se}[1]{\mathcal{E}_{#1}}

\DeclareMathOperator{\diverge}{div}

\providecommand{\norm}[1]{\left\Vert#1\right\Vert}
\def\ls{\lesssim}

\def\dt{\partial_t}

\def\RRvert2{\right \vert\! \right\vert}
\def\Lvert3{\left \vert\!\left\vert\!\left\vert}
\def\Rvert3{\right \vert\!\right\vert\!\right\vert}

\def\nab{\nabla}

\def\dt{\partial_t}

\def\ls{\lesssim}
\def\p{\partial}

\def\a{\mathcal{A}}
\def\f{\mathcal{F}_{2N}}

\def\fj1{\mathcal{J}^{-1}}

\numberwithin{equation}{section}

\allowdisplaybreaks[4]

\makeatletter
\@namedef{subjclassname@2020}{%
  \textup{2020} Mathematics Subject Classification}
\makeatother

\title[Magneto-micropolar fluids without resistivity and spin viscosity]{Stability of vertically charged steady magnetic field in 3D incompressible magneto-micropolar fluids without magnetic and angular viscosity in a strip domain}

\author[Z. Feng]{Zefu Feng}
\address[Z. Feng]{School of Mathematical Sciences, Chongqing Normal University, China}
\email{zefufeng@mails.ccnu.edu.cn}

\author[G. Hong]{Guangyi Hong$ ^\ast$}\thanks{$ ^\ast $Corresponding  author.}
\address[G. Hong]{School of Mathematics, South China University of Technology, China}
\email{magyhong@scut.edu.cn}

\author[J. Wu]{Jiahong Wu}
\address[J. Wu]{Department of Mathematics, University of Notre Dame, USA}
\email{jwu29@nd.edu}

\author[K. Zhao]{Kun Zhao}
\address[K. Zhao]{College of Mathematical Sciences, Harbin Engineering University, China}
\email{kzhao@hrbeu.edu.cn}



\begin{document}

\begin{abstract}
This paper intends to understand the regularity and stability problem on the 3D incompressible magneto-micropolar equations with zero magnetic and angular viscosities in a strip domain. The magneto-micropolar system models the electrically conducting micropolar fluid in the presence of a magnetic field. The lack of magnetic diffusion and angular dissipation makes it impossible to prove even small data global well-posedness result, let alone general large data global regularity. This paper presents a steady-state setup around which any perturbations can be shown to be globally regular and stable. More precisely, any small perturbation near a steady magnetic field perpendicular to the horizontal boundary leads to a unique global classical solution. In addition, the solution is shown to converge to the steady state at an almost exponential rate as time goes to infinity. These appear to be the very first rigorous global results on the magneto-micropolar equations concerned here.
\end{abstract}

\maketitle

\tableofcontents

 \section{Introduction} \label{sec:introduction}

Consider the initial-boundary value problem (IBVP) for the 3D incompressible magneto-micropolar fluid with no magnetic and angular viscosity:
\begin{subequations}\label{ns_euler0c}
\begin{alignat}{6}
\partial_t\tilde{\mathbf{u}}  +   \tilde{\mathbf{u}}\cdot \nabla \tilde{\mathbf{u}}  +\nabla \Big(\tilde{p}+\frac{1}{2}\vert \tilde{\mathbf{b}}\vert^2 \Big)
 &=(\mu+\zeta)\Delta \tilde{\mathbf{u}}+\tilde{\mathbf{b}} \cdot \nabla \tilde{\mathbf{b}}+ \zeta \nabla \times \tilde{\mathbf{w}}, \quad &\mathbf{x}&\in\Omega,\ t>0, \label{ns_euler0c1} \\
  \partial_t\tilde{\mathbf{b}}  +   \tilde{\mathbf{u}} \cdot \nabla \tilde{\mathbf{b}}  &= \tilde{\mathbf{b}} \cdot \nabla \tilde{\mathbf{u}}, &\mathbf{x}&\in\Omega,\ t>0, \label{ns_euler0c2} \\
  \partial _{t}\tilde{\mathbf{w}}+ \tilde{\mathbf{u}}\cdot \nabla
  \tilde{\mathbf{w}} &= - 2 \zeta \tilde{\mathbf{w}} +  \zeta \nabla \times
  \tilde{\mathbf{u}}, &\mathbf{x}&\in\Omega,\ t>0, \label{ns_euler0c3} \\
\mathop{\mathrm{div}}\nolimits{\tilde{\mathbf{u}}}&=
\mathop{\mathrm{div}}\nolimits{\tilde{\mathbf{b}}}=0, &\mathbf{x}&\in\Omega,\ t>0, \label{ns_euler0c4} \\
 (\tilde{\mathbf{u}}, \tilde{\mathbf{b}}, \tilde{\mathbf{w}})\vert_{t=0}&=(\tilde{\mathbf{u}}_{0},\tilde{\mathbf{b}}_0, \tilde{\mathbf{w}}_{0}), \label{ns_euler0c5} \\
\tilde{\mathbf{u}}\vert_{\partial\Omega}&=\mathbf{0}, \label{ns_euler0c6}
\end{alignat}
\end{subequations}
where $ \Omega= \mathbb{R}^{2}\times (0,1) $. This paper aims at the global regularity and stability problem on \eqref{ns_euler0c}. Needless to say, the general
large data global regularity problem for \eqref{ns_euler0c} is out of reach.
The lack of magnetic diffusion also makes the small data global well-posedness appear impossible. The Sobolev norms of the magnetic field generally grows in time and this destroys the smallness framework.

\vskip .1in
This paper presents a scenario for which we do have global, regular and stable solutions. Constant magnetic fields form a special class of steady-state solutions to \eqref{ns_euler0c}. When the background magnetic field is perpendicular to the domain boundary, namely $x_3=0,1$, we are able to establish that any small perturbations near the background magnetic field generates a unique global solution that remains small for all time. Furthermore, as time $t\to \infty$, the perturbation converges to zero at an almost exponential rate.

\vskip .1in
The global regularity and stability problems on the perturbations near the aforementioned background is not trivial. If we follow the conventional approach
to perform energy estimates on \eqref{ns_euler0c}, we encounter two difficulties immediately. The first is the lack of diffusion in the equation of $\tilde{\mathbf{b}}$. The second is the coupling of $\tilde{\mathbf{u}}, \tilde{\mathbf{b}}$ with $\tilde{\mathbf{w}}$ through the curl operator. This paper develops a novel approach based on Lagrangian formulation. The discovering and exploiting of the hyperbolic structures for the system governing the perturbations, and defining a two-tier energy functional are essential ingredients of our proof.

\vskip .1in
We recall some background information. Micropolar fluids are characterized by nonsymmetric stress tensor and exhibit micro-rotational effects and micro-rotational inertia (cf.\,\cite{Eringen}). Physical examples of micropolar fluids include liquid crystals made up of dumbbell particles, animal blood, bubbly liquids, and ferrofluids, etc (cf.\,\cite{erdougan1970polar,Eringen}). Comparing with general microfluids, micropolar fluids are more manageable from the modeling perspective. The motion of an incompressible and electrically conducting micropolar fluid in the presence of a magnetic field in $\mathbb{R}^3$ can be modeled by the following system (cf.\,\cite{Ahmadi,Eringen}):
\begin{subequations}\label{MM}
\begin{alignat}{4}
\partial_t\tilde{\mathbf{u}}  +   \tilde{\mathbf{u}}\cdot \nabla \tilde{\mathbf{u}}  +\nabla \Big(\tilde{p}+\frac{1}{2}\vert \tilde{\mathbf{b}}\vert^2 \Big)
 &=(\mu+\zeta)\Delta \tilde{\mathbf{u}}+\tilde{\mathbf{b}} \cdot \nabla \tilde{\mathbf{b}}+\zeta \nabla \times \tilde{\mathbf{w}}, \label{MM1} \\
  \partial_t\tilde{\mathbf{b}}  +   \tilde{\mathbf{u}} \cdot \nabla \tilde{\mathbf{b}}  &= \tilde{\mathbf{b}} \cdot \nabla \tilde{\mathbf{u}} + \nu\Delta \tilde{\mathbf{b}}, \label{MM2} \\
  \partial _{t}\tilde{\mathbf{w}}+ \tilde{\mathbf{u}}\cdot \nabla
  \tilde{\mathbf{w}} &= - 2 \zeta \tilde{\mathbf{w}} + \zeta \nabla \times
  \tilde{\mathbf{u}} + \kappa_1\Delta \tilde{\mathbf{w}} +
  (\kappa_2+\kappa_3)\nabla(\nabla\cdot\tilde{\mathbf{w}}), \label{MM3} \\
\mathop{\mathrm{div}}\nolimits{\tilde{\mathbf{u}}}&=
\mathop{\mathrm{div}}\nolimits{\tilde{\mathbf{b}}}=0, \label{MM4}
\end{alignat}
\end{subequations}
where $\mathbf{x} \in \mathbb{R}^3$, $t>0$, which results from the coupling of the incompressible magnetohydrodynamics (MHD) equations and the angular momentum equation through vorticities. Here, $\tilde{\mathbf{u}}=(\tilde{u} _{1},\tilde{u} _{2}, \tilde{u} _{3}) $, $\tilde{p}$, $ \tilde{\mathbf{b}}=( \tilde{b} _{1}, \tilde{b} _{2}, \tilde{b} _{3}) $, $ \tilde{\mathbf{w}}=(\tilde{w} _{1},\tilde{w} _{2},\tilde{w} _{3}) $ denote the fluid velocity, pressure, magnetic field, and angular velocity, respectively, and $\mu $ represents the kinematic viscosity, $ \zeta $ the vortex viscosity, $\nu$ the magnetic viscosity (permeability), and $\kappa_1$, $\kappa_2$ and $\kappa_3$ the angular (micro-rotational) viscosities. Following conventional terminology, we shall term system \eqref{MM} as the magneto-micropolar equations.

\vskip .1in
Mathematical study of the magneto-micropolar equations has been carried out since the 1970s. We recall some results that are relevant to this work. For the fully dissipative model \eqref{MM} (i.e., $\mu,\nu,\zeta,\kappa_1,\kappa_2,\kappa_3>0$), the existence, uniqueness and regularity of solutions were studied by Rojas-Medar-Boldrini \cite{RB-1998} and Xiang-Yang \cite{XY-2012}. A blow-up criterion for the 2D Cauchy problem was established by Wang-Wang \cite{WW-2011}. For the 2D model with mixed partial viscosities, the global well-posedness of classical solutions was obtained by Cheng-Liu \cite{CL-2015} and Regmi-Wu \cite{wu-Regmi-2016}. Later, Yamazaki \cite{Y-2015} investigated the global regularity of the 2D system with zero angular viscosities ($\kappa_1=\kappa_2=\kappa_3=0$). Recently, Lin-Xiang \cite{LX-2020} proved the global well-posedness and long-time behavior of strong solutions around the equilibrium $(\mathbf{0}, \mathbf{e}_1, \mathbf{0})$ for the fully dissipative 2D model. More results on the magneto-micropolar equations can be found in \cite{Ma1,Sh0, Sh1, Yuan}.

\vskip .1in
To put things into broader perspective, we note that system \eqref{MM} reduces to the incompressible magnetohydrodynamics (MHD) equations when $ \zeta =0 $ and $ \mathbf{w} \equiv \mathbf{0} $. There is a rich literature on the mathematical studies of the incompressible MHD model. For the viscous and resistive system (i.e., $\mu>0$ and $\nu>0$), the local well-posedness of large-data solutions and global well-posedness of small-data solutions were established in the classical Sobolev space $H^s(\mathbb{R}^d)$ ($s\geq d$) by Duvaut-Lions \cite{DL-1972}. The global well-posedness result in the two-dimensional space was later improved by Sermange-Temam \cite{ST-1983} by removing the smallness restriction on the initial data. Cao et al. \cite{CRW-2013,CW-2011} further upgraded the result of \cite{ST-1983} to the 2D model with anisotropic dissipation and resistivity. For the viscous  and non-resistive system (i.e., $\mu>0$ and $\nu=0$), Lin-Xu-Zhang \cite{Lin-Xu-ZHANGPING-2015-JDE}  and Xu-Zhang \cite{XZ-2015} constructed global classical solutions near non-trivial steady magnetic fields to the Cauchy problem in two- and three-dimensional spaces, respectively (see \cite{ZhANG-Ren-2016-nonlinearity,Tan-Wang-SIMA} for similar results in horizontal strip domains). The long-time behaviors of such solutions were investigated in \cite{ZhangZHIFEI-JFA-2016} for the 2D Cauchy problem and in \cite{Tan-Wang-SIMA} for the 3D IBVP in a strip domain. In addition to the mentioned works on the incompressible MHD system, we would also like to refer the readers to \cite{wang-hu-2010-ARMA,Hu-Wang-2008-JDE,hu2014global,jiangsong-fanjishan,ST-1983,zhangjianwen-SIMA,Hong-SIMA} for related results on the compressible  MHD model.

\vskip .1in
In this paper, we are interested in the mathematical analysis of \eqref{MM} when the magnetic and angular viscosities are negligible, i.e., $\nu=\kappa_1=\kappa_2=\kappa_3=0$. Our work is motivated by recent studies on the 2D magneto-micropolar equations without angular viscosities \cite{Y-2015}, the 3D non-resistive MHD equations \cite{Tan-Wang-SIMA}, and the 3D compressible version of \eqref{ns_euler0c} \cite{Feng-hong-zhu-compre}. However, the problem studied here is considerably more complicated than its predecessors. On one side, comparing with the 2D micro-rotationally inviscid (i.e., $\kappa_1=\kappa_2=\kappa_3=0$) version of \eqref{MM}, the higher dimensionality and absence of the magnetic diffusion bring the difficulty of the mathematical analysis of \eqref{ns_euler0c} to a different level. On the other side, in addition to the lack of magnetic diffusion (the main difficulty of the non-resistive MHD), the coupling of the fluid and angular velocity fields through curls presents an anti-symmetric form in the system, which formally leads to the loss of derivatives of $\tilde{\mathbf{w}}$. Such a situation has been encountered in the recent study of the compressible version of \eqref{ns_euler0c} \cite{Feng-hong-zhu-compre}, where the global dynamics of classical solutions to an IBVP of the model is studied via Lagrangian formulation and the two-tier energy method developed in \cite{Guo-LOCAL-2013,GT-per,Tan-Wang-SIMA}. One of the main ideas in \cite{Feng-hong-zhu-compre} is to study an auxiliary system formed by the equations of the fluid momentum and the vorticity of the angular velocity. Through the auxiliary system an intrinsic coupling structure is captured and the desired estimates of $\tilde{\mathbf{u}} $ and $\tilde{\mathbf{w}} $ are established via elliptic theory and div-curl analysis. For the incompressible model, however, since the pressure function is determined by an elliptic equation involving both the fluid and angular velocity fields, rather than the density function alone in the compressible case, the crucial analysis in \cite{Feng-hong-zhu-compre} cannot be accessed for the problem considered in this paper. Collectively, the synergy of high dimensionality, absence of magnetic and angular viscosities, and strong coupling between the unknown functions presents a significant challenge to the mathematical analysis of \eqref{ns_euler0c}.

\subsection{Reformulation}

Due to the coupling of fluid velocity with magnetic field and lack of magnetic diffusion, the Eulerian formulation \eqref{ns_euler0c} is not convenient for implementing any effective energy scheme. Instead, we reformulate the model in Lagrangian coordinates by using the flow map:
\begin{equation}\label{flowmap}
\begin{aligned}
  \partial _{t} \tilde{\boldsymbol{\eta}}(\mathbf{x},t)&=\tilde{\mathbf{u}}\big(\tilde{\boldsymbol{\eta}}(\mathbf{x},t),t\big),\\
  \tilde{\boldsymbol{\eta}}(\mathbf{x},0)&=\mathbf{x}.
 \end{aligned}
 \end{equation}
Let $ \tilde{\boldsymbol{\eta}}(\mathbf{x},t) = \mathbf{x}+ {\boldsymbol{\eta}}(\mathbf{x},t) $, and let $ \mathcal{A}=[(\mathcal{I}_3+\nabla {\boldsymbol{\eta}})^{-1}]^{\mathrm{T}} $ with $ \mathcal{I}_3 $ being the $ 3\times 3 $ identity matrix. Define the Lagrangian unknowns:
 \begin{gather*}
(\mathbf{u}, \mathbf{w}, p, \mathbf{b})(\mathbf{x},t)=\Big(\tilde{\mathbf{u}},\tilde{\mathbf{w}},\tilde{p}+\frac{1}{2}|\tilde{\mathbf{b}}|^2,\tilde{\mathbf{b}}\Big)\big(\mathbf{x}+{\boldsymbol{\eta}}(\mathbf{x},t),t\big),\quad (\mathbf{x},t)\in \Omega \times \mathbb{R}^+.
\end{gather*}
Then the IBVP \eqref{ns_euler0c} can be reformulated in terms of $ ({\boldsymbol{\eta}}, \mathbf{u}, \mathbf{w}, p, \mathbf{b}) $ as
\begin{subequations}\label{lagrangianc}
\begin{alignat}{7}
\partial_t \boldsymbol{\eta} &= \mathbf{u}, \\
 \partial_t \mathbf{u} -(\mu+ \zeta)\Delta_{\mathcal{A}} \mathbf{u} +\nabla_{\mathcal{A}} p &= {\mathbf{b}} \cdot \nabla_{\mathcal{A}} {\mathbf{b}}  + \zeta \nabla _{\mathcal{A}} \times \mathbf{w}, \\
\partial _{t}\mathbf{w}+2 \zeta \mathbf{w}&= \zeta \nabla _{\mathcal{A}} \times \mathbf{u}, \\
 \partial_t \mathbf{b}&=\mathbf{b} \cdot\nabla_{\mathcal{A}} \mathbf{u}, \\
  \operatorname{div} _{\mathcal{A}} \mathbf{u}&=\operatorname{div} _{\mathcal{A}} \mathbf{b}=0, \label{1.4e}\\
(\boldsymbol{\eta}, \mathbf{u}, \mathbf{w}, \mathbf{b})\vert_{t=0}&=( \tilde{\boldsymbol{\eta}}_{0}-\mathbf{x}, \tilde{\mathbf{u}}_0, \tilde{\mathbf{w}}_0, \tilde{\mathbf{b}}_0) =: (\boldsymbol{\eta}_{0}, \mathbf{u} _{0}, \mathbf{w} _{0}, \mathbf{b} _{0}), \\
 \mathbf{u}|_{\partial\Omega}&=\mathbf{0},
\end{alignat}
\end{subequations}
where the actions $\nabla _{\mathcal{A}}$, $\operatorname{div}_{\mathcal{A}}$, $\nabla_{\mathcal{A}}\times$ and $\Delta_{\mathcal{A}}$ are given by
\begin{equation*}
\begin{aligned}
(\nabla _{\mathcal{A}} f)_i &:= \mathcal{A}_{ij}\partial_j f, \\
(\nabla_{\mathcal{A}}\times \mathbf{X})_i &:=\epsilon _{ijk} \mathcal{A}_{jl}\partial _{l}X _{k},
\end{aligned}
\qquad \begin{aligned}
\operatorname{div}_{\mathcal{A}} \mathbf{X} &:= \mathcal{A}_{ij}\partial_j X_i, \\
\Delta _{\mathcal{A}} f &:=\operatorname{div}_{\mathcal{A}}  \nabla _{\mathcal{A}} f,
\end{aligned}
\end{equation*}
for any appropriate $f$ and $\mathbf{X}$, and $\epsilon _{ijk}$ denotes the Levi-Cevita symbol, i.e.,
 \begin{equation*}
\epsilon_{ijk}=\left\{
\begin{aligned}
+&1, &\ &ijk=123,231,312,\\
-&1, &\ &ijk=321,213,132,\\
&0, &\ &i=j, \text{ or } j=k, \text{ or } k=i.
\end{aligned}
\right.
\end{equation*}

\subsection{Reduction}
The following reduction can be found in \cite{Tan-Wang-SIMA} for the non-resistive MHD system. For the reader's convenience, we present the key steps in the procedure. First, denote $ J:= \operatorname{det}(\mathcal{I}_3+\nabla \boldsymbol{\eta}) $. It is straightforward to check that
$\partial_t J =J \operatorname{div} _{\mathcal{A}} \mathbf{u}$, which together with \eqref{1.4e} yields $\partial _{t}J=0$. This, alongside $\boldsymbol{\eta}(\mathbf{x},0)=\mathbf{0}$, implies
\begin{align}\label{v1}
\operatorname{det}\big[\mathcal{I}_3+\nabla \boldsymbol{\eta}(\mathbf{x},t)\big] = \operatorname{det} \big[\mathcal{I}_3+\nabla\boldsymbol{\eta}(\mathbf{x},0)\big] =1.
\end{align}
Using \eqref{v1}, it can be shown that
 \begin{align}\label{imp2}
\mathbf{b} = (\mathcal{I}_3+\nabla\boldsymbol{\eta})\mathbf{b}_0.
 \end{align}
Second, for the sake of simplicity, we assume that the initial magnetic field $\mathbf{b}_0=  \mathbf{e}_3:=(0,0,1)$. It follows from \eqref{imp2} that $\mathbf{b} = (\mathbf{e}_3+ \partial_3\boldsymbol{\eta})$. This implies $\mathbf{b} \cdot\nabla_{\mathcal{A}} \mathbf{b} = b_j\a_{jk}\partial_k \mathbf{b}
 =\p_{3}^2\boldsymbol\eta$, with which we rewrite the model \eqref{lagrangianc} as
\begin{subequations}\label{reformulationc}
\begin{alignat}{6}
\partial_t\boldsymbol{\eta} &=\mathbf{u}, \\
 \partial_t \mathbf{u} -(\mu+\zeta)\Delta_{\mathcal{A}} {\mathbf{u}}+\nabla_{\mathcal{A}}p
& = \partial_3^2\boldsymbol{\eta}+\zeta\nabla_{\mathcal{A}}\times \mathbf{w}, \label{1.7b}\\
\partial _{t}\mathbf{w}+2 \zeta \mathbf{w}&= \zeta \nabla _{\mathcal{A}} \times \mathbf{u}, \label{1.7c}\\
\operatorname{div} _{\mathcal{A}} \mathbf{u}&=0, \\
(\boldsymbol{\eta}, \mathbf{u}, \mathbf{w}) \vert_{t=0}&=( \boldsymbol{\eta}_0, \mathbf{u}_0,\mathbf{w}_0), \\
\mathbf{u}|_{\partial\Omega}&= \mathbf{0}.
\end{alignat}
\end{subequations}
In the sequel, we focus on the reduced model \eqref{reformulationc} around the
steady state $(\mathbf{0},\mathbf{0},\mathbf{0})$. With the solution
$(\boldsymbol{\eta},\mathbf{u},\mathbf{w})$ of \eqref{reformulationc} at our
disposal, we can easily recover the magnetic field $\mathbf{b} = (\mathbf{e}_3+
\partial_3 \boldsymbol{\eta})$. Then
$(\boldsymbol{\eta},\mathbf{u},\mathbf{w},\mathbf{b})$ solves
\eqref{lagrangianc} by imposing the initial condition: $\mathbf{b}_0=
(\mathbf{e}_3+ \partial_3 \boldsymbol{\eta}_0) = \mathbf{e}_3$. We remark that our main results remain valid even when  $\mathbf{b}_0$ is a small perturbation from $\mathbf{e}_3$.

\vskip .1in
Before moving to the next section, we give a heuristic explanation of the difficulty of the problem. Formally, solving \eqref{1.7c} and assuming $\mathbf{w}_0=\mathbf{0}$ give us
\begin{align*}
\mathbf{w} = \zeta \mathrm{e}^{-2\zeta t} \int_0^t \mathrm{e}^{2\zeta \tau} (\nabla_{\mathcal{A}} \times \mathbf{u}) \mathrm{d}\tau.
\end{align*}
Substituting this into \eqref{1.7b} and expressing $\boldsymbol\eta$ in terms of $\mathbf{u}$, we obtain
\begin{align*}
 \partial_t \mathbf{u} -(\mu+\zeta)\Delta_{\mathcal{A}} {\mathbf{u}}+\nabla_{\mathcal{A}}p
 = \int_0^t\partial_3^2\mathbf{u}\mathrm{d}\tau + \zeta^2 \mathrm{e}^{-2\zeta t} \int_0^t \mathrm{e}^{2\zeta \tau} (\nabla_{\mathcal{A}} \times \nabla_{\mathcal{A}} \times \mathbf{u}) \mathrm{d}\tau.
\end{align*}
Heuristically, when $\mathbf{u}$ is slightly perturbed from $\mathbf{0}$, it holds $\nabla_{\mathcal{A}} \times \nabla_{\mathcal{A}} \times \mathbf{u} \approx -\Delta_{\mathcal{A}} \mathbf{u}$. Hence,
\begin{align*}
 \partial_t \mathbf{u} -(\mu+\zeta)\Delta_{\mathcal{A}} {\mathbf{u}}+\nabla_{\mathcal{A}}p
 \approx  \int_0^t\partial_3^2\mathbf{u}\mathrm{d}\tau - \zeta^2 \mathrm{e}^{-2\zeta t} \int_0^t \mathrm{e}^{2\zeta \tau} (\Delta_{\mathcal{A}} \mathbf{u}) \mathrm{d}\tau.
\end{align*}
We note that the combination of the first terms on both sides presents a (partially) wave structure, which is one of the mathematical challenges encountered in the study of the non-resistive MHD system (i.e., when $\zeta=0$). On top of that, when $\zeta>0$, the two Laplacians on both sides compete, rather than cooperate. Such an effect presents additional (significant) challenges to the mathematical analysis of the problem under consideration. We will give a brief explanation of the key ideas for overcoming the technical obstructions in the next section.

\vskip .1in
The organization of the rest of the paper is as follows. Section 2 contains the statement of our main results and a brief explanation of the key ideas utilized in the proof. In Section 3, we collect preparatory materials for the proof of the main results, including the Helmholtz projection and its properties, anisotropic estimates related to the Helmholtz projection of the curl of a vector field, classic estimates of Hodge type, and of the Stokes system, along with the local well-posedness of \eqref{ns_euler0c}. The next three sections are devoted to the {\it a priori} estimates of the non-spatial, horizontal, and vertical derivatives of the solution, respectively. The {\it a priori} estimates are then synthesized in Section 7 to close the entire energy scheme.

\section{Results and Ideas}\label{main-result}

This section contains the statement of our main results on the global dynamics of classical solutions to the reformulated problem \eqref{reformulationc} and a brief explanation of the key ideas for attacking the problem. We first introduce some notations that are frequently accessed throughout the paper.

\vskip .1in
For $k \geq 0$, $H^k$ denotes the usual Sobolev space $W^{k,2}(\Omega)$, with norm $\|\cdot\|_k:=\|\cdot\|_{W^{k,2}(\Omega)}$. In particular, $H^0(\Omega) = L^2(\Omega)$. With $\|\cdot\|_k$, define the \emph{anisotropic} Sobolev norm:
\begin{align}\label{ani-Sobolev}
\left\|f\right\|^2_{k,l}:=\sum_{ \alpha_1+\alpha_2 \le l}\|\partial_1^{\alpha_1}\partial_2^{\alpha_2} f\|^2_k,\quad k\ge0,\quad l\ge0,
\end{align}
where $\alpha_1$, $\alpha_2$ are nonnegative integers. By definition, for any two pairs $(k_1,l_1)$,  $(k_2,l_2)$ with $k_1+l_1=k_2+l_2$ and $k_2>k_1$, the norm $\|f\|_{k_2,l_2}$ contains more derivatives that $\|f\|_{k_1,l_1}$ does.

\vskip .1in
For a solution to \eqref{reformulationc} and an integer $ n \geq 2 $, define
\begin{align}
\mathcal{E}_{n}(t) := &\sum_{j=0}^{n} \|\partial ^{j}_t \mathbf{u}\| _{2n-2j}^{2} + \sum^{n-1}_{j=0}\|\nabla\partial^j_{t}p\|^2_{2n-2j-2} +\sum_{j=0}^{n}\|\partial^j_{t}\mathbf{w}\|^2_{2n-2j}+\|\mathcal{P}\partial^2_3\boldsymbol{\eta}\|^2_{2n-1} \notag\\
&\qquad +\|\boldsymbol{\eta}\|^2_{1,2n}+\|\boldsymbol{\eta}\|^2_{2n},\label{p_energy_defc}\\
 \mathcal{D}_{n}(t) := &\sum_{j=0}^{n}  \|\partial _{t}^{j} \mathbf{u}\|^2_{2n-2j+1} + \sum^{n-1}_{j=1}\|\nabla\partial^j_{t}p\|^2_{2n-2j-1} +\sum^n_{j=0}\sum^{1}_{l=0}\|\partial_t^{j+l} \mathbf{w}\|^2_{2n-2j} + \|\mathcal{P} \partial^2_3\boldsymbol{\eta}\|^2_{2n-1} \notag\\
&\qquad  + \|\partial _{3}\boldsymbol{\eta}\|^2_{0,2n} +
\|\boldsymbol{\eta}\|^2_{2n} +\|\nabla p\|^2_{2n-2}, \label{p_dissipation_defc}
\end{align}
where $ \mathcal{P} $ denotes the Helmholtz projection operator defined by \eqref{defi-projection}. It is worth mentioning that $\|\boldsymbol{\eta}\|^2_{1,2n}$ contains more derivatives of $\boldsymbol{\eta}$ than $\|\partial _{3}\boldsymbol{\eta}\|^2_{0,2n}$ does. Hence, the total dissipation (i.e., $\mathcal{D}_{n}$) does not cover the total energy (i.e., $\mathcal{E}_{n}$) of the solution. This is one of the distinctive features of the problem under consideration, which is different from many fully dissipative problems.

\vskip .1in
Moreover, for an integer $N\ge1$, we define
\begin{align}\label{FJ}
\mathcal{F}_{2N}(t) :=\|\boldsymbol{\eta}\|^2_{4N+1} \qquad \text{and} \qquad \mathcal{J}_{2N}(t):=\|\mathbf{u}\|^2_{4N+1}+\|\nabla p\|^2_{4N-1}+\|\mathbf{w}\|^2_{4N}.
\end{align}
With \eqref{p_energy_defc}, \eqref{p_dissipation_defc} and \eqref{FJ}, we define for any $\theta>0$,
\begin{align}\label{G_defca}
\mathscr{G}_{2N}(t) :=\sup_{0 \le \tau \le t} \big(\mathcal{E}_{2N} + \mathcal{F}_{2N} + (1+\tau)^{2N-4} \mathcal{E}_{N+2} \big)(\tau) + \int_0^t \Big(\mathcal{D}_{2N}  + \frac{\mathcal{J}_{2N}}{(1+\tau)^{1+\theta}}\Big)(\tau)\mathrm{d}\tau.
\end{align}
Finally, we denote by $C$ a generic positive constant which is independent of the initial data and time. For two positive quantities, $A\lesssim B$ means that $A\leq C B$ for some generic constant $C>0$.

\vskip .1in
With the above notations at our disposal, we now state the main results of this paper.

\begin{theorem}\label{thm-main-result}
Consider the initial-boundary value problem \eqref{reformulationc}. Suppose that the initial data are compatible with the boundary conditions, and that for some integer $N\ge4$ there exists a sufficiently small constant $\varepsilon_0>0$ such that $ \mathcal{E} _{2N} (0)+\mathcal{F}_{2N}(0) \le \varepsilon_0$. Then there exists a unique global solution $(\boldsymbol{\eta},\mathbf{u},p,\mathbf{w})$ to \eqref{reformulationc} satisfying
\begin{equation}\label{G-2N-esti-thm}
\mathscr{G}_{2N}(t) \lesssim \mathcal{E}_{2N} (0)+\mathcal{F}_{2N}(0),\quad \forall\ t>0.
\end{equation}
\end{theorem}

\begin{remark}
From \eqref{G_defca} and \eqref{G-2N-esti-thm}, we immediately have the decay rate:
\begin{align}\label{decay}
\sum _{l=0}^{1}\big( \|\partial _{t}^{l} \mathbf{u}\|_{2(N+2)-2l}^{2}+
\|\partial _{t}^{l} \mathbf{w}\|_{2(N+2)-2l}^{2} \big) +\|\boldsymbol{\eta}\|_{2(N+2)}^{2} \lesssim  \mathcal{E}_{2N} (0)(1+t)^{4-2N}
\end{align}
for $ N \geq 4 $. When $N$ is sufficiently large, the decay rate is almost exponential. Moreover, from \eqref{imp2} and \eqref{decay} we see that the magnetic field decays to $\mathbf{e}_3$.
\end{remark}

\begin{remark}
Theorem \ref{thm-main-result} established the stability of a non-trivial
background magnetic field in an incompressible magneto-micropolar fluid
with zero resistivity and zero angular viscosity. In obtaining Theorem
\ref{thm-main-result}, the background magnetic field is assumed to be
perpendicular to the horizontal strip domain. We remark that our main results
may still hold true even when the background magnetic field does not meet the
horizontal strip domain at the right angle. In addition, it is more physically
relevant to study the problem on bounded domains with physical boundaries,
which yet poses more technical obstructions to the mathematical analysis. The investigations are left for the future.
\end{remark}

\vspace{-4mm}
\begin{remark}\label{rem:3}
To the best of our knowledge, our paper appears to be the first to use the Helmholtz decomposition to deal with the micropolar flows. Even in the absence of a magnetic field, that is, for the micropolar equations, there are few results in the literature on the well-posedness of smooth solutions in the 3D case. There are only a few results in the 2D case. For example, see \cite{Dong-Zhang-JDE} for the Cauchy problem and \cite{Liu-Wang} for the initial-boundary value problem. It should be emphasized that for 2D case, both the angular velocity $ \mathbf{w} $ and the $ \mathop{\rm curl}\nolimits $ of $ \mathbf{u}$ are scalar functions, and this special feature plays a key role in establishing the a priori estimates. However, it seems that the methods in these works are not applicable for the 3D case. It is also worth noting that, due to the presence of $ \nabla \times \mathbf{w} $, our approach to estimating the vertical derivatives of the solutions differs significantly from that used in studies of the viscous non-resistive MHD equations. In those cases, the estimates can be directly obtained by using the regularity theory of the Stokes operator. But in our case, this is not sufficient. Our idea is to resort to the Helmholtz projection to deal with the linear coupling structure of the equations. A key component of our analysis is the development of a method to recover the estimate of $ \nabla \times \mathbf{w}$ from its Helmholtz projection via the intrinsic structure of the equations and establishing a crucial estimate on the projected \emph{curl} (see Lemma \ref{key-lemma}).
\end{remark}

\vspace{-4mm}
\begin{remark}
The stabilizing effect of the magnetic field is essential in our stability analysis and in establishing the large-time behavior of the solution. This paper adopts the Lagrangian coordinate framework, in which the nonlinear convection terms such as \( \boldsymbol{u} \cdot \nabla \mathbf{u} \) are absorbed into the time derivative. This naturally introduces higher-order derivatives of the flow map \( \boldsymbol{\eta} \) (e.g., \( \nabla^2 \boldsymbol{\eta} \)), which arise from terms like \( \Delta_{\mathcal{A}} \mathbf{u} \). However, since the flow map equation is inviscid, it is not possible to directly control space-time norms of \( \boldsymbol{\eta} \) using standard energy estimates. To overcome this difficulty, our analysis exploits the smoothing and stabilizing effects resulting from the coupling between the equation for \( \boldsymbol{\eta} \) and the momentum equation. In particular, the term \( \partial_3^2 \boldsymbol{\eta} \) in the momentum equation contributes directly to the dissipation estimates of \( \boldsymbol{\eta} \). Without the presence of the magnetic field, the key higher-order derivatives of \( \boldsymbol{\eta} \) would not be controllable, and the argument cannot be closed in the absence of the dissipation term \( \partial_3^2 \boldsymbol{\eta} \). Consequently, obtaining the desired large-time behavior would also be impossible.
\end{remark}

The primary technical device utilized in the proof of Theorem \ref{thm-main-result} is the two-tier energy method developed in \cite{GT-per}. Some of the ideas for studying the compressible version of \eqref{ns_euler0c} will also be adapted to the problem considered here. However, it was mentioned before that the major difference between the compressible and incompressible models lies in the fact that the incompressible pressure function is determined by an elliptic equation involving both the fluid and angular velocity fields, rather than the density function in the compressible case. To overcome the obstacle, we utilize the Helmholtz projection operator $ \mathcal{P}  $ to eliminate the pressure term from the fluid velocity equation, and then establish energy estimates for $ \mathcal{P} \partial _{3}^{2} \boldsymbol{\eta} $ and $ \mathcal{P}  (\nabla \times \mathbf{w}) $ (see Lemmas \ref{lem-eta-w} and \ref{lem-time-deri-curlw}). Nevertheless, this introduces two additional difficulties. First, from \eqref{1.7b} we see that the dissipation estimate of $\mathbf{u}$ to the highest order (i.e., $ \|\mathbf{u}\|^2_{2n+1} $) depends on the estimate of $\|\mathbf{w}\|_{2n}^2$, whose control in turn relies on $ \|\mathbf{u}\|^2_{2n+1} $, according to \eqref{1.7c}. Such an intertwining issue cannot be untangled solely via the Stokes equation. Second, direct energy estimates only give us control of $ \mathcal{P} (\nabla \times \mathbf{w})$, which is not enough to recover the estimates of $\mathbf{w}$ via the div-curl analysis, and it seems impossible to establish the desired estimates for $\mathbf{w}$ by utilizing the elliptic structure of the system.

\vskip .1in
To resolve the first issue, we derive the elliptic estimates for $\mathbf{u}$ through the equation of the vorticity of the angular velocity rather than using the Stokes equation for the fluid velocity field (see \eqref{u-tuoyuan}--\eqref{u-elli-1*} for details). For the second difficulty, we establish certain anisotropic estimates related to the Helmholtz projection operator $\mathcal{P}$ by its explicit representation formula in the strip domain. In particular, for the decomposition $ \nabla \times \mathbf{w} =\mathcal{P} (\nabla \times \mathbf{w})+ \nabla \phi$, we can show that
\begin{align}\label{KE}
 \|\nabla_h\phi\|^2_{k,\lambda}\lesssim\|\partial_3\phi\|^2_{k,\lambda}+\sigma\|\nabla\times\mathbf{w}\|^2_{0,k+\lambda}+  \|\mathbf{w}\|_{0}\|\mathbf{w}\|_{1}\quad \mbox{for any }\sigma>0,
\end{align}
where $\nabla_h$ denotes the horizontal gradient, $ 0 \leq k \leq 2n-1 $, and $ 0 \leq \lambda \leq 2n-k-1$. For details, please see Lemma \ref{key-lemma}. Since the estimates of $\partial _{3}\phi$ can be extracted from the third component of the equation of the vorticity of the angular velocity, \eqref{KE} enables us to gain the control of $\mathbf{w}$ through the ones of $ \mathcal{P} (\nabla \times \mathbf{w})$. This paves us the way to rennovate the arguments in the previous work \cite{Feng-hong-zhu-compre} to study the incompressible magneto-micropolar fluids.

\section{Preliminaries}\label{preliminaries}

This section contains prerequistes for the proof of Theorem \ref{thm-main-result}, including certain anisotropic estimates associated with the Helmholtz projection operator and the local well-posedness of the initial-boundary value problem \eqref{reformulationc}.

\subsection{Helmholtz Projection}

For $ \Omega= \mathbb{R}^{2}\times (0,1) $, the Helmholtz projection is defined by
\begin{align}\label{defi-projection}
\mathcal{P} :\, L ^{2}(\Omega) \rightarrow L _{\mathrm{div}}^{2}(\Omega),
\end{align}
where
\begin{gather*}
L _{\mathrm{div}}^{2}(\Omega)= \left\{ \mathbf{v} \in L ^{2}(\Omega)\vert\, \mathop{\mathrm{div}}\nolimits \mathbf{v}= 0\ \text{in }\Omega,\ \ \mathbf{v} \cdot \mathbf{n}=0 \ \text{on }\partial \Omega\right\}.
\end{gather*}
It is known that for any $ \mathbf{v}=(v _{1},v _{2},v _{3}) \in L ^{2}(\Omega) $, there exists a unique Helmholtz decomposition
\begin{align*}
\mathbf{v}= \mathcal{P} \mathbf{v}+ \nabla \varphi,
\end{align*}
where the potential function $ \varphi $ solves the following Neumann problem of the Poisson equation:
\begin{equation}\label{vfi-eq}
\begin{aligned}
  \Delta \varphi &= \mathop{\mathrm{div}}\nolimits \mathbf{v}, &\mathbf{x}& \in \Omega,\\
  \partial _{3}\varphi(x _{1},x _{2},d)&= v  _{3}(x _{1}, x_{2},d), \quad &d&=0,1.
\end{aligned}
\end{equation}
Moreover, it holds that
\begin{align*}
 \|\mathbf{v}\|_{0}^{2}= \|\mathcal{P} \mathbf{v}\|_{0}^{2}+\|\nabla \varphi\|_{0}^{2}.
 \end{align*}
 Denote $ \bar{\partial}:=\partial _{1},\, \partial _{2}\, \mbox{or\ }\partial _{t} $. It is clear that $ \bar{\partial}\mathcal{P} =\mathcal{P} \bar{\partial} $. Therefore for any $ k \geq 1 $, we have
 \begin{align}\label{bar-decom-high}
 \|\bar{\partial}^{k}\mathcal{P} \mathbf{v}\| _{0}\leq \|\bar{\partial}^{k}\mathbf{v}\|_{0}.
 \end{align}
However, $ \partial _{3}\mathcal{P} \mathbf{v} \neq \mathcal{P} \partial _{3} \mathbf{v} $ in general. Next, we prove an anisotropic estimate related to $ \mathcal{P}  $.

\begin{lemma}\label{lem-anisotropi-1}
For $ k \geq 2 $ and $ \lambda \geq 0 $, we have
\begin{align*}
\|\mathcal{P} \mathbf{v}\| _{k, \lambda} \lesssim \|\mathbf{v}\| _{k, \lambda},
\end{align*}
where the norm $\|\cdot\|_{k,\lambda}$ is defined by \eqref{ani-Sobolev}.
\end{lemma}

\begin{proof}
  Let $ \mathbf{v}= \mathcal{P} \mathbf{v}+ \nabla \varphi $, with $ \varphi $ satisfying the elliptic system \eqref{vfi-eq}. Then
\begin{gather}\label{deco-high-form}
 \partial _{3}\mathcal{P} \mathbf{v}= \partial _{3} \mathbf{v}- \nabla \omega,
\end{gather}
where $ \omega:= \partial _{3}\varphi $ satisfies
  \begin{equation}\label{omega-eq}
\begin{aligned}
  \Delta \omega &= \partial _{3}\mathop{\mathrm{div}}\nolimits \mathbf{v}, &\mathbf{x}& \in \Omega,\\
  \omega(x _{1},x _{2},d) &= v  _{3}(x _{1}, x_{2},d), \quad &d&=0,1.
\end{aligned}
\end{equation}
Thanks to the classical regularity theory for the elliptic system (cf.\,\cite{AW}), it holds that
\begin{align}\label{x1}
 \|\omega\|_{k,\lambda} \lesssim \|\partial _{3}\mathop{\mathrm{div}}\nolimits \mathbf{v}\|_{k-2,\lambda}+ \|v_{3}(x _{1}, x_{2},d)\|_{H_{\mathbf{x}_{h}}^{k+\lambda-1/2}} \lesssim \|\mathbf{v}\| _{k, \lambda},
\end{align}
for $ k \geq 2 $ and $ \lambda \geq 0 $, where $H_{\mathbf{x}_h}^s$ denotes the Hilbert space with respect to the horizontal variables.
Combining \eqref{bar-decom-high}, \eqref{deco-high-form} and \eqref{x1} implies for $ k \geq 2 $ and $ \lambda \geq 0 $ that
\begin{align}
\|\mathcal{P} \mathbf{v}\|_{k,\lambda} \lesssim \|\mathbf{v}\|_{k, \lambda}+ \|\omega\| _{k,\lambda} \lesssim \|\mathbf{v}\| _{k,\lambda}.
\end{align}
This completes the proof of the lemma.
\end{proof}

\subsection{Anisotropic Estimates of Projected Curl}

We now show the anisotropic estimates related to the Helmholtz projection of $ \mathbf{W}=\nabla \times \mathbf{w} $ for a suitably smooth vector field $\mathbf{w}= (w_1,w_2,w_3)$. These estimates are of great importance in recovering the estimates of $ \mathbf{W} $ from its Helmholtz projection. This is also one of the major differences between the incompressible model considered in this paper and the compressible one studied in \cite{Feng-hong-zhu-compre}. Precisely, we let
\begin{align}\label{x4}
 \mathbf{W}=\mathcal{P} \mathbf{W}+\nabla\phi,
\end{align}
where both $ \mathcal{P}\mathbf{W} $ and $ \nabla \phi $ belong to $ L ^{2}(\Omega) $, and $\phi$ solves
\begin{equation}\label{fip-elliptic}
\begin{aligned}
\Delta \phi &=0, &\mathbf{x}& \in \Omega,\\
  \partial _{3}\phi(x _{1},x _{2},d) &= (\partial_1w_2-\partial_2w_1)(x _{1}, x_{2},d)=:f(x _{1},x _{2},d), \quad &d&=0,1.
\end{aligned}
\end{equation}
We remark that here and in Lemma \ref{key-lemma}, we do not explicitly write the time-dependence of the functions, in order to simplify the presentation, since all the calculations and estimates are made with respect to the spatial variables. Then we have the following:

\begin{lemma}\label{key-lemma}
Let $ \nabla _{h}$ denote $\partial _{1}$ or $\partial _{2}$. Then it holds for $n \geq 3$ that
\begin{align}\label{phi1}
\|\nabla _{h}\phi\|_{0,2n-1}^{2} \lesssim \|\partial _{3}\phi\|_{0,2n-1}^{2}+  \sigma\|\mathbf{W}\|_{0,2n-1}^{2}+    \|\mathbf{w}\|_{0}\|\mathbf{w}\|_{1}
\end{align}
for any $ \sigma>0 $. More general, for any $ \sigma>0 $, it holds that
\begin{align}\label{phi2}
\|\nabla_h\phi\|^2_{k,\lambda}\lesssim\|\partial_3\phi\|^2_{k,\lambda}+\sigma\|\mathbf{W}\|^2_{0,k+\lambda}+  \|\mathbf{w}\|_{0}\|\mathbf{w}\|_{1},
\end{align}
where $ 0 \leq k \leq 2n-1 $ and $ 0 \leq \lambda \leq 2n-k-1 $.
\end{lemma}

\begin{proof}
{\bf Step 1.} From \eqref{fip-elliptic}, we deduce that
\begin{equation*}
\begin{aligned}
   \partial _{3}^{2}\hat{\phi}( {\boldsymbol \xi},x _{3}) &= \vert {\boldsymbol \xi}\vert ^{2}\hat{\phi}({\boldsymbol \xi},x _{3}), \quad &x&_3\in(0,1),\\
  \partial _{3}\hat{\phi}({\boldsymbol \xi},d)&=\hat{f}({\boldsymbol \xi},d), \quad &d&=0,1,
\end{aligned}
\end{equation*}
where $ \hat{\phi}({\boldsymbol \xi},x _{3}) $ denotes the Fourier transform of $ \phi $ with respect to $ \mathbf{x} _{h} := (x _{1}, x _{2}) $ and $ {\boldsymbol \xi}=( \xi_{1}, \xi_{2}) $. Then from  standard ODE theory, we have
\begin{align*}
\hat{\phi}({\boldsymbol \xi},x _{3})=\frac{1}{|{\boldsymbol \xi}|}\bigg(\frac{{\mathop{\mathrm{e}}}^{|{\boldsymbol \xi}|}\hat{f}({\boldsymbol \xi},0)-\hat{f}({\boldsymbol \xi},1)}{ 1-{\mathop{\mathrm{e}}}^{2|{\boldsymbol \xi}|}}{\mathop{\mathrm{e}}}^{\vert {\boldsymbol \xi}\vert(1- x _{3})}+\frac{\hat{f}({\boldsymbol \xi},1){\mathop{\mathrm{e}}}^{|{\boldsymbol \xi}|}-\hat{f}({\boldsymbol \xi},0)}{{\mathop{\mathrm{e}}}^{2|{\boldsymbol \xi}|}-1}{\mathop{\mathrm{e}}}^{\vert {\boldsymbol \xi}\vert x _{3}}\bigg),
\end{align*}
where $ \vert {\boldsymbol \xi}\vert= \sqrt{\xi_{1}^{2}+\xi_{2}^{2}} $. Now we estimate $|{\boldsymbol \xi}|\hat{\phi}$ by using frequency localization.

\vskip .1in
{\bf Step 2.} When $|{\boldsymbol \xi}|\leq 1$, a direct calculation shows that
\begin{align}\label{phi3}
\vert {\boldsymbol \xi}\vert\hat{\phi}({\boldsymbol \xi},x _{3})=\frac{\hat{f}({\boldsymbol \xi},0)}{1-{\mathop{\mathrm{e}}}^{2|{\boldsymbol \xi}|}}\big({\mathop{\mathrm{e}}}^{|{\boldsymbol \xi}|(2-x _{3})}+{\mathop{\mathrm{e}}}^{ \vert {\boldsymbol \xi}\vert x _{3}}\big) - \frac{\hat{f}({\boldsymbol \xi},1)}{1-{\mathop{\mathrm{e}}}^{2|{\boldsymbol \xi}|}}\big({\mathop{\mathrm{e}}}^{\vert {\boldsymbol \xi}\vert(1-x _{3})}+{\mathop{\mathrm{e}}}^{|{\boldsymbol \xi}|(1+x _{3})}\big).
\end{align}
Then we have
\begin{align*}
&\big\| |{\boldsymbol \xi}|\hat{\phi}\big\|^2_{L^2(\{|{\boldsymbol \xi}|\leq 1\})}\\
\leq &\int_{\{|{\boldsymbol \xi}|\leq 1\}}\frac{|\hat{f}({\boldsymbol \xi},0)|^2}{({\mathop{\mathrm{e}}}^{2|{\boldsymbol \xi}|}-1)^2}\int^1_{0}{\mathop{\mathrm{e}}}^{2|{\boldsymbol \xi}|(2-x _{3})}\mathrm{d} x _{3} \mathrm{d}{\boldsymbol \xi}+\int_{\{\vert {\boldsymbol \xi}\vert\leq1\}}\frac{|\hat{f}({\boldsymbol \xi},0)|^2}{({\mathop{\mathrm{e}}}^{2|{\boldsymbol \xi}|}-1)^2}\int^1_{0}{\mathop{\mathrm{e}}}^{2|{\boldsymbol \xi}|x _{3}}\mathrm{d} x _{3} \mathrm{d}{\boldsymbol \xi}
\\
&\quad+\int_{\{\vert {\boldsymbol \xi}\vert\leq 1\}}\frac{|\hat{f}({\boldsymbol \xi},1)|^2}{({\mathop{\mathrm{e}}}^{2|{\boldsymbol \xi}|}-1)^2}\int^1_{0}{\mathop{\mathrm{e}}}^{2 \vert {\boldsymbol \xi}\vert(1-x _{3})}\mathrm{d}x _{3} \mathrm{d}{\boldsymbol \xi}+\int_{\{\vert {\boldsymbol \xi}\vert\leq1\}}\frac{|\hat{f}({\boldsymbol \xi},1)|^2}{({\mathop{\mathrm{e}}}^{2 \vert {\boldsymbol \xi}\vert}-1)^2}\int^1_{0}{\mathop{\mathrm{e}}}^{2 \vert {\boldsymbol \xi}\vert(1+x _{3})}\mathrm{d} x _{3} \mathrm{d}{\boldsymbol \xi}
\\
\lesssim&\int_{\{|{\boldsymbol \xi}|\leq 1\}}\frac{|\hat{f}({\boldsymbol \xi},1)|^2+|\hat{f}({\boldsymbol \xi},0)|^2}{|{\boldsymbol \xi}|^2}\mathrm{d}{\boldsymbol \xi}\leq\int_{\{|{\boldsymbol \xi}|\leq 1\}}\big(|\hat{\mathbf{w}}(\cdot,1)|^2+|\hat{\mathbf{w}}(\cdot,0)|^2\big) \mathrm{d}{\boldsymbol \xi},
\end{align*}
where we used the definition of $f$ (see \eqref{fip-elliptic}) and $\hat{\mathbf{w}} = (\hat{w}_1,\hat{w}_2,\hat{w}_3)$ denotes the Fourier transform of $\mathbf{w}$ with respect to the horizontal variables. This implies for $|{\boldsymbol \xi}|\leq 1$ that
\begin{equation}\label{esti-small-xi}
\big\| |{\boldsymbol \xi}|\hat{\phi} \big\|^2_{L^2(\{|{\boldsymbol \xi}|\leq 1\})}\lesssim \|\mathbf{w}\|_0\|\mathbf{w}\|_1,
\end{equation}
where we used the fact:
\begin{align}
\|\mathbf{w}(\mathbf{x} _{h},d)\| _{L _{\mathbf{x} _{h}}^{2}} ^{2} \lesssim  \|\mathbf{w}\|_{0} ^{2} + \|\mathbf{w}\|_{0}\| \partial _{3} \mathbf{w} \|_{0},\quad d=0,1,
\end{align}
where $ \|\cdot\| _{L _{\mathbf{x} _{h}}^{2}} $ denotes the $ L ^{2} $ norm with respect to the horizontal variables. Obviously, the estimate \eqref{esti-small-xi} is also true when $\vert {\boldsymbol \xi}\vert\leq K$ for any finite constant $ K>0 $. Next, we derive an estimate for $|{\boldsymbol \xi}|\hat{\phi}$ when $|{\boldsymbol \xi}|\gg 1$.

\vskip .1in
{\bf Step 3.} Notice that $|{\boldsymbol \xi}|\hat{\phi}$ can be rewritten as
\begin{align*}
\vert {\boldsymbol \xi}\vert\hat{\phi}({\boldsymbol \xi},x _{3}) &=\frac{{\mathop{\mathrm{e}}}^{|{\boldsymbol \xi}|}\hat{f}({\boldsymbol \xi},0)-\hat{f}({\boldsymbol \xi},1)}{ 1-{\mathop{\mathrm{e}}}^{2|{\boldsymbol \xi}|}}{\mathop{\mathrm{e}}}^{\vert {\boldsymbol \xi}\vert(1- x _{3})}+\frac{\hat{f}({\boldsymbol \xi},1){\mathop{\mathrm{e}}}^{|{\boldsymbol \xi}|}-\hat{f}({\boldsymbol \xi},0)}{{\mathop{\mathrm{e}}}^{2|{\boldsymbol \xi}|}-1}{\mathop{\mathrm{e}}}^{\vert {\boldsymbol \xi}\vert x _{3}}.
\end{align*}
Another direct calculation shows that
\begin{align*}
\partial_3\hat{\phi}({\boldsymbol \xi},x _{3}) &=-\frac{{\mathop{\mathrm{e}}}^{|{\boldsymbol \xi}|}\hat{f}({\boldsymbol \xi},0)-\hat{f}({\boldsymbol \xi},1)}{ 1-{\mathop{\mathrm{e}}}^{2|{\boldsymbol \xi}|}}{\mathop{\mathrm{e}}}^{\vert {\boldsymbol \xi}\vert(1- x _{3})}+\frac{\hat{f}({\boldsymbol \xi},1){\mathop{\mathrm{e}}}^{|{\boldsymbol \xi}|}-\hat{f}({\boldsymbol \xi},0)}{{\mathop{\mathrm{e}}}^{2|{\boldsymbol \xi}|}-1}{\mathop{\mathrm{e}}}^{\vert {\boldsymbol \xi}\vert x _{3}}.
\end{align*}
Then it holds that
\begin{align}\label{pa-h-fip-pa-3-fip-relat}
 \big\| \vert {\boldsymbol \xi}\vert \hat{\phi}\big\| _{L ^{2}(\{|{\boldsymbol \xi}|\gg 1\})}^{2}=\|\partial _{3}\hat{\phi}\|_{L ^{2}(\{|{\boldsymbol \xi}|\gg 1\})}^{2}-4 R _{0},
\end{align}
where
\begin{align}\label{R-formula}
R _{0} =\int^1_{0}\int_{\{|{\boldsymbol \xi}|\gg1\}}\frac{{\mathop{\mathrm{e}}}^{|{\boldsymbol \xi}|}\big[\hat{f}({\boldsymbol \xi},0){\mathop{\mathrm{e}}}^{\vert {\boldsymbol \xi}\vert}-\hat{f}({\boldsymbol \xi},1)\big] \big[\hat{f}({\boldsymbol \xi},1){\mathop{\mathrm{e}}}^{\vert {\boldsymbol \xi}\vert}-\hat{f}({\boldsymbol \xi},0)\big]}{(1-{\mathop{\mathrm{e}}}^{2|{\boldsymbol \xi}|})^2}\mathrm{d}{\boldsymbol \xi} \mathrm{d}x _{3}.
\end{align}
Notice that
\begin{align*}
&\ \big| \hat{f}({\boldsymbol \xi},0){\mathop{\mathrm{e}}}^{|{\boldsymbol \xi}|}-\hat{f}({\boldsymbol \xi},1)\big| = \big| \int^1_{0}\partial _{3}\big(\hat{f}({\boldsymbol \xi},x _{3}){\mathop{\mathrm{e}}}^{|{\boldsymbol \xi}|(1-x _{3})}\big)\mathrm{d} x _{3}\big| \\
 = &\ \big| \int^1_{0}\big(\partial_3\hat{f}({\boldsymbol \xi},x _{3}){\mathop{\mathrm{e}}}^{|{\boldsymbol \xi}|(1-x  _{3})}-|{\boldsymbol \xi}|\hat{f}({\boldsymbol \xi},x _{3}){\mathop{\mathrm{e}}}^{|{\boldsymbol \xi}|(1-x _{3})}\big)\mathrm{d} x _{3}\big| \\
=&\ \big| \int^1_{0}\mathrm{i}\,{\mathop{\mathrm{e}}}^{|{\boldsymbol \xi}|(1-x _{3}) }(\xi_1\partial_3\hat{w}_2- \xi_2\partial_3\hat{w}_1)\mathrm{d}x _{3}+\int^1_{0}|{\boldsymbol \xi}|\hat{f}({\boldsymbol \xi},x _{3}){\mathop{\mathrm{e}}}^{|{\boldsymbol \xi}|(1-x _{3})}\mathrm{d} x _{3}\big|
\\
\leq &\ \big| \int^1_{0}\mathrm{i}\,\xi_1{\mathop{\mathrm{e}}}^{|{\boldsymbol \xi}|(1-x _{3})}(\partial_3\hat{w}_2-\mathrm{i}\,\xi_2\hat{w}_3)\mathrm{d}x _{3}\big| + \big|\int^1_0\mathrm{i}\, \xi_2{\mathop{\mathrm{e}}}^{|{\boldsymbol \xi}|(1-x _{3})}(\partial_3\hat{w}_1-\mathrm{i}\,\xi_1\hat{w}_3)\mathrm{d}x _{3}\big|
 \nonumber \\
 & \quad+ \big| \int^1_{0}|{\boldsymbol \xi}|\hat{f}({\boldsymbol \xi},x _{3}){\mathop{\mathrm{e}}}^{|{\boldsymbol \xi}|(1-x _{3})}\mathrm{d} x _{3} \big|
 \\
\leq &\ \Big(\int^1_{0}\xi^2_{1}{\mathop{\mathrm{e}}}^{2|{\boldsymbol \xi}|(1-x _{3})}\mathrm{d} x _{3} \Big)^{\frac{1}{2}}\big( \|\hat{W}_{1}\|_{L _{x _{3}}^{2}}+\|\hat{W}_{2}\|_{L _{x _{3}}^{2}}\big)
 + \|\hat{f}({\boldsymbol \xi}, \cdot)\| _{L _{x _{3}}^{2}}\Big(\int^1_{0}|{\boldsymbol \xi}|^2{\mathop{\mathrm{e}}}^{2|{\boldsymbol \xi}|(1-x _{3})}\mathrm{d} x _{3}\Big)^{\frac{1}{2}}
 \\
\lesssim &\ |{\boldsymbol \xi}|^{\frac{1}{2}}\big({\mathop{\mathrm{e}}}^{2|{\boldsymbol \xi}|}-1\big)^{\frac{1}{2}} \big(\|\hat{W}_1({\boldsymbol \xi},\cdot)\|_{L^2_{x _{3}}}+\|\hat{W}_2({\boldsymbol \xi},\cdot)\|_{L^2_{x _{3}}} + \|\hat{W}_3({\boldsymbol \xi},\cdot)\|_{L^2_{x _{3}}} \big),
\end{align*}
where $\hat{\mathbf{W}}=(\hat{W}_1,\hat{W}_2,\hat{W}_3)$ denotes the Fourier transform of $\mathbf{W}=\nabla\times \mathbf{w}$ with respect to the horizontal variables. Similarly, we can show that
\begin{align*}
&\ \big| \hat{f}({\boldsymbol \xi},1){\mathop{\mathrm{e}}}^{|{\boldsymbol \xi}|}-\hat{f}({\boldsymbol \xi},0)\big|
= \big|\int^1_{0}\partial _{3}\big(\hat{f}({\boldsymbol \xi},x _{3}){\mathop{\mathrm{e}}}^{|{\boldsymbol \xi}|x _{3}}\big)\mathrm{d} x _{3}\big| \nonumber\\
=&\ \big| \int^1_{0}\big(\partial_3\hat{f}({\boldsymbol \xi},x _{3}){\mathop{\mathrm{e}}}^{|{\boldsymbol \xi}|x _{3}}+|{\boldsymbol \xi}|\hat{f}({\boldsymbol \xi},x _{3}){\mathop{\mathrm{e}}}^{|{\boldsymbol \xi}|x _{3}}\big)\mathrm{d} x _{3}\big| \nonumber\\
\leq &\ |{\boldsymbol \xi}|^{\frac{1}{2}}\big({\mathop{\mathrm{e}}}^{2|{\boldsymbol \xi}|}-1\big)^{\frac{1}{2}}\big(\|\hat{W}_1({\boldsymbol \xi},\cdot)\|_{L^2_{x _{3}}}+\|\hat{W}_2({\boldsymbol \xi},\cdot)\|_{L^2_{x _{3}}}+\|\hat{W}_3({\boldsymbol \xi},\cdot)\|_{L^2_{x _{3}}}\big).
\end{align*}
Using the above estimates, we derive from \eqref{R-formula} that for any $ \sigma>0 $, there exists a constant $ \bar K >0$ such that for $ \vert {\boldsymbol \xi}\vert \geq \bar K $,
\begin{align}
\vert R _{0}\vert & \lesssim  \int _{0}^{1}\int_{\{|{\boldsymbol \xi}| \geq \bar K \}}\big(\|\hat{W}_1({\boldsymbol \xi},\cdot)\|_{L^2_{x _{3}}}^{2}+\|\hat{W}_2({\boldsymbol \xi},\cdot)\|_{L^2_{x _{3}}}^{2}+\|\hat{W}_3({\boldsymbol \xi},\cdot)\|_{L^2_{x _{3}}}^{2}\big) \frac{\vert {\boldsymbol \xi}\vert}{\mathrm{e} ^{2 \vert {\boldsymbol \xi}\vert}-1}\mathrm{d}{\boldsymbol \xi} \mathrm{d}x _{3}
 \nonumber \\
 & \leq \frac{\sigma}{4}\big(\|\hat{W}_1\|^2_{0}+\|\hat{W}_2\|^2_0+\|\hat{W}_3\| _{0}^2 \big)= \frac{\sigma}{4}\|\mathbf{W}\|_{0}^{2}.
\end{align}
This, alongside \eqref{pa-h-fip-pa-3-fip-relat}, implies that
\begin{align}\label{esti-xi-large}
\big\| |{\boldsymbol \xi}|\hat{\phi}\big\|^2_{L^2(\{|{\boldsymbol \xi}|\ge \bar K\})}\leq \|\partial_3\hat{\phi}\|^2_{L^2(\{|{\boldsymbol \xi}|\ge \bar K \})}+ \sigma\|\mathbf{W}\|_{0}^{2}.
\end{align}
Combining \eqref{esti-small-xi} and \eqref{esti-xi-large}, we then get for any $ \sigma>0 $ that
\begin{align}\label{zero-fip-tan-esti}
 \big\| |{\boldsymbol \xi}|\hat{\phi}\big\|_{0}^2 \lesssim \|\partial_3\hat{\phi}\|^2_{0}+\sigma\|\mathbf{W}\|_{0}^{2}+\|\mathbf{w}\|_0\|\mathbf{w}\|_1.
\end{align}
Repeating the derivations of \eqref{esti-small-xi} and \eqref{esti-xi-large}, we can show for any $ 1 \leq k \leq 2n-1 $ that
\begin{align*}
 \big\|\vert {\boldsymbol \xi}\vert ^{k+1}\hat{\phi}({\boldsymbol \xi},x _{3})\big\|_{0}^2& \lesssim \big\| \vert {\boldsymbol \xi}\vert ^{k}\partial_3\hat{\phi}({\boldsymbol \xi},x _{3})\big\|^2_{0}+\sigma\big\| \vert {\boldsymbol \xi}\vert ^{k} \hat{\mathbf{W}}\big\|_{0}^{2}+   \|\mathbf{w}\|_0\|\mathbf{w}\|_1
 \nonumber \\
 & \lesssim \|\partial _{3}\phi\| _{0,k}^{2}+ \sigma \|\mathbf{W}\|_{0,k}^{2}+   \|\mathbf{w}\|_0\|\mathbf{w}\|_1 .
\end{align*}
That is,
\begin{gather}\label{tan-esti-fip}
\|\nabla _{h}\phi\| _{0,k} ^{2}\lesssim \|\partial _{3}\phi\| _{0,k}^{2}+ \sigma \|\mathbf{W}\|_{0,k}^{2}+\|\mathbf{w}\|_0\|\mathbf{w}\|_1 .
\end{gather}
This, along with \eqref{zero-fip-tan-esti}, gives \eqref{phi1}.

\vskip .1in
{\bf Step 4.} For any $ 0 \leq k \leq 2n-1 $ and $ 0 \leq\lambda \leq 2n-k-1 $, we have
\begin{align*}
 \vert {\boldsymbol \xi}\vert ^{\lambda+1}\partial _{3}^{k}\hat{\phi}=(-1)^{k}\vert {\boldsymbol \xi}\vert ^{\lambda+k}\frac{{\mathop{\mathrm{e}}}^{|{\boldsymbol \xi}|}\hat{f}({\boldsymbol \xi},0)-\hat{f}({\boldsymbol \xi},1)}{ 1-{\mathop{\mathrm{e}}}^{2|{\boldsymbol \xi}|}}{\mathop{\mathrm{e}}}^{\vert {\boldsymbol \xi}\vert(1- x _{3})}+\vert {\boldsymbol \xi}\vert ^{\lambda+k}\frac{\hat{f}({\boldsymbol \xi},1){\mathop{\mathrm{e}}}^{|{\boldsymbol \xi}|}-\hat{f}({\boldsymbol \xi},0)}{{\mathop{\mathrm{e}}}^{2|{\boldsymbol \xi}|}-1}{\mathop{\mathrm{e}}}^{\vert {\boldsymbol \xi}\vert x _{3}}
\end{align*}
and
\begin{align*}
\vert {\boldsymbol \xi}\vert ^{\lambda+k}\partial _{3}\hat{\phi}=  - \vert {\boldsymbol \xi}\vert ^{\lambda+k}\frac{{\mathop{\mathrm{e}}}^{|{\boldsymbol \xi}|}\hat{f}({\boldsymbol \xi},0)-\hat{f}({\boldsymbol \xi},1)}{ 1-{\mathop{\mathrm{e}}}^{2|{\boldsymbol \xi}|}}{\mathop{\mathrm{e}}}^{\vert {\boldsymbol \xi}\vert(1- x _{3})}+
\vert {\boldsymbol \xi}\vert ^{\lambda+k}\frac{\hat{f}({\boldsymbol \xi},1){\mathop{\mathrm{e}}}^{|{\boldsymbol \xi}|}-\hat{f}({\boldsymbol \xi},0)}{{\mathop{\mathrm{e}}}^{2|{\boldsymbol \xi}|}-1}{\mathop{\mathrm{e}}}^{\vert {\boldsymbol \xi}\vert x _{3}}.
\end{align*}
Then it follows that
\begin{align}
\big\|\vert {\boldsymbol \xi}\vert ^{\lambda+1}\partial _{3}^{k}\hat{\phi}({\boldsymbol \xi},x _{3})\big\| _{L^2(\{|{\boldsymbol \xi}|\gg 1\})}^{2} \leq \big\|\vert {\boldsymbol \xi}\vert ^{\lambda+k}\partial _{3}\hat{\phi}({\boldsymbol \xi},x _{3})\big\| _{L^2(\{|{\boldsymbol \xi}|\gg 1\})}^{2}+4 R _{1},
\end{align}
where
\begin{align}\label{R-formula-1}
R _{1} =\int^1_{0}\int_{\{|{\boldsymbol \xi}|\gg1\}}\vert {\boldsymbol \xi}\vert ^{2\lambda+2k}\frac{{\mathop{\mathrm{e}}}^{|{\boldsymbol \xi}|} \big |\hat{f}({\boldsymbol \xi},0){\mathop{\mathrm{e}}}^{\vert {\boldsymbol \xi}\vert}-\hat{f}({\boldsymbol \xi},1)\big| \cdot \big| \hat{f}({\boldsymbol \xi},1){\mathop{\mathrm{e}}}^{\vert {\boldsymbol \xi}\vert}-\hat{f}({\boldsymbol \xi},0)\big|}{(1-{\mathop{\mathrm{e}}}^{2|{\boldsymbol \xi}|})^2}\mathrm{d}{\boldsymbol \xi} \mathrm{d}x _{3}.
\end{align}
Again, by following the derivations of \eqref{esti-small-xi} and \eqref{esti-xi-large}, we get for any $ \sigma>0 $ that
\begin{align*}
\big \|\vert {\boldsymbol \xi}\vert ^{\lambda+1}\partial _{3}^{k}\hat{\phi}({\boldsymbol \xi},x _{3})\big\| _{0}^{2} \lesssim \big\|\vert {\boldsymbol \xi}\vert ^{\lambda+k}\partial _{3}\hat{\phi}({\boldsymbol \xi},x _{3})\big\| _{0}^{2}+  \sigma \|\mathbf{W}\|_{0,\lambda+k}^{2}+   \|\mathbf{w}\|_0\|\mathbf{w}\|_1.
\end{align*}
This implies that
 \begin{align}
 \|\nabla _{h}\phi\| _{k,\lambda}^2 \lesssim \|\partial _{3}\phi\| _{0,\lambda+k}^{2}+\sigma \|\mathbf{W}\|_{0,\lambda+k}^{2}+   \|\mathbf{w}\|_0\|\mathbf{w}\|_1.
\end{align}
This completes the proof of the lemma.
\end{proof}

Since the Helmholtz projection $ \mathcal{P} $ commutes with $ \partial _{t} $, we have
\begin{align}
\partial _{t}^{j}\mathbf{W}= \mathcal{P} \partial _{t}^{j}\mathbf{W} + \nabla \partial _{t}^{j}\phi
\end{align}
for any $ j \geq 1 $. By using the arguments in the proof of Lemma \ref{key-lemma}, we have the following:

\begin{lemma}\label{lemma-anisotropic-3}
Let $ \nabla _{h}$ denote $\partial _{1}$ or $\partial _{2} $ and let $ n > 1$. Then it holds for any $ \sigma>0 $ that
\begin{align}\label{phi2-pa-t}
\|\nabla_h \partial _{t}^{j}\phi\|^2_{k,\lambda}\lesssim\|\partial_3 \partial _{t}^{j}\phi\|^2_{k,\lambda}+\sigma\| \partial _{t}^{j} \mathbf{W}\|^2_{0,k+\lambda}+  \| \partial _{t}^{j} \mathbf{w}\|_{0}\|\partial _{t}^{j} \mathbf{w}\|_{1},
\end{align}
where $ 1 \leq j \leq n $, $ 0 \leq k \leq 2n-2j + 1 $ and $ 0 \leq \lambda \leq 2n-2j-k + 1 $.
\end{lemma}

\subsection{Hodge Estimates and the Stokes System}

In the proof of our main result we will frequently utilize the standard Hodge type estimates:

\begin{lemma}\label{div-curl-in-lem}
Let $ k \geq 1 $, $ \lambda \geq 0 $, $ \Omega:=\mathbb{R}^{2}\times (0,1) $, and let $\mathbf{u}$ be a smooth vector field in $\mathbb{R}^3$. Then
\begin{align}\label{div-curl-esti}
 \|\mathbf{u}\|_{k,\lambda}\lesssim \|\mathop{\mathrm{div}}\nolimits \mathbf{u}\|_{k-1,\lambda}+\|\nabla \times \mathbf{u}\|_{k-1,\lambda}+\|\mathbf{u}\|_{0,k+\lambda}.
\end{align}

\end{lemma}
\begin{proof}
  We shall only consider the case $ k=1 $ and $ \lambda=0 $ since the other cases can be proved similarly. We shall show that
  \begin{align}\label{case-1-0}
  \|\mathbf{u}\| _{1} \lesssim \|\mathop{\mathrm{div}}\nolimits \mathbf{u}\|_{0}+\|\nabla \times \mathbf{u}\|_{0}+\|\mathbf{u}\|_{0,1}.
  \end{align}
  It then suffices to control the normal derivative of $ \mathbf{u} $. Indeed, it holds that
\begin{align*}
\begin{cases}
  \|\partial _{3}u _{1}\|_{0} \lesssim \|\partial _{3}u _{1}- \partial _{1}u _{3}\|_{0}+\|\partial _{1}u _{3}\| _{0} \lesssim \|\nabla \times \mathbf{u}\|_{0}+\|\mathbf{u}\| _{0,1},\\
\|\partial _{3}u _{2}\|_{0} \lesssim \|\partial _{3}u _{2}- \partial _{2}u _{3}\|_{0}+\|\partial _{2}u _{3}\| _{0} \lesssim \|\nabla \times \mathbf{u}\|_{0}+\|\mathbf{u}\| _{0,1},
\\
\|\partial _{3}u _{3}\| _{0} \lesssim \|\mathop{\mathrm{div}}\nolimits \mathbf{u}\|_{0}+\|\partial _{1}u _{1}\| _{0}+\|\partial _{2}u _{2}\| _{0} \lesssim \|\mathop{\mathrm{div}}\nolimits \mathbf{u}\|_{0}+\|\mathbf{u}\|_{0,1}.
\end{cases}
\end{align*}
This immediately gives \eqref{case-1-0}.
\end{proof}

We will also apply the classical regularity theory for the Stokes system (cf. \cite{L,T}):
\begin{equation}\label{stokes eq}
\begin{aligned}
-\mu\Delta \mathbf{u} +\nabla p &=\mathbf{f},  \quad &\mathbf{x}& \in \Omega, \\
\diverge{\mathbf{u}} &=g,  \quad  &\mathbf{x}& \in \Omega, \\
\mathbf{u}|_{\partial\Omega} &=\mathbf{0}.
\end{aligned}
\end{equation}

\begin{lemma}\label{i_linear_elliptic2}
Let $r\ge 2$. If $\mathbf{f} \in H^{r-2}(\Omega)$, $g\in H^{r-1}(\Omega)$
and  $(\mathbf{u},p) $ solves \eqref{stokes eq}, then
\begin{equation}\label{stokes es}
\norm{\mathbf{u}}_{r}+\norm{\nabla p}_{r-2}\lesssim\norm{\mathbf{f}}_{r-2}+\norm{g}_{r-1}+\norm{\mathbf{u}}_0.
\end{equation}
\end{lemma}

\subsection{Local Well-posedness}

To wrap up this section, we remark that the third equation in \eqref{reformulationc} is an ODE, from which $ \mathbf{w} $ can be recovered in terms of $ \nabla \times \mathbf{u} $. Then the problem reduces to one in terms of $ (\boldsymbol{\eta}, \mathbf{u}) $. The local well-posedness for the reduced problem on $ (\boldsymbol{\eta}, \mathbf{u}) $ can be proved by similar arguments as in \cite{Tan-Wang-SIMA}.  The procedure is standard and we omit the details for brevity. The rest of this paper is devoted to building the uniform \emph{a priori} estimates of the local solution, based on which and standard continuity argument Theorem \ref{thm-main-result} will follow.

\section{Estimates on Non-spatial Derivatives} \label{sec:non_spatial_estimates}

This section contains the \emph{a priori} estimates of the non-spatial derivatives of the local solution. To facilitate the presentation, we first introduce some notations for convenience.

\begin{notation}\label{notation1}
For two positive quantities $A$ and $B$, we write $A \sim B$ if $C _{1}B \leq  A \leq C_2B $ for some generic constants $ C_1, C_2 >0 $. Let $\mathbb{Z}_+^{1+m} = \{ \alpha = (\alpha_0,\alpha_1,\dots,\alpha_m) : \alpha_i \in \mathbb{Z}_+, i=0,1,\dots,m\}$ be the collection of multi-indices for mixed temporal-spatial derivatives, where $\mathbb{Z}_+$ is the usual set of non-negative integers. For $\alpha \in \mathbb{Z}_+^{1+m}$, we write $\partial^\alpha = \partial _{t}^{\alpha_0} \partial_1^{\alpha_1}\dots \partial_m^{\alpha_m}$ and define the parabolic count of such an index as $\vert \alpha\vert = 2\alpha_0 + \alpha_1 + \dots + \alpha_m$.
\end{notation}

Throughout the rest of the paper, all the energy estimates are derived based on the following:

\vskip .1in
\noindent {\textit{\textbf{A priori}} {\bf assumption}: {\it For some $ T>0 $ within the lifespan of the local solution to \eqref{reformulationc}, assume
\begin{align}\label{a-priori}
 \mathscr{G}_{2N}(T) \le \delta
\end{align}
for some constant $\delta>0$, where $ \mathscr{G}_{2N}(T) $ is defined by \eqref{G_defca} and $ N \geq 4 $.}

\vskip .1in
We begin with the basic $L^2$-level estimate of the solution.

\begin{lemma}\label{Lem-L2}
Let $ (\boldsymbol{\eta}, \mathbf{u}, p, \mathbf{w}) $ be the local solution to \eqref{reformulationc} satisfying \eqref{a-priori} for some sufficiently small constant $\delta$. Then for any $ n \geq 2 $, there exist energy functionals $\bar{\mathcal{E}}_{0}$ and $\bar{\mathcal{D}}_{0}$, such that
\begin{align*}
\bar{\mathcal{E}}_{0}(t) \sim \big(\|\mathbf{u}\|_0^2+\|\partial_3\boldsymbol{\eta}\|^2_0+\|\mathbf{w}\|_{0}^{2}\big),\qquad
\bar{\mathcal{D}}_{0}(t) \sim \big(\|\mathbf{u}\|_{1}^{2}+\|\mathbf{w}\|_{0}^{2} + \|\partial_t \mathbf{w}\|_0^2\big),
\end{align*}
and
\begin{align}\label{con-lem-L2-diff}
 \frac{\mathrm{d}}{\mathrm{d}t}\bar{\mathcal{E}}_{0}+\bar{\mathcal{D}}_{0}\lesssim \sqrt{\mathcal{E}_{n}}\mathcal{D}_{n}.
\end{align}
\end{lemma}

\begin{proof}
Taking $L^2$ inner product of $\eqref{reformulationc}_2$ with $\mathbf{u}$ and using the boundary condition, we have
\begin{align}\label{u-0-ineq}
&\frac{1}{2} \frac{\mathrm{d}}{\mathrm{d}t}\big( \|\mathbf{u}\|_0^2 +\|\partial _{3}\boldsymbol{\eta}\|_0^2 \big)
+  (\mu+\zeta) \int_{\Omega}  \vert \nabla _{\mathcal{A}}  \mathbf{u} \vert^2\mathrm{d}\mathbf{x} -  \zeta \int _{\Omega} (\nabla _{\mathcal{A}} \times  \mathbf{w}) \cdot \mathbf{u}\mathrm{d}\mathbf{x}=0,
\end{align}
where we also used $\eqref{reformulationc}_1$. From $\eqref{reformulationc}_3$, we can show that
\begin{gather}\label{w-0-iden}
 \frac{1}{2}\frac{\mathrm{d}}{\mathrm{d}t}\|\mathbf{w} \|_{0}^{2} + 2 \zeta\int _{\Omega}  \vert  \mathbf{w}\vert ^{2}\mathrm{d}\mathbf{x} - \zeta \int _{\Omega}  (\nabla _{\mathcal{A}} \times  \mathbf{u}) \cdot  \mathbf{w} \mathrm{d}\mathbf{x} =0.
 \end{gather}
Let $\mathfrak{R}_{1}^{0}$ and $\mathfrak{R}_{2}^{0}$ denote the last terms on the left sides of \eqref{u-0-ineq} and \eqref{w-0-iden}, respectively. Combining \eqref{u-0-ineq} and \eqref{w-0-iden} gives rise to
 \begin{align}\label{E-0-COUPLED}
  \frac{1}{2}\frac{\mathrm{d}}{\mathrm{d}t}\big(\|\mathbf{u}\|_0^2 +\|\partial _{3}\boldsymbol{\eta}\|_0^2 + \|\mathbf{w}\|_{0}^{2} \big) + \mu \|\nabla _{\mathcal{A}} \mathbf{u} \|_0^2 + \zeta \int_{\Omega} \big( \vert \nabla _{\mathcal{A}} \mathbf{u} \vert^2 + 2 \vert \mathbf{w}\vert ^{2}\big) \mathrm{d}\mathbf{x} -\mathfrak{R}_{1}^{0}-\mathfrak{R}_{2}^{0}=0,
 \end{align}
 where, by virtue of $ \mathbf{u} \vert _{\partial \Omega}=\mathbf{0} $,  $ \operatorname{div}_\mathcal{A}\mathbf{u}=0 $ and integration by parts, we have
 \begin{align}\label{l2-diss-recover}
 &\ \zeta \int_{\Omega}\big( \vert \nabla _{\mathcal{A}}  \mathbf{u} \vert^2 + 2 \vert \mathbf{w}\vert ^{2} \big)\mathrm{d}\mathbf{x} -\mathfrak{R}_{1}^{0}-\mathfrak{R}_{2}^{0}\nonumber\\
   =&\ \zeta \int_{\Omega} \big( \vert \nabla _{\mathcal{A}} \times \mathbf{u} \vert^2 + 2 \vert \mathbf{w}\vert ^{2} \big)\mathrm{d}\mathbf{x} - 2\zeta\int _{\Omega}(\nabla _{\mathcal{A}}\times \mathbf{u}) \cdot \mathbf{w}\mathrm{d}\mathbf{x}
   =\zeta\|\mathbf{w}\|_0^2+\zeta\| \nabla _{\mathcal{A}}\times \mathbf{u}- \mathbf{w}\|_0^2.
 \end{align}
 Then we update \eqref{E-0-COUPLED} as
 \begin{align}\label{E-0-T-FINAL}
\frac{\mathrm{d}}{\mathrm{d}t} \big(\|\mathbf{u}\|_0^2 +\|\partial _{3} \boldsymbol{\eta}\|_0^2 + \|\mathbf{w} \|_{0}^{2} \big)
   + 2\mu \|\nabla _{\mathcal{A}}  \mathbf{u} \|_0^2 + 2\zeta\|\mathbf{w}\|_0^2 + 2\zeta\| \nabla _{\mathcal{A}}\times \mathbf{u}-2 \mathbf{w}\|_0^2=0.
 \end{align}
   Moreover, taking the $L^2$ inner product of $\eqref{reformulationc}_3$ with $\partial_t \mathbf{w} $ gives us
 \begin{align}\label{x2}
  \zeta\frac{\mathrm{d}}{\mathrm{d}t} \|\mathbf{w} \|_{0}^{2} + \|\partial _{t} \mathbf{w}\|_0^2 = \int _{\Omega} \zeta (\nabla _{\mathcal{A}} \times \mathbf{u}) \cdot \partial_t \mathbf{w}\mathrm{d}\mathbf{x}
    \leq  \frac{1}{2} \| \partial _{t} \mathbf{w}\|_0^{2} + \frac{\zeta^2}{2} \|\nabla _{\mathcal{A}}\times \mathbf{u}\|_0^{2}.
  \end{align}
  Since $ \|\nabla _{\mathcal{A}}\times \mathbf{u}\|_{0} \leq \|\nabla _{\mathcal{A}}\mathbf{u}\|_{0} $, we get from \eqref{x2}:
   \begin{align}\label{W-T-MULTI}
  \zeta\frac{\mathrm{d}}{\mathrm{d}t} \|\mathbf{w} \|_{0}^{2} + \frac12 \|\partial _{t} \mathbf{w}\|_0^2
    \leq \frac{\zeta^2}{2} \|\nabla _{\mathcal{A}} \mathbf{u}\|_0^{2}.
  \end{align}
Multiplying \eqref{W-T-MULTI} with $\frac{2\mu}{\zeta^2}$, and then adding the result to \eqref{E-0-T-FINAL}, we obtain
   \begin{align}\label{x3}
\frac{\mathrm{d}}{\mathrm{d}t} \Big(\|\mathbf{u}\|_0^2 +\|\partial _{3}\boldsymbol{\eta}\|_0^2 +\frac{2\mu+\zeta}{\zeta} \|\mathbf{w} \|_{0}^{2} \Big)
   + \mu \|\nabla _{\mathcal{A}}  \mathbf{u} \|_0^2 + 2\zeta\| \mathbf{w}\|_0^2 + \frac{\mu}{\zeta^2} \|\partial _{t} \mathbf{w}\|_0^2  \le 0.
  \end{align}
where we dropped the non-negative term $2\zeta\| \nabla _{\mathcal{A}}\times \mathbf{u}- \mathbf{w}\|_0^2$ from the left hand side. This implies the existence of an energy functional $ \bar{\mathcal{E}}_{0} $ such that
\begin{align*}
\bar{\mathcal{E}}_{0}(t) \sim \big(\|\mathbf{u}\|_0^2 +\|\partial _{3}\boldsymbol{\eta}\|_0^2+\|\mathbf{w}\|_{0}^{2}\big)
\end{align*}
and
  \begin{align}\label{barE-0-DIFF-CON}
   \frac{\mathrm{d}}{\mathrm{d}t}\bar{\mathcal{E}}_{0} + \big( \| \nabla _{\mathcal{A}} \mathbf{u}\|_0 ^{2} + \| \mathbf{w}\|_0^{2} + \|\partial _{t} \mathbf{w}\|_0^{2} \big) \lesssim 0.
  \end{align}
 In order to replace $ \nabla_{\mathcal{A}}  \mathbf{u}$ with $ \nabla \mathbf{u} $
in \eqref{barE-0-DIFF-CON}, we write
\begin{align}\label{L1}
 \vert \nabla _{\mathcal{A}} \mathbf{u}\vert^2 = \vert \nabla \mathbf{u}\vert^2+\left(\nabla _{\mathcal{A}}  \mathbf{u} + \nabla \mathbf{u}\right): \left(\nabla _{\mathcal{A}}  \mathbf{u} - \nabla \mathbf{u}\right),
\end{align}
where $:$ denotes the Frobenius inner product. Recall for each component $u_k$ $(k=1,2,3)$ of $\mathbf{u}$, $(\nabla_{\mathcal{A}} u_k)_i = \mathcal{A}_{ij}\partial_j u_k$. Hence, for each $u_k$, we have $(\nabla _{\mathcal{A}} u_k - \nabla u_k)_i = (\mathcal{A}_{ij} - \delta_{ij}) \partial_j u_k$. Since $ \mathcal{A}=[(\mathcal{I}_3+\nabla {\boldsymbol{\eta}})^{-1}]^{\mathrm{T}} $, under the {\it a priori} assumption \eqref{a-priori} and when $\delta$ is sufficiently small, $\mathcal{A}$ is well defined and the maximum norms of $(\mathcal{A}_{ij} - \delta_{ij})$ are bounded (according to    \eqref{p_energy_defc} and Sobolev embedding) by $\sqrt{\mathcal{E}_n}$. Moreover, for the same reason, the maximum norms of $(\mathcal{A}_{ij} + \delta_{ij})$ are uniformly bounded by some constant around 2. Utilizing such information, we know that
\begin{align}\label{L2}
\big| \int_{\Omega} \left(\nabla _{\mathcal{A}}  \mathbf{u} + \nabla \mathbf{u}\right): \left(\nabla _{\mathcal{A}}  \mathbf{u} - \nabla \mathbf{u}\right) \mathrm{d}\mathbf{x} \big| \lesssim  \sqrt{\mathcal{E}_n} \|\nabla\mathbf{u}\|_0^2 \lesssim \sqrt{\mathcal{E}_n} \mathcal{D}_n,
\end{align}
where we have used \eqref{p_dissipation_defc}. Integrating \eqref{L1} over $\Omega$ and using \eqref{L2}, we have
\begin{equation}\label{i_te_9c}
  \| \nabla _{\mathcal{A}}  \mathbf{u}\|_0^2 \gtrsim  \| \nabla \mathbf{u}\|_0^2 - \sqrt{\mathcal{E}_{n} } \mathcal{D}_{n}.
\end{equation}
 Inserting \eqref{i_te_9c} into \eqref{barE-0-DIFF-CON} and invoking the Poincar\'e inequality for $\mathbf{u}$, we get
\begin{align}\label{barE-0-DIFF-CON-final}
   \frac{\mathrm{d}}{\mathrm{d}t}\bar{\mathcal{E}}_{0}+ \big( \|\mathbf{u}\|_{1}^{2} + \| \mathbf{w}\|_0^{2} + \|\partial _{t} \mathbf{w}\|_0^{2} \big)  \lesssim \sqrt{\mathcal{E}_{n} } \mathcal{D}_{n},
\end{align}
This completes the proof of the lemma.
\end{proof}

Next, we derive the \emph{a priori} estimates of the temporal derivatives of the solution. For this, applying $\partial _{t}^j$ with $j=1,\dots,n$ to the system \eqref{reformulationc}, we have
\begin{equation}\label{linear_geometricc}
\begin{aligned}
  \partial _{t} ( \partial _{t}^j \boldsymbol{\eta}) &= \partial _{t}^j \mathbf{u}, \\
\partial _{t} ( \partial _{t}^j \mathbf{u}) -(\mu+\zeta)\Delta_{\mathcal{A}}\partial _{t}^j \mathbf{u}+\nabla_{\mathcal{A}}\partial^j_{t}p
 &=\partial_3^2\partial _{t}^j \boldsymbol{\eta} +  \zeta \nabla _{\mathcal{A}} \times \partial _{t}^{j} \mathbf{w}+ \mathbf{F}^{1,j},\\
\partial _{t} (\partial _{t}^{j} \mathbf{w})+2 \zeta \partial _{t}^{j} \mathbf{w} &=\zeta \nabla _{\mathcal{A}} \times \partial _{t}^{j} \mathbf{u} + \mathbf{F} ^{2,j} ,\\
\operatorname{div}_{\mathcal{A}}\partial^{j}_t \mathbf{u} &=F^{3,j},\\
\partial _{t}^j \mathbf{u} \vert_{\partial\Omega}&=\mathbf{0},
\end{aligned}
\end{equation}
where, for $i=1,2,3,$
\begin{align}
 \mathbf{F}^{1,j} & = \sum_{0 < \ell \le j}  C_j^\ell\big\{(\mu+\zeta)\operatorname{div}_{\mathcal{A}}(\nabla_{\partial^{\ell}_t\mathcal{A}}\partial^{j-\ell}_t\mathbf{u})+(\mu+\zeta)\operatorname{div}_{\partial^{\ell}_t\mathcal{A}}\partial^{j-\ell}_t(\nabla_{\mathcal{A}}\mathbf{u})-\nabla_{\partial^{\ell}_t\mathcal{A}}\partial^{j-\ell}_tp \notag\\
 &\qquad\qquad \quad +2 \zeta \nabla_{\partial^{\ell}_t\mathcal{A}}\times\partial _{t}^{j- \ell}\mathbf{ w}\big\}, \label{F_def_startc1}\\
\mathbf{F}^{2,j}& = 2\zeta\sum_{0 < \ell \le j}  C_j^\ell\nabla_{\partial^{\ell}_t\mathcal{A}}\times \partial^{j-\ell}_t\mathbf{u}, \label{F_def_startc2}\\
F^{3,j}&=\operatorname{div} \mathbf{Q}^{3,j} \quad \operatorname{with} \quad Q^{3,j}_i=-\sum_{0<\ell\leq j}C^{\ell}_j\partial^{\ell}_t\mathcal{A}_{mi}\partial^{j-\ell}_tu_m.\label{g}
\end{align}
We remark that the linear part of the system \eqref{linear_geometricc} for $ (\partial _{t}^{j}\boldsymbol{\eta}, \partial _{t}^{j}\mathbf{u}, \partial_t^jp, \partial _{t}^{j}\mathbf{w}) $ has essentially the same structure as \eqref{reformulationc} for $ (\boldsymbol{\eta},\mathbf{u}, p,\mathbf{w}) $. In order to obtain the estimates of the temporal derivatives of the solution, it is crucial to derive the estimates of $\mathbf{F}^{1,j}$, $\mathbf{F}^{2,j}$ and $F^{3,j}$. We have the following:

\begin{lemma}\label{lem-}
Let $ (\boldsymbol{\eta},\mathbf{u}, p,\mathbf{w}) $ be the local solution to \eqref{reformulationc} satisfying \eqref{a-priori} for some sufficiently small constant $\delta$. Then for $ n=N+2 $ or $ n=2N $ with $ N \geq 4 $, and for $ j=1,\dots, n $,
\begin{align}\label{p_F_e_001c}
\| \mathbf{F}^{1,j}\|_{0}^{2}+ \|\mathbf{F}^{2,j}\|_{0}^{2} +\|\mathbf{Q}^{3,j}\|^2_{0}+\|\partial_t\mathbf{Q}^{3,j}\|^2_{0}\lesssim \mathcal{E}_{N+2}\mathcal{D}_{n}
 \end{align}
 and
 \begin{align}\label{p_F_e_02}
 \|\mathbf{Q}^{3,j}\|^2_{0}\lesssim \mathcal{E}_{N+2}\mathcal{E}_n.
 \end{align}
\end{lemma}

\begin{proof} We shall give a detailed estimate of $\mathbf{F}^{1,j}$. The other terms can be treated similarly. First, we note that under the {\it a priori} assumption \eqref{a-priori}, the entries of $\mathcal{A}$ are uniformly bounded by some constant. By Sobolev embedding, we have the following estimates when $0 < \ell\leq \frac{j}{2}$ and $n\geq 3$:
\begin{align}\label{F1}
&\,\|\sum_{0 < \ell \le \frac{j}{2}}  C_j^\ell(\mu+\zeta)\operatorname{div}_{\mathcal{A}}(\nabla_{\partial^{\ell}_t\mathcal{A}}\partial^{j-\ell}_t\mathbf{u})\|_0^2\nonumber\\
\lesssim &\,\sum_{0<\ell\leq \frac{j}{2}}\big(\|\nabla\partial^{\ell}_t\mathcal{A}\partial_t^{j-\ell}\nabla\mathbf{u}\|_0^2+\|\partial_t^{\ell}\mathcal{A}\partial^{j-\ell}_t\nabla ^{2}\mathbf{u}\|_0^2\big)\nonumber\\
\lesssim &\,\sum_{0<\ell\leq \frac{j}{2}}\big(\|\partial^{\ell-1}_t\nabla^2\mathbf{u}\|^2_{L^{\infty}}\|\partial_t^{j-\ell}\nabla \mathbf{u}\|_0^2+\|\partial_t^{\ell-1}\nabla \mathbf{u}\|^2_{L^\infty}\|\partial^{j-\ell}_t\nabla^2\mathbf{u}\|_0^2\big)\nonumber\\
\lesssim &\, \sum_{0<\ell\leq \frac{j}{2}} \|\partial^{\ell-1}_t\mathbf{u}\|^2_{2n-2\ell+2}\|\partial_t^{j-\ell}\mathbf{u}\|^2_{2n-2j+2\ell+1}\leq\mathcal{E}_{N+2}\mathcal{D}_{n}.
\end{align}
Here, when computing the time derivative of $\mathcal{A}$, we used the equation $\partial_t \boldsymbol{\eta} =\mathbf{u}$, and for the last inequality in \eqref{F1}, we used the definitions in \eqref{p_energy_defc}--\eqref{p_dissipation_defc}. For $\frac{j}{2} < \ell \le j$, we have
\begin{align}\label{F2}
&\,\|\sum_{\frac{j}{2}<\ell\leq j}  C_j^\ell(\mu+\zeta)\operatorname{div}_{\mathcal{A}}(\nabla_{\partial^{\ell}_t\mathcal{A}}\partial^{j-\ell}_t\mathbf{u})\|_0^2\nonumber\\
\lesssim &\,\sum_{\frac{j}{2}<\ell\leq j} \big(\|\partial^{\ell}_t\nabla\mathcal{A}\partial_t^{j-\ell}\nabla\mathbf{u}\|_0^2+\|\partial_t^{\ell}\mathcal{A}\partial^{j-\ell}_t\nabla^2\mathbf{u}\|_0^2\big)\nonumber\\
\lesssim &\,\sum_{\frac{j}{2}<\ell\leq j} \big(\|\partial^{\ell-1}_t\nabla^2\mathbf{u}\|_0^2\|\partial_t^{j-\ell}\nabla \mathbf{u}\|^2_{L^{\infty}}+\|\partial_t^{\ell-1}\nabla \mathbf{u}\|_0^2\|\partial^{j-\ell}_t\nabla^2\mathbf{u}\|^2_{L^{\infty}}\big)\nonumber\\
\lesssim &\, \sum_{\frac{j}{2}<\ell\leq j} \|\partial^{\ell-1}_t\mathbf{u}\|^2_{2n-2\ell+3}\|\partial_t^{j-\ell}\mathbf{u}\|^2_{2n-2j+2\ell}\leq\mathcal{E}_{N+2}\mathcal{D}_{n}.
\end{align}
Combining \eqref{F1} and \eqref{F2}, we get the estimate of the first summation on the right of \eqref{F_def_startc1}:
\begin{align}\label{F3}
\|\sum_{0<\ell\leq j}  C_j^\ell(\mu+\zeta)\operatorname{div}_{\mathcal{A}}(\nabla_{\partial^{\ell}_t\mathcal{A}}\partial^{j-\ell}_t\mathbf{u})\|_0^2\leq \mathcal{E}_{N+2}\mathcal{D}_n.
\end{align}
The estimate of the second summation is similar to \eqref{F3}. We omit the details for brevity. For the third one, we can show that
\begin{align}\label{F4}
&\,\|\sum_{0<\ell\leq j} C^{\ell}_j \nabla_{\partial^{\ell}_t\mathcal{A}}\partial^{j-\ell}_tp \|_0^2 \lesssim
\|\sum_{0<\ell\leq \frac{j}{2}} \nabla_{\partial^{\ell}_t\mathcal{A}}\partial^{j-\ell}_tp\|_0^2 +  \|\sum_{\frac{j}{2}<\ell\leq j} \nabla_{\partial^{\ell}_t\mathcal{A}}\partial^{j-\ell}_tp\|_0^2 \nonumber\\
\lesssim &\, \sum_{0<\ell\leq \frac{j}{2}} \|\partial^{\ell-1}_t\nabla \mathbf{u}\|^2_{L^{\infty}}\|\partial^{j-\ell}_t\nabla p\|_0^2 + \sum_{\frac{j}{2}<\ell\leq j} \|\partial^{\ell-1}_t\nabla \mathbf{u}\|_0^2\|\partial^{j-\ell}_t\nabla p\|^2_{L^{\infty}} \nonumber\\
\lesssim &\, \sum_{0<\ell\leq \frac{j}{2}} \|\partial^{\ell-1}_t\mathbf{u}\|^2_{2n-2\ell+2}\|\partial^{j-\ell}_t \nabla p\|_0^2 + \sum_{\frac{j}{2}<\ell\leq j} \|\partial^{\ell-1}_t\nabla \mathbf{u}\|_0^2\|\partial^{j-\ell}_t\nabla p\|^2_{{2n-2j+2\ell-2}} \lesssim \mathcal{E}_{N+2}\mathcal{D}_n,
\end{align}
where we have used $ N \geq 4 $. The estimate of the last summation in \eqref{F_def_startc1} is similar to \eqref{F4}. Again we omit the details for brevity. This completes the proof of the lemma.
\end{proof}

With Lemma \ref{lem-} at our disposal, we deduce the estimates for $ (\partial _{t}^{j}\boldsymbol{\eta}, \partial _{t}^{j}\mathbf{u}, \partial _{t}^{j}\mathbf{w}) $.

\begin{lemma}\label{t-estimate}
Let $ (\boldsymbol{\eta},\mathbf{u}, p,\mathbf{w}) $ be the local solution to \eqref{reformulationc} satisfying \eqref{a-priori} for some sufficiently small constant $\delta$. Then for $ n=N+2 $ or $ n=2N $ with $ N \geq 4 $, there exist energy functionals $\bar{\mathcal{E}}_{n} $ and $ \bar{\mathcal{D}}_{n} $, such that
\begin{align}
\bar{\mathcal{E}}_{n}(t) &\sim \sum_{j=0}^{n} \big(\|\partial _{t}^j \mathbf{u}\|_0^2+ \|\partial _{t}^{j}\partial _{3}\boldsymbol{\eta} \|_0^2+\|\partial _{t}^{j}\mathbf{w}\|_{0}^{2}\big),\label{Ent}\\
\bar{\mathcal{D}}_n(t) &\sim \sum_{j=0}^{ n}\big(\|\partial _{t}^{j}\mathbf{u}\|_{1}^{2}+ \|\partial _{t}^{j}\mathbf{w}\|_{0}^{2} + \|\partial _{t}^{j+1}\mathbf{w}\|_{0}^{2}\big), \label{Dnt}
\end{align}
and
\begin{align}\label{barEN-differ-esti}
 \frac{\mathrm{d}}{\mathrm{d}t}\Big(\bar{\mathcal{E}}_{n}+\int_{\Omega}\nabla\partial^{n-1}_tp \cdot \mathbf{Q}^{3,n}\mathrm{d}\mathbf{x}\Big)+ \bar{\mathcal{D}}_{n} \lesssim \sqrt{\mathcal{E}_{N+2}}\mathcal{D}_{n}.
\end{align}
\end{lemma}

\begin{proof}
Taking $L^2$ inner product of $\eqref{linear_geometricc}_2$ and $\eqref{linear_geometricc}_3$ with $\partial _{t}^{j} \mathbf{u}$ and $ \partial _{t}^{j} \mathbf{w} $ ($ j=0,1,\dots,n $), respectively, and using similar arguments as in \eqref{l2-diss-recover}, we can show that
\begin{align}\label{E-n-T-ineq}
&\ \frac{\mathrm{d}}{\mathrm{d}t}\big(\|\partial _{t}^j \mathbf{u}\|_0^2+ \|\partial _{t}^{j}\partial _{3}\boldsymbol{\eta} \|_0^2+ \|\partial _{t}^{j}\mathbf{w}\|_{0}^{2}\big) +  2\mu\| \nabla _{\mathcal{A}} \partial _{t}^j  \mathbf{u} \|_0^2 + 2\zeta\|\partial_t^j \mathbf{w}\|_0^2 \nonumber \\
\le &\ 2\int_\Omega \big(\partial _{t}^{j} \mathbf{u}\cdot \mathbf{F} ^{1,j}+\partial^j_{t}p\,F^{3,j}\big) \mathrm{d}\mathbf{x} +2\int _{\Omega}\partial _{t}^{j}\mathbf{w} \cdot \mathbf{F}^{2,j}\mathrm{d}\mathbf{x},
 \end{align}
 where we dropped two non-negative terms, namely, $2\zeta\|\operatorname{div}_{\mathcal{A}}\partial^{j}_t  \mathbf{u}\|_0^2$ and $2\zeta\| \nabla _{\mathcal{A}}\times \partial _{t}^j \mathbf{u}- \partial _{t}^{j} \mathbf{w}\|_0^{2}$. By virtue of \eqref{p_F_e_001c} and H\"older's inequality, we have
 \begin{align}
 2\int_\Omega \partial _{t}^{j} \mathbf{u}\cdot \mathbf{F}^{1,j}\mathrm{d}\mathbf{x} + 2\int _{\Omega}\partial _{t}^{j}\mathbf{w} \cdot \mathbf{F}^{2,j}\mathrm{d}\mathbf{x} \lesssim \|\partial _{t}^{j}\mathbf{u}\|_{0} \|\mathbf{F} ^{1,j}\|_{0}+ \|\partial _{t}^{j}\mathbf{w}\|_{0}\|\mathbf{F} ^{2,j}\|_{0}\lesssim \sqrt{\mathcal{D}_{n}}\sqrt{\mathcal{E}_{N+2}\mathcal{D}_{n}}.
 \end{align}
 The estimate involving  $F^{3,j}$ is a bit subtle. We first utilize the structure of $F^{3,j}$ and employ integration by parts in space. Indeed, by \eqref{g} and the boundary condition, we have
 \begin{align}
 2\int_{\Omega}\partial^j_{t}p\, F^{3,j}\mathrm{d}\mathbf{x}=2\int_{\Omega}\partial^j_{t}p\, \operatorname{div}\mathbf{Q}^{3,j}\mathrm{d}\mathbf{x}=-2\int_{\Omega}\nabla\partial^j_{t}p \cdot \mathbf{Q}^{3,j}\mathrm{d}\mathbf{x}.
 \end{align}
Then for $j<n$, by \eqref{p_F_e_001c}, we deduce that
 \begin{align}
 -2\int_{\Omega}\nabla\partial^j_{t}p \cdot \mathbf{Q}^{3,j}\mathrm{d}\mathbf{x}\leq 2\|\nabla\partial^j_{t}p\|_0\|\mathbf{Q}^{3,j}\|_0\lesssim\sqrt{\mathcal{D}_n}\sqrt{\mathcal{E}_{N+2}\mathcal{D}_n}.
 \end{align}
When $j=n$, we do not have a proper bound of $\nabla\partial^n_{t}p$. We integrate by parts in time to get
 \begin{align}
 -2\int_{\Omega}\nabla\partial^n_{t}p \cdot \mathbf{Q}^{3,n}\mathrm{d}\mathbf{x}=- \frac{\mathrm{d}}{\mathrm{d}t} \int_{\Omega} 2\nabla\partial^{n-1}_tp \cdot \mathbf{Q}^{3,n}\mathrm{d}\mathbf{x}+2\int_{\Omega}\nabla\partial^{n-1}_tp \cdot \partial_t \mathbf{Q}^{3,n}\mathrm{d}\mathbf{x},
 \end{align}
 where, thanks to \eqref{p_F_e_001c}, we have
 \begin{align}
 2\int_{\Omega}\nabla\partial^{n-1}_tp \cdot \partial_t \mathbf{Q}^{3,n}\mathrm{d}\mathbf{x}\lesssim\sqrt{\mathcal{D}_n}\sqrt{\mathcal{E}_{N+2}\mathcal{D}_n}.
 \end{align}
We thus update \eqref{E-n-T-ineq} as
 \begin{align}\label{u-t-j-ineq}
  &  \frac{\mathrm{d}}{\mathrm{d}t}\Big(\|\partial _{t}^j \mathbf{u}\|_0^2+ \|\partial _{t}^{j}\partial _{3}\boldsymbol{\eta}\|_0^2+\|\partial _{t}^{j}\mathbf{w}\|_{0}^{2} +2\int_{\Omega}\nabla\partial^{n-1}_tp \cdot \mathbf{Q}^{3,n}\mathrm{d}\mathbf{x}\Big) + 2\mu\| \nabla _{\mathcal{A}} \partial _{t}^j  \mathbf{u} \|_0^2 + 2\zeta\|\partial_t^j \mathbf{w}\|_0^2\notag \\
  \lesssim & \sqrt{\mathcal{E}_{N+2}}\mathcal{D}_{n}.
  \end{align}
On the other hand, taking the $L^2$ inner product of $\eqref{linear_geometricc}_3$ with $\partial _{t}^{j+1} \mathbf{w}$, we get
 \begin{align}\label{w-j-1-ineq}
 &\  \zeta\frac{\mathrm{d}}{\mathrm{d}t} \| \partial _{t}^{j}\mathbf{w}\|_0^{2} + \| \partial _{t}^{j+1}\mathbf{w}\|_0^{2}  =  \zeta \int _{\Omega} \partial _{t}^{j+1} \mathbf{w} \cdot (\nabla _{\mathcal{A}} \times \partial _{t}^{j} \mathbf{u}) \mathrm{d}\mathbf{x} + \int _{\Omega}\partial _{t}^{j+1}\mathbf{w} \cdot \mathbf{F}^{2,j}\mathrm{d}\mathbf{x}
  \nonumber \\
  \leq &\ \frac{1}{2} \| \partial _{t}^{j+1}\mathbf{w}\|_0^{2} + \frac{\zeta^2}{2} \|\nabla _{\mathcal{A}}\times\partial _{t}^{j}\mathbf{u}\|_0^{2} + \|\partial _{t}^{j+1}\mathbf{w}\|_{0}\|\mathbf{F}^{2,j}\|_{0}
   \nonumber \\
   \leq &\ \frac{1}{2} \| \partial _{t}^{j+1}\mathbf{w}\|_0^{2} + \frac{\zeta^2}{2} \| \nabla _{\mathcal{A}}\times\partial _{t}^{j}\mathbf{u}\|_0^{2} + C\sqrt{\mathcal{E}_{N+2}}\mathcal{D}_{n}.
 \end{align}
Manipulating \eqref{u-t-j-ineq} and \eqref{w-j-1-ineq} as in the proof of Lemma \ref{Lem-L2}, we conclude the existence of energy functionals ${\bar{\mathcal{E}}_{n}} $ and $ {\bar{\mathcal{D}}_{n}} $ fulfilling \eqref{Ent}--\eqref{barEN-differ-esti}. This completes the proof of the lemma.
\end{proof}

\section{Estimates on Horizontal Derivatives}

Due to the geometric configuration of $\Omega=\mathbb{R}^2\times (0,1)$, we separate the \emph{a priori} estimates of the spatial derivatives of the solution in two sections. This section contains the estimates in the horizontal directions. We first introduce further notations for convenience.

\begin{notation}\label{notation2}
Similarly to $\mathbb{Z}_+^{1+m}$ (see Notation \ref{notation1}), we use $\mathbb{Z}_+^m$ to denote the set of multi-indices for spatial derivatives only. We will write $\nabla_\ast$ for the horizontal gradient, $\operatorname{div}_\ast$ for the horizontal divergence and $\Delta_\ast$ for the horizontal Laplace operator. We also introduce the space-time energy:
\begin{align*}
\|\bar{\nabla}^l f\|_k^2 := \sum_{\substack{\alpha \in \mathbb{Z}_+^{1+3}, \vert \alpha\vert\le l} } \|\partial ^{\alpha}f\|_{k}^2, \quad l\ge0,\quad k \geq 0.
\end{align*}
\end{notation}

The following two subsections deal with the horizontal estimates of $(\mathbf{u},\mathbf{w})$ and $\boldsymbol{\eta}$, respectively.

\subsection{Horizontal estimates of $ (\mathbf{u}, \mathbf{w}) $} 

As in \cite{Feng-hong-zhu-compre}, we consider the linear perturbed formulation:
\begin{equation}\label{perturbc}
\begin{aligned}
\partial_t \boldsymbol{\eta} &=\mathbf{u}, \\
\partial_t \mathbf{u} -(\mu+\zeta)\Delta  \mathbf{u}+\nabla p &=\partial_3^2\boldsymbol{\eta}+ \zeta \nabla  \times  \mathbf{w}+ \mathbf{G}^1, \\
\partial _{t}\mathbf{w}+2 \zeta  \mathbf{w} &=\zeta \nabla  \times \mathbf{u} + \mathbf{G}^2,\\
\operatorname{div}\mathbf{u} &=G^3,\\
\mathbf{u}|_{\partial\Omega}&= \mathbf{0},
\end{aligned}
\end{equation}
where, for $i=1,2,3$,
\begin{align*}
\mathbf{G}^1 &= (\mu+\zeta)(\Delta_{\mathcal{A}} \mathbf{u}-\Delta\mathbf{u})-(\nabla_{\mathcal{A}}p-\nabla p)+\zeta(\nabla_{\mathcal{A}}\times \mathbf{w}-\nabla\times \mathbf{w}),\\
\mathbf{G}^2 &= \zeta(\nabla_{\mathcal{A}}\times \mathbf{u}-\nabla\times \mathbf{u}),\\
G^3 &= -(\operatorname{div}_{\mathcal{A}}\mathbf{u}-\operatorname{div}\mathbf{u}).
\end{align*}
We begin with the estimates of the space-time energy of $\mathbf{G}^1$, $\mathbf{G}^2$ and $G^3$.

\begin{lemma}\label{lem-G-TILDA-G}
Let $ (\boldsymbol{\eta},\mathbf{u}, p,\mathbf{w}) $ be the local solution to \eqref{reformulationc} satisfying \eqref{a-priori} for some sufficiently small constant $\delta$. Then for $ n=N+2 $ or $ n=2N $ with $ N \geq 1 $, it holds that
\begin{align}
\bullet\quad &\,\|\bar{\nabla}^{2n-2} \mathbf{G}^1\|_0^{2}+\|\bar{\nabla}^{2n-2}\mathbf{G}^2\|^2_{1}+\|\bar{\nabla}^{2n-2}G^3\|^2_{1} \lesssim  \mathcal{E}_{N+2}\mathcal{E}_{n}, \label{p_G_e_0c}\\
\bullet\quad &\,\|\bar{\nabla}^{4N-1}\mathbf{G}^1\|_0^{2}+\|\bar{\nabla}^{4N-1}\mathbf{G}^2\|^2_1+ \|\bar{\nabla}^{4N-1}G^3\| ^{2}_1 \lesssim \mathcal{E}_{N+2}(\mathcal{D}_{2N}+\mathcal{J}_{2N}+\mathcal{F}_{2N}), \label{p_G_e_001c}\\
\bullet\quad &\,\|\bar{\nabla}^{2N+3}\mathbf{G}^1\|^2_{1}+\|\bar{\nabla}^{2N+3}\mathbf{G}^2\|^2_{2}+\|\bar{\nabla}^{2N+3}G^3\|^2_{2} \lesssim\mathcal{E}_{2N}\mathcal{D}_{N+2}. \label{p_G_e_002c}
\end{align}
\end{lemma}

To prove this lemma, we apply the space-time differential operator to $ \mathbf{G}^1 $, $ \mathbf{G}^2 $ and $G^3$, and then expand by the Leibniz rule. The resulting quantities resemble similar structures as those in \eqref{F_def_startc1}--\eqref{g}. The proof then follows from similar arguments as in the proof of Lemma \ref{lem-}. We omit the tedious details here for brevity. Using Lemma \ref{lem-G-TILDA-G}, we can show the following:

\begin{lemma}\label{horizontal-uw}
Let $ (\boldsymbol{\eta},\mathbf{u}, p,\mathbf{w}) $ be the local solution to \eqref{reformulationc} satisfying \eqref{a-priori} for some sufficiently small constant $\delta$. Then for $ n=N+2 $ or $ n=2N $ with $ N \geq 4 $, there exist energy functionals $ \bar{\mathcal{E}}_{n}^{\ast} $ and $  \bar{\mathcal{D}}_{n}^{\ast} $, such that
\begin{align}
\bar{\mathcal{E}}^{\ast}_{n}(t) &\sim \sum^{n-1}_{j=0}\big(  \| \partial^j_{t} \mathbf{u}\| _{0,2n-2j}^2
 + \|\partial _{t}^{j} \partial _{3} \boldsymbol{\eta} \|_{0,2n-2j}^2 +\| \partial^{j}_{t}\mathbf{w}\|_{0,2n-2j}^{2} \big), \label{En*}\\
\bar{\mathcal{D}}_n^{\ast}(t) &\sim \sum^{n-1}_{j=0} \big( \|\partial _{t}^{j}\mathbf{u}\|_{1,2n-2j}^2+ \|\partial^{j}_t\mathbf{w}\|_{0,2n-2j}^{2} + \|\partial^{j+1}_t\mathbf{w}\|_{0,2n-2j}^{2} \big), \label{Dn*}
\end{align}
and the following hold:
 \begin{equation}\label{p_u_e_00c}
\frac{\mathrm{d}}{\mathrm{d}t}{\bar{\mathcal{E}}_{2N}^{\ast}}+ {\bar{\mathcal{D}}_{2N}^{\ast}} \lesssim  \sqrt{ \mathcal{E}_{N+2}  } (\mathcal{D}_{2N}+\mathcal{J}_{2N}+\mathcal{F}_{2N})
\end{equation}
and
\begin{equation}\label{p_u_e_00c1}
\frac{\mathrm{d}}{\mathrm{d}t}{\bar{\mathcal{E}}_{N+2}^{\ast}}+ {\bar{\mathcal{D}}_{N+2}^{\ast}}  \lesssim  \sqrt{ \mathcal{E}_{2N}  }\mathcal{D}_{N+2}.
\end{equation}
 \end{lemma}

 \begin{proof}
Given $\alpha=(j,\alpha _{1},\alpha _{2})\in \mathbb{Z}_+^{1+2}$ with $ 1 \leq \vert \alpha\vert=2 j+ \alpha _{1}+\alpha _{2} \leq 2n $ and $ j=0,1,\cdots,n-1 $, applying $ \partial ^{\alpha} $ to $\eqref{perturbc}_2 $ and $ \eqref{perturbc}_3 $, and then taking $
 L^2$ inner products of the resulting equations with $\partial ^\alpha \mathbf{u}$ and $ \partial ^{\alpha}\mathbf{w} $, respectively, we deduce that
\begin{align} \label{p_u_e_1110c}
&\frac{1}{2} \frac{\mathrm{d}}{\mathrm{d}t}  \big(\|\partial ^{\alpha}\mathbf{u}\|_0^2+ \| \partial ^{\alpha}\partial _{3}\boldsymbol{\eta} \|_0^2 + \| \partial ^{\alpha}\mathbf{w}\|_0^{2}\big) + \mu  \| \nabla \partial ^{\alpha}\mathbf{u}\|_0^2 +\zeta\|\partial^\alpha\mathbf{w}\|_0^2
   \nonumber \\
   \le & \int_\Omega   \partial^{\alpha}  \mathbf{u}  \cdot  \partial^{\alpha}  \mathbf{G}^1 \mathrm{d}\mathbf{x} +\int _{\Omega}\partial ^{\alpha}\mathbf{w} \cdot \partial ^{\alpha}\mathbf{G}^2 \mathrm{d}\mathbf{x} +\int_{\Omega}\partial^{\alpha}p\, \partial^{\alpha}G^3 \mathrm{d}\mathbf{x},
\end{align}
where again we dropped two non-negative terms $\zeta\| \operatorname{div} \partial ^{\alpha} \mathbf{u} \|_0^2$ and $\zeta \|\nabla \times \partial ^{\alpha}\mathbf{u}-\partial ^{\alpha}\mathbf{w}\|_0^{2}$. Next, we estimate the integrals on the right of \eqref{p_u_e_1110c}. Since $\vert \alpha\vert \geq 1$, we may write $\alpha = \gamma +(\alpha-\gamma)$ for some $\gamma \in \mathbb{Z}_+^{2}$ with $\vert \gamma\vert=1$.

\vskip .1in
{\bf Case 1: $n=2N$.} By integration by parts and \eqref{p_G_e_001c}, we have
\begin{align}\label{i_de_4c}
  \int_\Omega     \partial^ \alpha  \mathbf{u} \cdot   \partial^{\alpha}  \mathbf{G}^1 \mathrm{d}\mathbf{x}  &=- \int_\Omega     \partial^{\alpha+\gamma}   \mathbf{u} \cdot   \partial^{\alpha-\gamma}  \mathbf{G}^1 \mathrm{d}\mathbf{x}
 \le \left\|\partial^{\alpha+\gamma}   \mathbf{u}\right\|_{0}  \|\partial^{\alpha-\gamma}   \mathbf{G}^1 \|_{0}\nonumber \\
 &\le \left\|\partial ^{\alpha}\mathbf{u}\right\| _{1}  \|\bar{\nabla}^{2n-1}\mathbf{G}^1\|_0
\lesssim \sqrt{ \mathcal{D}_{2N} } \sqrt{ \mathcal{E}_{N+2}(\mathcal{D}_{2N}+\mathcal{J}_{2N}+\mathcal{F}_{2N})} .
\end{align}
Similarly, by \eqref{p_G_e_001c},
\begin{gather}\label{ide4d}
\int_\Omega     \partial^{\alpha}   \mathbf{w} \cdot   \partial^{\alpha} \mathbf{G}^2 \mathrm{d}\mathbf{x} \le  \|\partial _{t}^{j}\mathbf{w}\|_{0,2n-2j} \|\bar{\nabla}^{2n-1} \mathbf{G}^2\|_{1} \lesssim \sqrt{ \mathcal{D}_{2N} } \sqrt{ \mathcal{E}_{N+2} (\mathcal{D}_{2N}+\mathcal{J}_{2N}+\mathcal{F}_{2N})}.
\end{gather}
For $G^3$, it follows from \eqref{p_G_e_001c} that
\begin{align}\label{ide4e}
\int_{\Omega}\partial^{\alpha}p\, \partial^{\alpha}G^3 \mathrm{d}\mathbf{x} \leq\|\partial^{\alpha-\gamma}\partial^{\gamma}p\| _{0}\|\partial^{\alpha}G^3\| _{0}
  &\leq  \|\nabla _{\ast}\partial _{t}^{j} p\|_{2n-2j-1}\|\bar{\nabla}^{4N-1}G ^{3}\|_{1}\nonumber\\
&\lesssim (\sqrt{\mathcal{J}_{2N}}+\sqrt{\mathcal{D}_{2N}})\sqrt{\mathcal{E}_{N+2}(\mathcal{D}_{2N}+\mathcal{J}_{2N}+\mathcal{F}_{2N})}.
\end{align}
Inserting \eqref{i_de_4c}--\eqref{ide4e} into \eqref{p_u_e_1110c}, taking the sum over $\alpha$, and then manipulating the result as in the proof of Lemma \ref{Lem-L2} and applying Poincar\'e's inequality, we can show that
\begin{align}\label{E-STAR-DIFF-1}
&\ \frac{\mathrm{d}}{\mathrm{d}t} \bar{\mathcal{E}}^{\ast}_{2N} + \sum^{2N-1}_{j=0}\big( \|\nabla _{\ast}\partial _{t}^{j}\mathbf{u}\|_{1,4N-2j-1}^2+\|\nabla _{\ast}\partial^{j}_t\mathbf{w}\|_{0,4N-2j-1}^{2} \big)
\nonumber\\
 \lesssim &\ \sqrt{ \mathcal{E}_{N+2} }(\mathcal{D}_{2N}+\mathcal{J}_{2N}+\mathcal{F}_{2N}),
\end{align}
where the energy functional $\bar{\mathcal{E}}^{\ast}_{2N}$ satisfies \eqref{En*}. Moreover, applying $ \partial ^{\alpha} $ to $ \eqref{perturbc}_3 $ and then taking $L^2$ inner product of  the result with $\partial ^\alpha \partial _{t}\mathbf{w}$ with $ \alpha \in \mathbb{Z}_+^{1+2} $ ($ 1 \leq \vert \alpha\vert \leq 4N $), we can show by \eqref{p_G_e_001c} that
\begin{align*}
&\  \zeta \frac{\mathrm{d}}{\mathrm{d}t}\|\partial ^{\alpha}\mathbf{w}\|_{0}^{2}+\|\partial ^{\alpha}\partial _{t}\mathbf{w}\|_{0}^{2}
 \nonumber \\
 \leq &\ \frac{1}{2}\|\partial ^{\alpha}\partial _{t}\mathbf{w}\|_{0}^{2}+ \frac{\zeta^2}{2} \|\nabla \times \partial ^{\alpha}\mathbf{u}\|_{0}^{2}+C\sqrt{ \mathcal{E}_{N+2}} (\mathcal{D}_{2N}+\mathcal{J}_{2N}+\mathcal{F}_{2N}).
\end{align*}
This, alongside \eqref{E-STAR-DIFF-1} and the proof of Lemma \ref{Lem-L2}, implies the existence of an energy functional $ \bar{\mathcal{D}}_{2N}^{\ast} $ satisfying \eqref{Dn*} such that
\begin{align*}
\frac{\mathrm{d}}{\mathrm{d}t} \bar{\mathcal{E}}_{2N}^{\ast} + \bar{\mathcal{D}}_{2N}^{\ast}\lesssim  \sqrt{ \mathcal{E}_{N+2} }(\mathcal{D}_{2N}+\mathcal{J}_{2N}+\mathcal{F}_{2N}),
\end{align*}
which gives \eqref{p_u_e_00c}.

\vskip .1in
{\bf Case 2: $n=N+2$.} By \eqref{p_G_e_002c}, we have
\begin{align}
&\int_{\Omega}\partial^{\alpha}\mathbf{u} \cdot\partial^{\alpha}\mathbf{G}^1 \mathrm{d}\mathbf{x} + \int_{\Omega}\partial^{\alpha}\mathbf{w} \cdot\partial^{\alpha}\mathbf{G}^2 \mathrm{d}\mathbf{x}
 \nonumber \\
 \leq & \|\partial _{t}^{j}\mathbf{u}\|_{0,2n-2j} \|\bar{\nabla}^{2N+3}\mathbf{G}^1\|_{1} + \| \partial _{t}^{j} \mathbf{w}\|_{0,2n-2j}\| \bar{\nabla}^{2N+3}\mathbf{G}^2\|_{1} \notag\\
\lesssim&\sqrt{\mathcal{D}_{N+2}}\sqrt{\mathcal{E}_{2N}\mathcal{D}_{N+2}}.\label{n-2-ets-G1}
\end{align}
For $G^3$, when $1\leq|\alpha|\leq 2(N+2)-1$, it follows from \eqref{p_G_e_002c} that
\begin{align}
\int_{\Omega}\partial^{\alpha}p\,\partial^{\alpha}G^3 \mathrm{d}\mathbf{x} \leq\|\partial^{\alpha-\gamma}\partial^{\gamma}p\|_0\|\partial^{\alpha}G^3\|_0 &\leq \|\nabla_{*}\partial _{t}^{j}p\|_{2(N+2)-2j-2}\| \bar{\nabla}^{2N+3}G^3\|_{1}\nonumber\\
&\lesssim\sqrt{\mathcal{D}_{N+2}}\sqrt{\mathcal{E}_{2N}\mathcal{D}_{N+2}}.
\end{align}
When $|\alpha|=2(N+2)$, we may write $\alpha=\beta+\gamma+(\alpha-\beta-\gamma)$ for some $\beta$, $\gamma\in \mathbb{Z}_+ ^2$ with $|\beta|=|\gamma|=1$. Then we integrate by parts to get
\begin{align}\label{n+2+esti-G3}
\int_{\Omega}\partial^{\alpha}p\, \partial ^{\alpha} G^3 \mathrm{d}\mathbf{x} =-\int_{\Omega}\partial^{\alpha-\beta-\gamma}\partial^{\gamma}p\,\partial^{\alpha+\beta}G^3 \mathrm{d}\mathbf{x} &\leq \|\nabla_{*}\partial _{t}^{j}p\|_{2(N+2)-2j-2}\|\bar{\nabla}^{2N+3} G^3\|_{2}\nonumber\\
&\lesssim\sqrt{\mathcal{D}_{N+2}}\sqrt{\mathcal{E}_{2N}\mathcal{D}_{N+2}}.
\end{align}
Combining the estimates \eqref{p_u_e_1110c} and \eqref{n-2-ets-G1}--\eqref{n+2+esti-G3}, we can show that
\begin{align}\label{E-STAR-DIFF-2}
&\frac{\mathrm{d}}{\mathrm{d}t}\bar{\mathcal{E}}^{\ast}_{N+2} +  \sum^{N+1}_{j=0}\big( \|\nabla _{\ast}\partial _{t}^{j}\mathbf{u}\|_{1,2N-2j+3}^2+\|\nabla _{\ast}\partial^{j}_t\mathbf{w}\|_{0,2N-2j+3}^{2} \big)\lesssim \sqrt{ \mathcal{E}_{2N} }\mathcal{D}_{N+2},
\end{align}
where the energy functional satisfies \eqref{En*} and again the arguments in the the proof of Lemma \ref{Lem-L2} are applied. We also have for $ \alpha \in \mathbb{Z}_+ ^{1+2} $ with $ \vert \alpha\vert=2(N+2)$ that
\begin{align*}
 \zeta \frac{\mathrm{d}}{\mathrm{d}t}\|\partial ^{\alpha}\mathbf{w}\|_{0}^{2}+\|\partial ^{\alpha}\partial _{t}\mathbf{w}\|_{0}^{2}\leq \frac{1}{2}\|\partial ^{\alpha}\partial _{t}\mathbf{w}\|_{0}^{2}+\frac{\zeta^2}{2} \|\nabla \times \partial ^{\alpha}\mathbf{u}\|_{0}^{2}+C\sqrt{\mathcal{E}_{2N}}\mathcal{D}_{N+2}.
\end{align*}
This, alongside \eqref{E-STAR-DIFF-2}, implies the existence of an energy functional $ \bar{\mathcal{D}}_{N+2}^{\ast} $ satisfying \eqref{Dn*}, such that
\begin{align}
\frac{\mathrm{d}}{\mathrm{d}t} \bar{\mathcal{E}}_{N+2}^{\ast}+\bar{\mathcal{D}}_{N+2}^{\ast}\lesssim  \sqrt{ \mathcal{E}_{2N} }\mathcal{D}_{N+2},
\end{align}
which gives \eqref{p_u_e_00c1}. This completes the proof of the lemma.
 \end{proof}

\subsection{Horizontal dissipation of $ \boldsymbol{\eta} $}
\label{sub:estimates_on_horizontal_derivatives_of_gs} 

Lemma \ref{horizontal-uw} does not provide the dissipation estimates of $\boldsymbol{\eta}$ in the horizontal directions (see \eqref{Dn*}). In this subsection, we follow the spirit of \cite{Tan-Wang-SIMA} to recover such estimates. Note that since $\det(\mathcal{I}_3+\nabla\boldsymbol{\eta})=1$ (see \eqref{v1}), it follows from \cite{Tan-Wang-SIMA} that
\begin{equation}\label{Psi}
\operatorname{div}\boldsymbol{\eta}=\operatorname{div}\boldsymbol{\Psi}\ \ \mbox{with}\ \
\boldsymbol{\Psi}=\left(\begin{array}{c}\eta_1\partial_2\eta_2+\eta_1\partial_3\eta_3+\eta_1(\partial_2\eta_2\partial_3\eta_3-\partial_3\eta_2\partial_2\eta_3)
\\[1mm]
-\eta_1\partial_1\eta_2+\eta_2\partial_3\eta_3+\eta_1(\partial_3\eta_2\partial_1\eta_3-\partial_1\eta_2\partial_3\eta_3)
\\[1mm]
-\eta_1\partial_1\eta_3-\eta_2\partial_2\eta_3+\eta_1(\partial_1\eta_2\partial_2\eta_3-\partial_2\eta_2\partial_1\eta_3)\end{array}\right).
\end{equation}
With \eqref{Psi}, the proof of the following lemma is similar to that of Lemma \ref{lem-G-TILDA-G} (see also \cite{Tan-Wang-SIMA}).

\begin{lemma}\label{lem-G-TILD-G-technical}
Let $ (\boldsymbol{\eta},\mathbf{u}, p,\mathbf{w}) $ be the local solution to \eqref{reformulationc} satisfying \eqref{a-priori} for some sufficiently small constant $\delta$. Then it holds that
\begin{align}
\bullet\quad &\,\|\operatorname{div}\boldsymbol{\eta}\|_{4N}^2 \lesssim\mathcal{E}_{N+2}(\mathcal{D}_{2N}+\mathcal{J}_{2N}+\mathcal{F}_{2N}), \label{pgn1}\\
\bullet\quad &\,\|\operatorname{div}\boldsymbol{\eta}\|_{2(N+2)+1}^2 \lesssim\mathcal{E}_{2N}\mathcal{D}_{N+2}, \label{pgn2}\\
\bullet\quad &\, \|\boldsymbol\Psi\| _{0}^{2} \lesssim \mathcal{E}_{3}\mathcal{D}_{3}. \label{pgn3}
\end{align}
\end{lemma}

Then we have the following estimates that recover the horizontal dissipation of $\boldsymbol{\eta}$.

\begin{lemma}
Let $ (\boldsymbol{\eta},\mathbf{u}, p,\mathbf{w}) $ be the local solution to \eqref{reformulationc} satisfying \eqref{a-priori} for some sufficiently small constant $\delta$. Then the following estimates hold:
\begin{align}\label{p_u_e_00'132c}
\frac{\mathrm{d}}{\mathrm{d}t} \widetilde{\bar{\mathcal{E}}^\sharp_{2N}} +\bar{\mathcal{D}}_{2N}^\sharp \lesssim \sqrt{ \mathcal{E}_{N+2}} (\mathcal{D}_{2N}+\mathcal{J}_{2N}+\mathcal{F}_{2N}) +\bar{\mathcal{D}}_{2N}^{\ast},
\end{align}
and
\begin{align}\label{p_u_e_00'132c2}
\frac{\mathrm{d}}{\mathrm{d}t} \widetilde{\bar{\mathcal{E}}^\sharp_{N+2}} +\bar{\mathcal{D}}_{N+2}^\sharp \lesssim \sqrt{ \mathcal{E}_{2N}} \mathcal{D}_{N+2}+\bar{\mathcal{D}}_{N+2}^{\ast},
\end{align}
where the energy functionals satisfy
$$
\widetilde{\bar{\mathcal{E}}^\sharp_{n}} := \bar{\mathcal{E}}^\sharp_{n}+\sum_{\alpha\in \mathbb{Z}_+^2, \, |\alpha| \le 2n}\int_\Omega \big( 2\partial^\alpha \mathbf{u} \cdot \partial^\alpha \boldsymbol{\eta}+ \partial^{\alpha}\mathbf{w} \cdot(\nabla\times\partial^{\alpha}\boldsymbol{\eta})\big)\mathrm{d}\mathbf{x},
$$
and
\begin{equation}\label{Esharp}
\bar{\mathcal{E}}^\sharp_{n}(t) \sim \|\nabla  \boldsymbol{\eta}\|_{0,2n}^2 \qquad \text{and} \qquad \bar{\mathcal{D}}_n^\sharp(t) \sim  \big(\|\partial _{3}\boldsymbol{\eta}\|_{0,2n }^2
  +  \|\boldsymbol{\eta}\|_{0,2n}^2\big).
\end{equation}
\end{lemma}

\begin{proof}
{\bf Step 1.} Let $ n $ be $ 2N $ or $ N+2 $. Applying $ \partial ^{\alpha} $ to $ \eqref{perturbc}_2 $ with $ \alpha \in \mathbb{Z}_+^{2} $ and $ \vert \alpha\vert \leq 2n $, and then taking $L^2$ inner product of the resulting equation with $ \partial ^{\alpha}\boldsymbol{\eta} $, we get by virtue of $ \boldsymbol{\eta} \vert _{\partial \Omega}=0 $ and integration by parts that
\begin{align} \label{p_u_e_111'c}
&\,\int_\Omega \partial _{t} (\partial^\alpha \mathbf{u}) \cdot \partial^\alpha \boldsymbol{\eta}
   \mathrm{d}\mathbf{x}+ \frac{\mu+\zeta}{2}\frac{\mathrm{d}}{\mathrm{d}t} \|\nabla \partial ^{\alpha}\boldsymbol{\eta}\|_0^{2} +  \|\partial ^{\alpha}\partial _{3}\boldsymbol{\eta} \|_0^{2} \nonumber \\
=&\, \zeta \int _{\Omega}\partial ^{\alpha}\mathbf{w} \cdot (\nabla \times \partial ^{\alpha}\boldsymbol{\eta}) \mathrm{d}\mathbf{x}+ \int_\Omega   \partial^{\alpha} \boldsymbol{\eta}  \cdot  \partial ^{\alpha} \mathbf{G}^1 \mathrm{d}\mathbf{x} +\int_{\Omega}\partial^{\alpha}p\,\partial^{\alpha} \operatorname{div}\boldsymbol{\eta} \mathrm{d}\mathbf{x}.
\end{align}
Since $\partial_t\boldsymbol{\eta} = \mathbf{u}$, we rewrite the first term on the left side of \eqref{p_u_e_111'c}  as
\begin{align}\label{uetajiaochaguji}
\int_\Omega  \partial _{t} (\partial^\alpha \mathbf{u}) \cdot \partial^\alpha \boldsymbol{\eta} \mathrm{d}\mathbf{x}
=\frac{\mathrm{d}}{\mathrm{d}t}\int_\Omega \partial^\alpha \mathbf{u} \cdot \partial^\alpha \boldsymbol{\eta} \mathrm{d}\mathbf{x}- \|\partial ^{\alpha}\mathbf{u}\|_0^2.
\end{align}
Then we get from \eqref{p_u_e_111'c} that
\begin{align}\label{eta-horizon-ine}
&\, \frac{\mu+\zeta}{2}\frac{\mathrm{d}}{\mathrm{d}t} \|\nabla \partial ^{\alpha}\boldsymbol{\eta}\|_0^{2} + \frac{\mathrm{d}}{\mathrm{d}t}\int_\Omega \partial^\alpha \mathbf{u} \cdot \partial^\alpha \boldsymbol{\eta} \mathrm{d}\mathbf{x}
 +    \| \partial ^{\alpha}\partial _{3}\boldsymbol{\eta} \|_0^{2}
  \nonumber \\
  = &\, \zeta \int _{\Omega}\partial ^{\alpha}\mathbf{w} \cdot (\nabla \times \partial ^{\alpha}\boldsymbol{\eta}) \mathrm{d}\mathbf{x}+ \int_\Omega   \partial^{\alpha} \boldsymbol{\eta}  \cdot  \partial ^{\alpha} \mathbf{G}^1 \mathrm{d}\mathbf{x}+ \| \partial ^{\alpha}\mathbf{u}\|_0^{2} + \int_{\Omega}\partial^{\alpha}p\,\partial^{\alpha} \operatorname{div}\boldsymbol{\eta}  \mathrm{d}\mathbf{x}.
  \end{align}
Applying $\partial^{\alpha}$ to $ \eqref{perturbc}_{3} $ and taking $L^2$ inner product of the result with $\frac{1}{2}\nabla\times\partial^{\alpha}\boldsymbol{\eta}$, we obtain
\begin{align}\label{jiaochaguji*}
   \zeta\int _{\Omega}\partial ^{\alpha}\mathbf{w} \cdot(\nabla\times \partial ^{\alpha}\boldsymbol{\eta})\mathrm{d}\mathbf{x} &= \frac{\zeta}{4}\frac{\mathrm{d}}{\mathrm{d}t} \|\nabla\times\partial^{\alpha}\boldsymbol{\eta}\|_0^2-\frac{1}{2}\frac{\mathrm{d}}{\mathrm{d}t}\int _{\Omega}\partial ^{\alpha} \mathbf{w} \cdot(\nabla \times\partial ^{\alpha}\boldsymbol{\eta})\mathrm{d}\mathbf{x}\nonumber\\
  &\quad+\frac{1}{2}\int _{\Omega}\partial^{\alpha}\mathbf{w} \cdot(\nabla \times \partial ^{\alpha} \mathbf{u})  \mathrm{d}\mathbf{x}+\frac{1}{2}\int_{\Omega}\partial^{\alpha}\mathbf{G}^2\cdot(\nabla\times\partial^{\alpha}\boldsymbol{\eta})\mathrm{d}\mathbf{x}.
\end{align}
Replacing the first term on the right side of \eqref{eta-horizon-ine} by the right of \eqref{jiaochaguji*}, we have
\begin{align}\label{v2}
 &\, \frac{\mathrm{d}}{\mathrm{d}t} \Big(\frac{2\mu+\zeta}{2}\|\nabla \partial^{\alpha}\boldsymbol{\eta}\|_0^{2} + \frac{\zeta}{2} \| \operatorname{div} \partial ^{\alpha}\boldsymbol{\eta}\|_0^{2}  + 2\int_\Omega \partial^\alpha \mathbf{u} \cdot \partial^\alpha \boldsymbol{\eta} \mathrm{d}\mathbf{x} + \int _{\Omega}\partial ^{\alpha} \mathbf{w} \cdot(\nabla \times\partial ^{\alpha}\boldsymbol{\eta})\mathrm{d}\mathbf{x}\Big)
 +  2 \| \partial ^{\alpha}\partial _{3}\boldsymbol{\eta} \|_0^{2}
  \nonumber \\
  =&\, \int _{\Omega}\partial^{\alpha}\mathbf{w} \cdot(\nabla \times \partial ^{\alpha} \mathbf{u})\mathrm{d}\mathbf{x} + 2\|\partial ^{\alpha}\mathbf{u}\|_0 ^{2} + 2\int_\Omega   \partial^{\alpha} \boldsymbol{\eta}  \cdot  \partial ^{\alpha} \mathbf{G}^1 \mathrm{d}\mathbf{x} + \int_{\Omega}\partial^{\alpha}\mathbf{G}^2\cdot(\nabla\times\partial^{\alpha}\boldsymbol{\eta}) \mathrm{d}\mathbf{x}
   \nonumber \\
   &\, + 2\int_{\Omega}\partial^{\alpha}p\,\partial^{\alpha}\operatorname{div}\boldsymbol{\eta} \mathrm{d}\mathbf{x}=:\sum _{i=1}^{5}J _{i},
\end{align}
where we used the identity: $\|\nabla \partial ^{\alpha}\boldsymbol{\eta}\| _{0}^2=\|\nabla\times\partial ^{\alpha}\boldsymbol{\eta}\| _{0}^2+\|\operatorname{div} \partial ^{\alpha}\boldsymbol{\eta}\|_{0}^2$, due to $ \boldsymbol{\eta}=0 $ on $ \partial \Omega $. For the dominating terms, $J_1$ and $J_2$, recalling \eqref{Dn*}, we have
\begin{align}\label{v3}
 J_1 + J_2 \lesssim \|\mathbf{w}\|^2_{0,2n}+\|\mathbf{u}\|^2_{1,2n} \lesssim \bar{\mathcal{D}}^{*}_n.
 \end{align}

\vskip .1in
{\bf Step 2.} We consider the case $ n=2N $. When $|\alpha|=0$, it is clear from \eqref{p_dissipation_defc} and \eqref{p_G_e_001c} that
\begin{align}\label{v4}
J_3 \le 2\|\boldsymbol{\eta}\|_0\|\mathbf{G}^1\|_0 &\lesssim \sqrt{\mathcal{D}_{2}} \sqrt{\mathcal{E}_{3}(\mathcal{D}_{2}+\mathcal{J}_{2}+\mathcal{F}_{2})} \notag\\
&\lesssim \sqrt{\mathcal{E}_{N+2}}(\mathcal{D}_{2N}+\mathcal{J}_{2N}+\mathcal{F}_{2N}).
\end{align}
When $|\alpha|\neq0$, note that the highest order of derivative of $\mathbf{G}^1$ appearing in \eqref{p_G_e_001c} is $4N-1$. Since $|\alpha|\le 4N$, we cannot directly apply H\"older's inequality to bound $J_3$ when $|\alpha|=4N$. Instead, we integrate by parts (with respect to the horizontal variables) once and revoke  \eqref{FJ} to get
\begin{align}\label{v5}
J_3  \le  2\|\boldsymbol{\eta}\| _{0,4N+1} \|\mathbf{G}^1\|_{0,4N-1}
 &\lesssim \sqrt{ \mathcal{F}_{2N}}\sqrt{ \mathcal{E}_{N+2}(\mathcal{D}_{2N}+\mathcal{J}_{2N}+\mathcal{F}_{2N})} \notag\\
 &\lesssim \sqrt{\mathcal{E}_{N+2}}(\mathcal{D}_{2N}+\mathcal{J}_{2N}+\mathcal{F}_{2N}).
\end{align}
Since the estimate of $\mathbf{G}^2$ is up to the order of $4N$ (see \eqref{p_G_e_001c}), by H\"older's inequality and \eqref{FJ}, we can show that
\begin{align}\label{v6}
 J_4\leq \|\boldsymbol{\eta}\|_{1,4N} \|\mathbf{G}^2\|_{0,4N} &\lesssim \sqrt{\mathcal{F}_{2N} }\sqrt{\mathcal{E}_{N+2}(\mathcal{D}_{2N}+\mathcal{J}_{2N}+\mathcal{F}_{2N})}\notag\\
  &\lesssim \sqrt{\mathcal{E}_{N+2}}(\mathcal{D}_{2N}+\mathcal{J}_{2N}+\mathcal{F}_{2N}).
\end{align}
For $J_5$, when $ \vert \alpha\vert =0$, we note that none of the energy functionals in \eqref{p_energy_defc}--\eqref{FJ} contains the zeroth order derivative of $p$. To overcome the obstacle, we utilize \eqref{Psi} and \eqref{pgn3} to get
\begin{align}\label{v7}
J_5=2\int _{\Omega}p\, \operatorname{div}\boldsymbol{\eta} \mathrm{d}\mathbf{x} &=2\int _{\Omega}p\, \operatorname{div}\boldsymbol{\Psi} \mathrm{d}\mathbf{x} =- 2\int _{\Omega}\nabla p \cdot \boldsymbol{\Psi}  \mathrm{d}\mathbf{x} \notag\\
&\leq 2\|\nabla p\|_{0}\|\boldsymbol{\Psi}\|_{0} \lesssim \sqrt{\mathcal{D}_{2}}\sqrt{\mathcal{E}_{3}\mathcal{D}_{3}} \lesssim \sqrt{\mathcal{E}_{N+2}}(\mathcal{D}_{2N}+\mathcal{J}_{2N}+\mathcal{F}_{2N}).
\end{align}
When $ \left\vert \alpha\right\vert \neq 0 $, by using \eqref{pgn1} and \eqref{FJ}, we have
\begin{align}\label{v8}
J_5=2\int_{\Omega}\partial^{\alpha}p\,\partial^{\alpha} \operatorname{div}\boldsymbol{\eta} \mathrm{d}\mathbf{x}\leq 2\|\nabla_*p\|_{4N-1}\|\operatorname{div}\boldsymbol{\eta}\|_{4N} &\lesssim\sqrt{\mathcal{J}_{2N}}\sqrt{\mathcal{E}_{N+2}(\mathcal{D}_{2N}+\mathcal{J}_{2N}+\mathcal{F}_{2N})}\notag\\
&\lesssim \sqrt{\mathcal{E}_{N+2}}(\mathcal{D}_{2N}+\mathcal{J}_{2N}+\mathcal{F}_{2N}).
\end{align}
The combination of \eqref{v3}--\eqref{v8} gives us
\begin{align}\label{v9}
J_1+J_2+J_3+J_4+J_5\lesssim \sqrt{\mathcal{E}_{N+2}}(\mathcal{D}_{2N}+\mathcal{J}_{2N}+\mathcal{F}_{2N}) + \bar{\mathcal{D}}^{*}_n.
\end{align}
Substituting \eqref{v9} into \eqref{v2}, taking the summation over $\alpha$ with $0\le |\alpha|\le 4N$ and noticing $\boldsymbol{\eta}$ satisfies the Poincar\'e inequality: $\|\partial^\alpha \boldsymbol{\eta}\|_0 \lesssim \|\partial^\alpha\partial_3\boldsymbol{\eta}\|_0$, we arrive at \eqref{p_u_e_00'132c}.

\vskip .1in
{\bf Step 3.} When $ n=N+2 $, we use \eqref{p_dissipation_defc} and \eqref{p_G_e_002c} to get
$$
J_3 \le \left\{ \begin{aligned}
&2\|\boldsymbol{\eta}\|_0\|\mathbf{G}^1\|_0 \lesssim \sqrt{\mathcal{D}_{2}} \sqrt{\mathcal{E}_{2}\mathcal{D}_{3}} \lesssim \sqrt{\mathcal{E}_{2N}} \mathcal{D}_{N+2},  &|\alpha|&=0,\\
&2 \|\boldsymbol{\eta}\|_{2(N+2)} \|\mathbf{G}^1\|_{2(N+2)}\lesssim \sqrt{ \mathcal{E}_{2N}} \mathcal{D}_{N+2}, &|\alpha|&\neq 0.
\end{aligned}
\right.
$$
It follows from integration by parts and \eqref{p_G_e_002c} that
\begin{align*}
 J_4=\int_{\Omega}\partial^{\alpha}\boldsymbol{\eta}\cdot\nabla\times\partial^{\alpha}\mathbf{G}^2 \mathrm{d}\mathbf{x}\leq \|\boldsymbol{\eta}\|_{2(N+2)} \|\mathbf{G}^2\|_{2(N+2)+1} \lesssim \sqrt{ \mathcal{E}_{2N}} \mathcal{D}_{N+2}.
\end{align*}
Moreover, by \eqref{v7}, integration by parts, \eqref{pgn2} and \eqref{p_dissipation_defc}, we have
\begin{align*}
J_5 \leq \left\{ \begin{aligned}
&2\|\nabla p\|_{0}\|\boldsymbol{\Psi}\|_{0} \lesssim \sqrt{\mathcal{D}_{2}}\sqrt{\mathcal{E}_{3}\mathcal{D}_{3}} \lesssim \sqrt{\mathcal{E}_{2N}} \mathcal{D}_{N+2}, &|\alpha|&=0,\\
&2\|\nabla_*p\|_{2(N+2)-2}\|\operatorname{div}\boldsymbol{\eta}\|_{2(N+2)+1}\lesssim \sqrt{\mathcal{E}_{2N}} \mathcal{D}_{N+2}, &|\alpha|&\neq 0.
\end{aligned}
\right.
\end{align*}
These estimates, alongside \eqref{v2} and \eqref{v3}, give rise to \eqref{p_u_e_00'132c2}. The proof is complete.
\end{proof}


\section{Estimates on Vertical Derivatives}
\label{sub:estimates_on_the_vertical_derivatives}

After establishing the temporal and horizontal estimates, we now deal with the evolution of the vertical derivatives of the solution by the ODE and elliptic structures of \eqref{perturbc}. We stress that this section contains the most significant technical differences than the previous works \cite{Feng-hong-zhu-compre,Tan-Wang-SIMA}.  The major difference lies in obtaining the estimates of the angular velocity field. Due to the deep coupling between the fluid and angular velocity fields, accessing such estimates is highly non-trivial. To achieve the goal, we derive a set of auxiliary equations for various components of the solution (see \eqref{first-two-compo-ODE}, \eqref{phi-3-eq}, \eqref{div-eq-w}, \eqref{u-tuoyuan}), apply the key estimates in Section 3, and carry out extensive combinations of individual energy estimates. The combined estimates are finally assembled by induction to produce the desired energetic quantities necessary for closing the entire energy scheme in the next section.

\vskip .1in
We begin with applying $\nabla\times$ to $\eqref{perturbc}_3$ to get
\begin{align}\label{xuanduw}
  \mathbf{W}_{t} + 2 \zeta \mathbf{W} =  \zeta \nabla \mathop{\mathrm{div}}\nolimits \mathbf{u} -  \zeta \Delta \mathbf{u}+ \boldsymbol{\mathcal{G}},
  \end{align}
  where $ \mathbf{W}:=\nabla \times \mathbf{w} $ and $ \boldsymbol{\mathcal{G}}:=(\mathcal{G}_{1}, \mathcal{G}_{2}, \mathcal{G}_{3})=\nabla \times \mathbf{G}^{2} $, where $\mathbf{G}^2$ is the same as in \eqref{perturbc}. Define the projections:
\begin{align}\label{x5}
\boldsymbol{\mathcal{Q}}:=\mathcal{P} (\partial _{3}^{2}\boldsymbol\eta) \qquad \text{and} \qquad \boldsymbol{\mathcal{V}}:=\mathcal{P} \mathbf{W},
\end{align}
where $\mathcal{P} $ is the Helmholtz projection operator defined by \eqref{defi-projection}. Denote $ \boldsymbol{\mathcal{Q}}=(\boldsymbol{\mathcal{Q}} _{\ast}, \mathcal{Q}_{3}) $ and $ \boldsymbol{\mathcal{V}}=(\boldsymbol{\mathcal{V}}_{\ast}, \mathcal{V}_{3}) $. Notice that $ \mathcal{P} (\nabla f)=0 $ and $ \mathcal{P} (\partial _{h} \mathbf{f})=\partial _{h}\mathcal{P} \mathbf{f} $ for $ \partial _{h}=\partial _{1},\partial _{2} $ or $ \partial _{t} $, and for any smooth functions $f$ and $\mathbf{f}$. Applying the operator $\mathcal{P} $ to the equations $\eqref{perturbc}_2$ and \eqref{xuanduw}, we get
   \begin{subequations}\label{first-two-compo-ODE}
 \begin{alignat}{2}
  (\mu+\zeta)\partial _{t}\boldsymbol{\mathcal{Q}}+ \boldsymbol{\mathcal{Q}}+ \zeta\boldsymbol{\mathcal{V}}&=\boldsymbol{\mathcal{H}},\label{pa-3-2eta-eq}
 \\
  \partial _{t}\boldsymbol{\mathcal{V}} +2 \zeta \boldsymbol{\mathcal{V}} + \zeta\partial _{t}\boldsymbol{\mathcal{Q}}&=\tilde{\boldsymbol{\mathcal{G}}},\label{eq-w-star}
 \end{alignat}
 \end{subequations}
 where the quantities on the right hand sides are given by
 \begin{align}
 \boldsymbol{\mathcal{H}}&=\partial_t\mathcal{P} \mathbf{u} -(\mu+\zeta)\Delta_{*}\mathcal{P} \mathbf{u} - \mathcal{P} \mathbf{G}^1, \label{HG1}\\
 \tilde{\boldsymbol{\mathcal{G}}}&=-  \zeta \Delta _{\ast}\mathcal{P} \mathbf{u} + \mathcal{P} \boldsymbol{\mathcal{G}}.\label{HG2}
 \end{align}
The following lemma provides the estimates of the vertical derivatives of $ \boldsymbol{\eta} $, $ \nabla \times \mathbf{w} $ and $ \mathop{\mathrm{div}}\nolimits \mathbf{w} $.

 \begin{lemma}\label{lem-eta-w}
Let $ (\boldsymbol{\eta},\mathbf{u}, p,\mathbf{w}) $ be the local solution to \eqref{reformulationc} satisfying \eqref{a-priori} for some sufficiently small constant $\delta$. There exists an energy functional
 \begin{align}\label{y1}
 \mathfrak{E}_{n}(t) \sim\big( \|\boldsymbol{\mathcal{Q}}\|_{2n-1}^2+\|\boldsymbol{\mathcal{V}}\|^2_{2n-1}+\|\partial_3\phi\|^2_{2n-1}
+ \|\operatorname{div} \mathbf{w}\|_{2n-1}^{2}+\|\boldsymbol\eta\|^2_{2n}\big),
 \end{align}
 where $\phi$ denotes the solution to \eqref{fip-elliptic}, such that for any integer $ N \geq 4 $,
 \begin{align}
 \bullet \ \ &\,\frac{\mathrm{d}}{\mathrm{d}t}\mathfrak{E}_{2N}+
\sum^{1}_{l= 0}\big(\|\partial^l_{t}\boldsymbol{\mathcal{Q}}\|^2_{4N-1}+\|\partial^l_{t} \mathbf{W}\|^2_{4N-1}+\|\partial_t^{l}\mathbf{w}\|^2_{4N}\big)+\|\boldsymbol{\eta}\|^2_{4N}+\|\mathbf{u}\|^2_{4N+1}+\|\nabla p\|^2_{4N-2}\nonumber\\
       \lesssim &\, \|\partial _{t}\mathbf{u}\|_{4N-1}^2 + \bar{\mathcal{D}}_{2N}^\ast+\bar{\mathcal{D}}_{2N}^\sharp + { \mathcal{E}_{N+2}  }(\mathcal{D}_{2N} +  \mathcal{J}_{2N} +\mathcal{F}_{2N}) , \label{con-lem-eta-w}
      \\
        \bullet \ \ &\,\frac{\mathrm{d}}{\mathrm{d}t}\mathfrak{E}_{N+2}+
\sum^{1}_{l= 0}\big(\|\partial^l_{t}\boldsymbol{\mathcal{Q}}\|^2_{2N+3}+\|\partial^l_{t}\mathbf{W}\|^2_{2N+3}+\|\partial_t^{l}\mathbf{w}\|^2_{2N+4}\big)+\|\boldsymbol\eta\|^2_{2N+4}+\|\mathbf{u}\|^2_{2N+5}+\|\nabla p\|^2_{2N+2}\nonumber\\
       \lesssim &\,  \|\partial _{t}\mathbf{u}\|_{2N+3}^2 +\bar{\mathcal{D}}_{N+2}^\ast+\bar{\mathcal{D}}_{N+2}^\sharp + { \mathcal{E}_{2N}  }\mathcal{D}_{N+2}.\label{con-lem-eta-w1}
 \end{align}
 \end{lemma}

\begin{proof}
Let $n=2N$ or $n=N+2$. For any $ 0 \leq k \leq 2n-1 $, let $ \alpha \in \mathbb{Z}_+^{2} $ and $ \beta \in \mathbb{Z}_+^{3} $ be such that $ \left\vert \beta\right\vert \leq k$ and $ \left\vert \alpha\right\vert \leq 2n-1-k $.

\vskip .1in
{\bf Step 1.} Applying $\partial^{\beta}\partial^{\alpha}$ to \eqref{pa-3-2eta-eq} and  testing against $ \partial^{\beta}\partial^{\alpha} \partial _{t}\boldsymbol{\mathcal{Q}}$, we have
  \begin{align}\label{3.91}
   &\,\frac{1 }{2}\frac{\mathrm{d}}{\mathrm{d}t} \|  \partial^{\beta} \partial^{\alpha}\boldsymbol{\mathcal{Q}} \| _{0} ^{2}+(\mu+\zeta)  \| \partial^{\beta} \partial^{\alpha}\partial _{t}\boldsymbol{\mathcal{Q}} \| _{0}^{2} +  \zeta \int _{\Omega}  \partial^{\beta} \partial^{\alpha}\partial _{t}\boldsymbol{\mathcal{Q}} \cdot \partial^{\beta} \partial^{\alpha}\boldsymbol{\mathcal{V}} \mathrm{d}\mathbf{x}
  \nonumber\\
  =&\, \int_{\Omega} \partial^{\beta} \partial^{\alpha}\boldsymbol{\mathcal{H}}\cdot  \partial^{\beta} \partial^{\alpha}\partial _{t}\boldsymbol{\mathcal{Q}} \mathrm{d}\mathbf{x}.
  \end{align}
  Next, we apply $\partial^{\beta}\partial^{\alpha}$ to  \eqref{eq-w-star} and test  against $\partial^{\beta}\partial^{\alpha} \boldsymbol{\mathcal{V}} $ to get
  \begin{align}\label{W*guji}
 \frac{1}{2}\frac{\mathrm{d}}{\mathrm{d}t} \| \partial^{\beta} \partial^{\alpha}\boldsymbol{\mathcal{V}} \| _{0}^{2}+2 \zeta \| \partial^{\beta} \partial^{\alpha} \boldsymbol{\mathcal{V}} \| _{0}^{2} +  \zeta\int _{\Omega} \partial^{\beta} \partial^{\alpha}\partial _{t}\boldsymbol{\mathcal{Q}} \cdot  \partial^{\beta} \partial^{\alpha}\boldsymbol{\mathcal{V}} \mathrm{d}\mathbf{x} = \int _{\Omega} \partial^{\beta} \partial^{\alpha}\tilde{\boldsymbol{\mathcal{G}}}\cdot  \partial^{\beta} \partial^{\alpha}\boldsymbol{\mathcal{V}} \mathrm{d}\mathbf{x}.
  \end{align}
  Combining \eqref{3.91}, \eqref{W*guji}, we obtain
    \begin{align*}\displaystyle
  &\,\frac{1}{2}\frac{\mathrm{d}}{\mathrm{d}t}  \big(  \| \partial^{\beta} \partial^{\alpha} \boldsymbol{\mathcal{Q}}\|_{0} ^{2} +\|\partial^{\beta} \partial^{\alpha}\boldsymbol{\mathcal{V}} \|_{0}^{2} \big) +\mu \| \partial^{\beta} \partial^{\alpha}\partial _{t}\boldsymbol{\mathcal{Q}} \|_{0}^{2} + \zeta\|\partial^{\beta} \partial^{\alpha}\boldsymbol{\mathcal{V}}\|_{0}^2\\
  &\,\quad+\zeta  \int _{\Omega}\left( \vert \partial ^{\beta}\partial ^{\alpha}\partial _{t}\boldsymbol{\mathcal{Q}}\vert ^{2}+2 \partial^{\beta} \partial^{\alpha}\partial _{t}\boldsymbol{\mathcal{Q}} \cdot \partial^{\beta} \partial^{\alpha}\boldsymbol{\mathcal{V}}+\vert  \partial^{\beta} \partial^{\alpha}\boldsymbol{\mathcal{V}}\vert ^{2}\right) \mathrm{d}\mathbf{x}
   \nonumber \\
   \le &\,\int _{\Omega} \big( \partial^{\beta} \partial^{\alpha}\boldsymbol{\mathcal{H}} \cdot  \partial^{\beta} \partial^{\alpha}\partial _{t}\boldsymbol{\mathcal{Q}} + \partial^{\beta} \partial^{\alpha}\tilde{\boldsymbol{\mathcal{G}}} \cdot \partial^{\beta} \partial^{\alpha}\boldsymbol{\mathcal{V}} \big) \mathrm{d}\mathbf{x},
  \end{align*}
  that is,
   \begin{align*}
  &\,\frac{1}{2}\frac{\mathrm{d}}{\mathrm{d}t}  \big(  \| \partial^{\beta} \partial^{\alpha} \boldsymbol{\mathcal{Q}}\|_{0} ^{2} +\|\partial^{\beta} \partial^{\alpha}\boldsymbol{\mathcal{V}} \|_{0}^{2} \big) +\mu \| \partial^{\beta} \partial^{\alpha}\partial _{t}\boldsymbol{\mathcal{Q}} \|_{0}^{2} + \zeta\|\partial^{\beta} \partial^{\alpha}\boldsymbol{\mathcal{V}}\|_{0}^2
   \nonumber \\
   &\,\quad+\zeta \|\partial ^{\beta}\partial ^{\alpha}\boldsymbol{\mathcal {V}}+\partial ^{\beta}\partial ^{\alpha}\partial _{t}\boldsymbol{\mathcal{Q}}\|_{0}^{2}
   \nonumber \\
   \le &\,\int _{\Omega} \big( \partial^{\beta} \partial^{\alpha}\boldsymbol{\mathcal{H}} \cdot  \partial^{\beta} \partial^{\alpha}\partial _{t}\boldsymbol{\mathcal{Q}} + \partial^{\beta} \partial^{\alpha}\tilde{\boldsymbol{\mathcal{G}}} \cdot \partial^{\beta} \partial^{\alpha}\boldsymbol{\mathcal{V}} \big) \mathrm{d}\mathbf{x}.
  \end{align*}
  Then we get
  
    \begin{align}\label{w*chongjifen}
  &\,\frac{1}{2}\frac{\mathrm{d}}{\mathrm{d}t}  \big(  \| \partial^{\beta} \partial^{\alpha} \boldsymbol{\mathcal{Q}}\|_{0} ^{2} +\|\partial^{\beta} \partial^{\alpha}\boldsymbol{\mathcal{V}} \|_{0}^{2} \big) +\mu \| \partial^{\beta} \partial^{\alpha}\partial _{t}\boldsymbol{\mathcal{Q}} \|_{0}^{2} + \zeta\|\partial^{\beta} \partial^{\alpha}\boldsymbol{\mathcal{V}}\|_{0}^2
   \nonumber \\
   \le &\,\int _{\Omega} \big( \partial^{\beta} \partial^{\alpha}\boldsymbol{\mathcal{H}} \cdot  \partial^{\beta} \partial^{\alpha}\partial _{t}\boldsymbol{\mathcal{Q}} + \partial^{\beta} \partial^{\alpha}\tilde{\boldsymbol{\mathcal{G}}} \cdot \partial^{\beta} \partial^{\alpha}\boldsymbol{\mathcal{V}} \big) \mathrm{d}\mathbf{x}.
  \end{align}
  where the non-negative term $\zeta\| \partial^{\beta} \partial^{\alpha}\boldsymbol{\mathcal{V}}+ \partial^{\beta} \partial^{\alpha} \partial _{t}\boldsymbol{\mathcal{Q}}\|_{0}^2$ was dropped. Applying Cauchy's inequality, we can update \eqref{w*chongjifen} as
   \begin{align}\label{w*chongjifen-update}
   \frac{\mathrm{d}}{\mathrm{d}t}  \big( \| \partial^{\beta} \partial^{\alpha} \boldsymbol{\mathcal{Q}}\|_{0} ^{2} + \| \partial^{\beta} \partial^{\alpha}\boldsymbol{\mathcal{V}} \|_{0}^{2}\big) + \mu\| \partial^{\beta} \partial^{\alpha}\partial _{t}\boldsymbol{\mathcal{Q}} \|_{0}^{2}+ \zeta \| \partial^{\beta} \partial^{\alpha}\boldsymbol{\mathcal{V}}\| _{0}^{2} \lesssim \| \partial^{\beta} \partial^{\alpha}\boldsymbol{\mathcal{H}}\|_{0}^2+\| \partial^{\beta} \partial^{\alpha}\tilde{\boldsymbol{\mathcal{G}}}\|_{0}^2.
    \end{align}

\vskip .1in
{\bf Step 2.} Applying $\partial^{\beta}\partial^{\alpha}$ to \eqref{pa-3-2eta-eq} and \eqref{eq-w-star}, and then testing against $ \partial^{\beta}\partial^{\alpha}\boldsymbol{\mathcal{Q}} $ and $\partial^{\beta}\partial^{\alpha}\partial _{t}\boldsymbol{\mathcal{V}}$, respectively, and applying Cauchy's inequality, we obtain
\begin{align*}
&\,\frac{\mu+\zeta}{2}\frac{\mathrm{d}}{\mathrm{d}t}\| \partial^{\beta} \partial^{\alpha}\boldsymbol{\mathcal{Q}}\| _{0}^2+ \| \partial^{\beta} \partial^{\alpha}\boldsymbol{\mathcal{Q}}\| _{0}^2=\int_{\Omega} \big(\partial^\beta \partial^{\alpha}\boldsymbol{\mathcal{H}} - \zeta\partial^{\beta} \partial^{\alpha} \boldsymbol{\mathcal{V}}\big) \cdot \partial^\beta \partial^{\alpha}\boldsymbol{\mathcal{Q}} \mathrm{d}\mathbf{x} \nonumber\\
\leq &\, \frac{1 }{2}\| \partial^{\beta} \partial^{\alpha}\boldsymbol{\mathcal{Q}}\| _{0}^2+\zeta^2\| \partial^{\beta} \partial^{\alpha}\boldsymbol{\mathcal{V}}\| _{0}^2
+\| \partial^{\beta} \partial^{\alpha}\boldsymbol{\mathcal{H}}\| _{0}^2
\end{align*}
and
\begin{align*}
&\,  \zeta \frac{\mathrm{d}}{\mathrm{d}t} \|\partial ^{\beta}\partial ^{\alpha}\boldsymbol{\mathcal{V}}\|_{0}^{2}+\|\partial ^{\beta}\partial ^{\alpha}\partial _{t}\boldsymbol{\mathcal{V}} \|_{0}^{2}=\int _{\Omega} \big( \partial^{\beta} \partial^{\alpha}\tilde{\boldsymbol{\mathcal{G}}}- \zeta\partial^{\beta} \partial^{\alpha}\partial _{t}\boldsymbol{\mathcal{Q}} \big) \cdot  \partial^{\beta} \partial^{\alpha}\partial _{t}\boldsymbol{\mathcal{V}} \mathrm{d}\mathbf{x}
\nonumber \\
 \leq &\,\frac{1}{2}\|\partial ^{\beta}\partial ^{\alpha}\partial _{t}\boldsymbol{\mathcal{V}}\|_{0}^{2}+ \frac{\zeta^2}{2}\|\partial ^{\beta}\partial ^{\alpha}\partial _{t}\boldsymbol{\mathcal{Q}}\|_{0}^{2}+\frac{1}{2}\|\partial ^{\beta}\partial ^{\alpha}\tilde{\boldsymbol{\mathcal{G}}}\|_{0}^{2}.
\end{align*}
After rearranging terms and combining the results, we have
\begin{align}\label{coueta}
&\,\frac{\mathrm{d}}{\mathrm{d}t}\big((\mu+\zeta)\| \partial^{\beta} \partial^{\alpha}\boldsymbol{\mathcal{Q}}\| _{0}^2 + 2\zeta \|\partial ^{\beta}\partial ^{\alpha}\boldsymbol{\mathcal{V}}\|_{0}^{2}\big) +  \| \partial^{\beta} \partial^{\alpha}\boldsymbol{\mathcal{Q}}\| _{0}^2 + \|\partial ^{\beta}\partial ^{\alpha}\partial _{t}\boldsymbol{\mathcal{V}} \|_{0}^{2} \notag\\
\lesssim &\, \|\partial ^{\beta}\partial ^{\alpha}\partial _{t}\boldsymbol{\mathcal{Q}}\|_{0}^{2} + \| \partial^{\beta} \partial^{\alpha}\boldsymbol{\mathcal{V}}\| _{0}^2 + \| \partial^{\beta} \partial^{\alpha}\boldsymbol{\mathcal{H}}\| _{0}^2 + \|\partial ^{\beta}\partial ^{\alpha}\tilde{\boldsymbol{\mathcal{G}}}\|_{0}^{2}.
\end{align}
By properly combining \eqref{w*chongjifen-update}--\eqref{coueta}, there exist two energy functionals $\mathcal{M}_0^1$ and $\mathcal{M}_0^2$, such that
\begin{equation}\label{M1}
\begin{aligned}
\mathcal{M}_0^1 & \sim \big(\| \partial^{\beta} \partial^{\alpha}\boldsymbol{\mathcal{Q}} \|_{0}^{2}+\| \partial^{\beta} \partial^{\alpha}\boldsymbol{\mathcal{V}}\|_{0} ^{2}\big),\\
\mathcal{M}_0^2 & \sim \big(\| \partial^{\beta} \partial^{\alpha}\partial _{t}\boldsymbol{\mathcal{Q}} \|_{0}^{2} + \| \partial^{\beta} \partial^{\alpha}\boldsymbol{\mathcal{Q}}\| _{0}^2 + \|\partial ^{\beta}\partial ^{\alpha}\partial _{t}\boldsymbol{\mathcal{V}} \|_{0}^{2} +  \| \partial^{\beta} \partial^{\alpha}\boldsymbol{\mathcal{V}}\| _{0}^{2}\big),
\end{aligned}
\end{equation}
and
\begin{align}\label{w-star-eta-star-combined}
\frac{\mathrm{d}}{\mathrm{d}t}  \mathcal{M}_0^1 +\mathcal{M}_0^2 \lesssim \| \partial^{\beta} \partial^{\alpha}\boldsymbol{\mathcal{H}}\|_{0}^2+\| \partial^{\beta} \partial^{\alpha}\tilde{\boldsymbol{\mathcal{G}}}\|_{0}^2.
\end{align}
Recall that $ \alpha \in \mathbb{Z}_+^{2} $ and $ \beta \in \mathbb{Z}_+^{3} $ with $ \vert \alpha\vert \leq 2n-1-k  $ and $ \vert \beta\vert \leq k $. After summing up such $ \alpha $ and $ \beta $, we get from \eqref{w-star-eta-star-combined} that
\begin{align}\label{M2}
\frac{\mathrm{d}}{\mathrm{d}t} \mathcal{M}_{k,2n-k-1}^1 + \mathcal{M}_{k,2n-k-1}^2     \lesssim \|\boldsymbol{\mathcal{H}}\|^2_{k,2n-k-1}+\|\tilde{\boldsymbol{\mathcal{G}}}\|^2_{k,2n-k-1},
\end{align}
    where
\begin{equation}\label{M3}
\begin{aligned}
\mathcal{M}_{k,2n-k-1}^1 &\sim \big(\|\boldsymbol{\mathcal{V}}\|^2_{k,2n-k-1}+\|\boldsymbol{\mathcal{Q}}\|^2_{k,2n-k-1}\big),\\
\mathcal{M}_{k,2n-k-1}^2 &\sim \big(\|\partial_t\boldsymbol{\mathcal{Q}}\|^2_{k,2n-k-1} + \|\boldsymbol{\mathcal{Q}}\|^2_{k,2n-k-1} + \|\partial_t\boldsymbol{\mathcal{V}}\|^2_{k,2n-k-1} + \|\boldsymbol{\mathcal{V}}\|^2_{k,2n-k-1}\big)
\end{aligned}
\end{equation}
for $ 0 \leq k \leq 2n-1 $. Recalling \eqref{HG1}, \eqref{HG2} and Lemma \ref{lem-anisotropi-1}, we have
\begin{align*}
\|\boldsymbol{\mathcal{H}}\|^2_{k,2n-k-1}+\|\tilde{\boldsymbol{\mathcal{G}}}\|^2_{k,2n-k-1}
\lesssim \|\partial_t \mathbf{u}\|^2_{k,2n-k-1}+\|\mathbf{u}\|^2_{k+1,2n-k}+\|\mathbf{G}^1\|^2_{k,2n-k-1}+\|\mathbf{G}^2\|^2_{k+1,2n-k-1},
\end{align*}
by which we update \eqref{M2} as
\begin{align}\label{w-star-eta-star-combined-sumup}
&\,\frac{\mathrm{d}}{\mathrm{d}t} \mathcal{M}_{k,2n-k-1}^1 + \mathcal{M}_{k,2n-k-1}^2 \notag\\
\lesssim &\,\|\partial_t \mathbf{u}\|^2_{k,2n-k-1}+\|\mathbf{u}\|^2_{k+1,2n-k}+\|\mathbf{G}^1\|^2_{k,2n-k-1}+\|\mathbf{G}^2\|^2_{k+1,2n-k-1}.
\end{align}

\vskip .1in
{\bf Step 3.} Notice that $W_3 = (\mathcal{P} \mathbf{W})_3 +\partial _{3}\phi = \mathcal{V}_3+\partial_3\phi$ (see \eqref{x4}, \eqref{x5}). This, along with \eqref{xuanduw} and \eqref{eq-w-star}, implies that $\partial_3\phi$ solves
 \begin{align}\label{phi-3-eq}
(\partial _{3}\phi)_{t}+2\zeta \partial _{3}\phi= \zeta\partial_3 \mathop{\mathrm{div}} _{\ast}\nolimits \mathbf{u} _{\ast} -  \zeta \Delta _{\ast}u _{3}+\mathcal{G}_{3} + \zeta \partial _{t}\mathcal{Q}_{3}-(- \zeta \Delta _{\ast}\mathcal{P} \mathbf{u} + \mathcal{P} \boldsymbol{\mathcal{G}})_{3},
\end{align}
where $\mathbf{u} _{\ast} = (u_1,u_2)$, from which we can derive that
\begin{align}\label{phi*}
\frac{\mathrm{d}}{\mathrm{d}t}\|\partial _{3}\phi\|_{k,2n-k-1}^{2}+ \sum _{l=0}^{1}\|\partial_t^{l}\partial _{3}\phi\|_{k,2n-k-1}^{2} \lesssim \|\partial _{t}\boldsymbol{\mathcal{Q}}\|_{k,2n-k-1}^{2}+\|\mathbf{u}\|^2_{k+1,2n-k}+\|\boldsymbol{\mathcal{G}}\|^2_{k,2n-k-1}.
\end{align}
For $ \operatorname{div} \mathbf{w} $, from $\eqref{perturbc}_3$ we get
\begin{align}\label{div-eq-w}
\partial_t\operatorname{div} \mathbf{w} + 2 \zeta \operatorname{div} \mathbf{w}  = \operatorname{div} \mathbf{G}^2.
\end{align}
It thus follows from a direct calculation that
\begin{align}\label{divw}
&  \frac{\mathrm{d}}{\mathrm{d}t}\|\operatorname{div}\mathbf{w} \|_{k,2n-k-1}^{2}+ \sum _{l=0}^{1}\|\partial^l_t\operatorname{div}\mathbf{w} \|_{k,2n-k-1}^{2}
 \lesssim \|\operatorname{div}\mathbf{G}^2\|^2_{k,2n-k-1}  ,
\end{align}
where $ 0 \leq k \leq 2n-1 $. Moreover, by taking $ \lambda=2n-k-1 $ in \eqref{phi2} and $ j=1 $, $ \lambda= 2n-k-1$ in \eqref{phi2-pa-t}, we get
\begin{align}\label{key-fip-esti1}
\|\nabla _{h}\phi\|_{k,2n-k-1}^{2} \lesssim \|\partial _{3}\phi\|_{k,2n-k-1}^{2}+\sigma \|\mathbf{W} \|_{0,2n-1}^{2}+  \|\mathbf{w} \|_{0}\|\mathbf{w} \|_{1}
\end{align}
and
\begin{align}\label{key-fip-esti1-pat1}
 \|\nabla _{h}\partial _{t}\phi\|_{k,2n-k-1}^{2} \lesssim \|\partial _{3}\partial _{t}\phi\|_{k,2n-k-1}^{2}+\sigma \|\partial _{t} \mathbf{W} \|_{0,2n-1}^{2}+  \| \partial _{t} \mathbf{w} \|_{0}\| \partial _{t} \mathbf{w} \|_{1}.
\end{align}
for any $ \sigma>0 $.

\vskip .1in
{\bf Step 4.}
With the estimates \eqref{w-star-eta-star-combined-sumup}, \eqref{phi*}, \eqref{divw}, \eqref{key-fip-esti1} and \eqref{key-fip-esti1-pat1} at our disposal, we can show that, after proper combinations, there exists an energy functional $ \mathfrak{E}_{k,2n-k-1} $ such that
\begin{align}\label{E-kn-defi}
 \mathfrak{E}_{k,2n-k-1} \ \sim \ \|\boldsymbol{\mathcal{Q}}\|_{k,2n-k-1}^{2}+\|\boldsymbol{\mathcal{V}}\|_{k,2n-k-1} ^{2}+\|\partial _{3}\phi\|_{k,2n-k-1}^{2}+\|\operatorname{div} \mathbf{w}\|_{k,2n-k-1}^{2}
\end{align}
and
\begin{align}\label{QV-first-sumup}
 &\ \frac{\mathrm{d}}{\mathrm{d}t} \mathfrak{E}_{k,2n-k-1}+\sum _{l=0}^{1}\big( \|\partial _{t}^{l}\boldsymbol{\mathcal{Q}}\|_{k,2n-k-1} ^{2}+\|\partial _{t}^{l}\boldsymbol{\mathcal{V}}\|_{k,2n-k-1}^{2}+\|\partial _{t}^{l}\nabla\phi\|_{k,2n-k-1}^{2}+\|\partial _{t}^{l}\operatorname{div} \mathbf{w} \|_{k,2n-k-1}^{2}\big)
  \nonumber \\
  \lesssim &\ \|\partial _{t} \mathbf{u}  \|^2_{k,2n-k-1} + \|\mathbf{u} \|^2_{k+1,2n-k} +\|\mathbf{G}^1\|^2_{k,2n-k-1} + \|\mathbf{G}^2\|^2_{k+1,2n-k-1} +\sigma \sum _{l=0}^{1}\| \partial _{t}^{l} \mathbf{W} \|_{0,2n-1}^{2}
     \nonumber \\
   &\quad+  \left( \|\mathbf{w} \|_{0}\|\mathbf{w} \|_{1}+\|\partial _{t} \mathbf{w} \|_{0}\|\partial _{t} \mathbf{w} \|_{1} \right)
\end{align}
for any $ \sigma>0 $, where we used the definition of $\boldsymbol{\mathcal{G}} = \nabla\times \mathbf{G}^2$. Notice that by \eqref{x4} and \eqref{x5},
\begin{align}\label{M4}
\sum _{l=0}^{1}\|\partial _{t}^{l}\mathbf{W}\|_{k, 2n-k-1} \lesssim \sum _{l=0}^{1}\|\partial _{t}^{l}\boldsymbol{\mathcal{V}}\|_{k,2n-k-1}+ \sum _{l=0}^{1}\|\partial _{t}^{l}\nabla \phi\| _{k,2n-k-1},
\end{align}
and by virtue of \eqref{div-curl-esti},
\begin{align}\label{M5}
 \sum _{l=0}^{1}\|\partial _{t}^{l}\mathbf{w}\|_{k+1, 2n-k-1} \lesssim \sum _{l=0}^{1}\|\partial _{t}^{l}\mathbf{W}\|_{k, 2n-k-1} +\sum _{l=0}^{1}\|\partial _{t}^{l}\mathop{\mathrm{div}}\nolimits \mathbf{w}\|_{k, 2n-k-1}+ \sum^{1}_{l=0}\| \partial _{t}^{l} \mathbf{w}\|_{0,2n}.
\end{align}
Utilizing \eqref{M4} and \eqref{M5}, we update \eqref{QV-first-sumup} as
\begin{align}\label{QV}
   &\ \frac{\mathrm{d}}{\mathrm{d}t} \mathfrak{E}_{k,2n-k-1}+ \sum _{l=0}^{1}\big( \|\partial _{t}^{l}\boldsymbol{\mathcal{Q}}\|_{k,2n-k-1} ^{2} + \|\partial _{t}^{l}\mathbf{W}\|_{k,2n-k-1}^{2}+\|\partial _{t}^{l}\mathbf{w}\| _{k+1,2n-k-1}^{2} \big)  \nonumber \\
 \lesssim &\ \|\partial _{t} \mathbf{u} \|^2_{k,2n-k-1}+\|\mathbf{u}\|^2_{k+1,2n-k}+\|\mathbf{G}^1\|^2_{k,2n-k-1}+\|\mathbf{G}^2\|^2_{k+1,2n-k-1}+\sigma \sum _{l=0}^{1}\| \partial _{t}^{l} \mathbf{W}\|_{0,2n-1}^{2}
   \nonumber \\
   &\quad+  \left( \|\mathbf{w}\|_{0}\|\mathbf{w}\|_{1}+\|\partial _{t} \mathbf{w}\|_{0}\|\partial _{t} \mathbf{w}\|_{1} \right)+\sum^{1}_{l=0}\|\partial _{t} ^{l} \mathbf{w}\|_{0,2n}^{2},
\end{align}
where $ 0 \leq k \leq 2n-1 $.

\vskip .1in
{\bf Step 5.}
 When $1\leq k\leq 2n-1$, acting $\nabla_{\mathcal{A}}\times$ on $\eqref{reformulationc}_3$, we get
\begin{equation}\label{u-tuoyuan}
\begin{aligned}
-\zeta\Delta \mathbf{u} &=\partial _{t} \mathbf{W} + 2\zeta \mathbf{W} + \mathbf{G}^4,\\
\mathbf{u}|_{\partial\Omega}&=\mathbf{0},
\end{aligned}
\end{equation}
where the last term on the right hand side is defined by
\begin{align*}
\mathbf{G}^4 = \zeta(\Delta_{\mathcal{A}}\mathbf{u}-\Delta \mathbf{u}) + 2\zeta(\nabla_{\mathcal{A}}\times \mathbf{w}-\nabla\times \mathbf{w}) + (\nabla_{\mathcal{A}}\times \partial_t\mathbf{w} - \nabla\times \partial_t\mathbf{w}).
\end{align*}
We remark that the calculations in this step appeared neither in the case when the angular velocity is absent \cite{Tan-Wang-SIMA}, nor the compressible case \cite{Feng-hong-zhu-compre}. The main purpose is, when later combined with an iteration scheme, to recover the dissipation mechanism of the velocity field $\mathbf{u}$. By similar arguments as in the proof of Lemma \ref{lem-}, we can show that
\begin{align}\label{G-4-esti}
\|\bar{\nabla}_{0}^{4N-1} \mathbf{G}^{4}\|_0^2 \lesssim \mathcal{E}_{N+2}(\mathcal{D}_{2N}+\mathcal{J}_{2N}+\mathcal{F}_{2N})\quad \text{and} \quad
\|\bar{\nabla}_{0}^{2(N+2)-1} \mathbf{G}^{4}\|_0^2 \lesssim \mathcal{E}_{2N}\mathcal{D}_{N+2}.
\end{align}
Thanks to the regularity theory for the elliptic system, it holds for $1\leq k\leq 2n$ that
\begin{align}\label{u-elli-1*}
\|\mathbf{u}\|_{k+1 , 2n-k}^2 \lesssim \| \partial _{t} \mathbf{W}\|^2_{k-1,2n-k} + \|\mathbf{W}\|^2_{k-1,2n-k} +\|\mathbf{G}^4\|^2_{k-1,2n-k}.
\end{align}
To proceed, we let $\boldsymbol{\Theta}=\mathbf{u}+\frac{1 }{\mu+\zeta}\boldsymbol{\eta}$. Then $\Theta$ solves (see \eqref{perturbc})
\begin{equation}\label{h}
 \begin{aligned}
 -(\mu+\zeta)\Delta \boldsymbol{\Theta}+\nabla p&= -\Delta_*\boldsymbol{\eta} - \partial _{t} \mathbf{u} + \zeta \mathbf{W}+\mathbf{G}^1,\\
 \operatorname{div} \boldsymbol{\Theta}&=G^3+\frac{1 }{\mu+\zeta}\operatorname{div}\boldsymbol{\eta},\\
\boldsymbol{\Theta}|_{\partial\Omega}&=\mathbf{0}.
 \end{aligned}
 \end{equation}
 To facilitate the subsequent presentation, we introduce the notation:
  \begin{equation}\label{n1}
 \mathcal{Y}_n :=  \|\bar{\nab}^{2n-1} \mathbf{G}^1\|_{0}^2  +  \| \bar{\nab}^{2n-1} \mathbf{G}^2\|_{1}^2 + \| \bar{\nab}^{2n-1} G^3\|_{1}^2 + \|\bar{\nab}^{2n-1} \mathbf{G}^4\|_0^2 + \| \operatorname{div}\boldsymbol\eta\|_{2n-1}^2.
\end{equation}
 By the regularity theory for the Stokes system (cf. Lemma \ref{i_linear_elliptic2}), we have for $1\leq k\leq 2n-1$ that
 \begin{align}\label{est-eta}
 &\ \|\boldsymbol{\Theta}\|^2_{k+1,2n-k-1}+\|\nabla p\|^2_{k-1,2n-k-1}
  \nonumber \\
\lesssim &\ \|\Delta_{*}\boldsymbol{\eta}\|^2_{k-1,2n-k-1}+\|\partial_t\mathbf{u}\|^2_{k-1,2n-k-1}+\|\mathbf{W}\|^2_{k-1,2n-k-1} +\|\mathbf{G}^1\|^2_{k-1,2n-k-1} \nonumber\\
 &\quad+\|G^3\|^2_{k,2n-k-1}+\|\operatorname{div}\boldsymbol{\eta} \|^2_{k,2n-k-1}\nonumber\\
 \lesssim &\ \|\boldsymbol\eta\|^2_{k-1,2n-k+1}+\|\partial_t \mathbf{u}\|^2_{k-1,2n-k-1}+\|\mathbf{W}\|^2_{k-1,2n-k-1} +\mathcal{Y}_n.
 \end{align}
Notice that since $\partial_t\boldsymbol{\eta}=\mathbf{u}$, it holds that
 \begin{align}
 \|\boldsymbol{\Theta}\|^2_{k+1,2n-k-1}&=\Big\|\mathbf{u}+\frac{1 }{\mu+\zeta}\boldsymbol{\eta}\Big\|^2_{k+1,2n-k-1}
 \nonumber\\
 &=\|\mathbf{u}\|^2_{k+1,2n-k-1}+\frac{1}{(\mu+\zeta)^2}\|\boldsymbol{\eta}\|^2_{k+1,2n-k-1}+\frac{1 }{(\mu+\zeta)}\frac{\mathrm{d}}{\mathrm{d}t}\|\boldsymbol{\eta}\|^2_{k+1,2n-k-1}.
 \end{align}
Then we update \eqref{est-eta} as
 \begin{align}\label{3.115}
 &\ \frac{\mathrm{d}}{\mathrm{d}t}\|\boldsymbol{\eta}\|^2_{k+1,2n-k-1}+\|\boldsymbol{\eta}\|^2_{k+1,2n-k-1}+\|\mathbf{u}\|^2_{k+1,2n-k-1}+\|\nabla p\|^2_{k-1,2n-k-1}
 \nonumber\\
\lesssim &\ \|\boldsymbol{\eta}\|^2_{k-1,2n-k+1}+\|\partial_t\mathbf{u}\|^2_{k-1,2n-k-1} +\|\mathbf{W}\|^2_{k-1,2n-k-1} +\mathcal{Y}_n
\end{align}
for $ 1 \leq k \leq 2n-1 $. Combining \eqref{QV} and \eqref{3.115}, and invoking \eqref{u-elli-1*} and \eqref{n1}, we obtain
\begin{align}\label{QV-combine}
&\ \frac{\mathrm{d}}{\mathrm{d}t}\big(\mathfrak{E}_{k,2n-k-1}+\|\boldsymbol\eta\|^2_{k+1,2n-k-1}\big) + \sum _{l=0}^{1}\big( \|\partial _{t}^{l}\boldsymbol{\mathcal{Q}}\|_{k,2n-k-1} ^{2} + \|\partial _{t}^{l}\mathbf{W}\|_{k,2n-k-1}^{2}+\|\partial _{t}^{l}\mathbf{w}\| _{k+1,2n-k-1}^{2}\big)\nonumber \\
    &\qquad+\|\boldsymbol\eta\|^2_{k+1,2n-k-1} +\|\mathbf{u}\|_{k+1,2n-k}^{2}  +\|\nabla p\|^2_{k-1,2n-k-1}\nonumber \\
   \lesssim &\ \|\partial _{t} \mathbf{u} \|^2_{k,2n-k-1} + \|\boldsymbol{\eta}\|^2_{k-1,2n-k+1} +\sigma\sum^{1}_{l=0} \| \partial _{t}^{l} \mathbf{W}\|_{0,2n-1}^{2} + \sum^{1}_{l=0}\big(\| \partial _{t} ^{l} \mathbf{W}\|^2_{k-1,2n-k} + \|\partial _{t} ^{l}  \mathbf{w}\|_{0,2n}^{2}\big) \nonumber \\
    &\qquad+  \left( \|\mathbf{w}\|_{0}\|\mathbf{w}\|_{1}+\|\partial _{t} \mathbf{w}\|_{0}\|\partial _{t} \mathbf{w}\|_{1} \right) +\mathcal{Y}_n
   \end{align}
   for any $ \sigma>0 $, where $ 1 \leq k \leq 2n-1 $.

\vskip .1in
   {\bf Step 6.} Let $k=2n-1$ in \eqref{QV-combine}. Then we get
   \begin{align}\label{QV-combine-1}
&\ \frac{\mathrm{d}}{\mathrm{d}t} \big(\mathfrak{E}_{2n-1,0}+\|\boldsymbol\eta\|^2_{2n}\big) + \sum _{l=0}^{1}\big( \|\partial _{t}^{l}\boldsymbol{\mathcal{Q}}\|_{2n-1} ^{2} + \|\partial _{t}^{l}\mathbf{W}\|_{2n-1}^{2}+\|\partial _{t}^{l}\mathbf{w}\| _{2n}^{2}\big) \notag\\
&\qquad +\|\boldsymbol\eta\|^2_{2n} +\|\mathbf{u}\|_{2n,1}^{2}  +\|\nabla p\|^2_{2n-2}\nonumber \\
   \lesssim &\ \|\partial _{t} \mathbf{u} \|^2_{2n-1} + \|\boldsymbol{\eta}\|^2_{2n-2,2} +\sigma\sum^{1}_{l=0} \| \partial _{t}^{l} \mathbf{W}\|_{0,2n-1}^{2} + \sum^{1}_{l=0}\big(\| \partial _{t} ^{l} \mathbf{W}\|^2_{2n-2,1} + \|\partial _{t} ^{l}  \mathbf{w}\|_{0,2n}^{2}\big) \nonumber \\
    &\qquad+  \left( \|\mathbf{w}\|_{0}\|\mathbf{w}\|_{1}+\|\partial _{t} \mathbf{w}\|_{0}\|\partial _{t} \mathbf{w}\|_{1} \right) +\mathcal{Y}_n.
   \end{align}
Choosing $\sigma>0$ sufficiently small and releasing a small portion of $\|\mathbf{w}\|_1^2+\|\partial_t\mathbf{w}\|_1^2$ from the right hand side of \eqref{QV-combine-1} by Cauchy's inequality,  we have
      \begin{align}\label{QV-combine-2}
&\ \frac{\mathrm{d}}{\mathrm{d}t}\big(\mathfrak{E}_{2n-1,0}+\|\boldsymbol\eta\|^2_{2n}\big)+ \sum _{l=0}^{1}\big( \|\partial _{t}^{l}\boldsymbol{\mathcal{Q}}\|_{2n-1} ^{2} + \|\partial _{t}^{l}\mathbf{W}\|_{2n-1}^{2}+\|\partial _{t}^{l}\mathbf{w}\| _{2n}^{2}\big) \notag\\
&\qquad +\|\boldsymbol\eta\|^2_{2n} +\|\mathbf{u}\|_{2n,1}^{2}  +\|\nabla p\|^2_{2n-2}\nonumber \\
   \lesssim &\ \|\partial _{t} \mathbf{u} \|^2_{2n-1} + \|\boldsymbol{\eta}\|^2_{2n-2,2} + \sum^{1}_{l=0}\big(\| \partial _{t} ^{l} \mathbf{W}\|^2_{2n-2,1} + \|\partial _{t} ^{l}  \mathbf{w}\|_{0,2n}^{2}\big)  +\mathcal{Y}_n.
   \end{align}
Taking $k=2n-2$ in \eqref{QV-combine} and repeating the preceding argument, we have
\begin{align}\label{QV-combine-2a}
&\ \frac{\mathrm{d}}{\mathrm{d}t}\big(\mathfrak{E}_{2n-2,1}+\|\boldsymbol\eta\|^2_{2n-1,1}\big) + \sum _{l=0}^{1}\big( \|\partial _{t}^{l}\boldsymbol{\mathcal{Q}}\|_{2n-2,1} ^{2} + \|\partial _{t}^{l}\mathbf{W}\|_{2n-2,1}^{2}+\|\partial _{t}^{l}\mathbf{w}\| _{2n-1,1}^{2}\big)\nonumber \\
    &\qquad+\|\boldsymbol\eta\|^2_{2n-1,1} +\|\mathbf{u}\|_{2n-1,2}^{2}  +\|\nabla p\|^2_{2n-3,1}\nonumber \\
   \lesssim &\ \|\partial _{t} \mathbf{u} \|^2_{2n-2,1} + \|\boldsymbol{\eta}\|^2_{2n-3,3}  + \sum^{1}_{l=0}\big(\| \partial _{t} ^{l} \mathbf{W}\|^2_{2n-3,2} + \|\partial _{t} ^{l}  \mathbf{w}\|_{0,2n}^{2}\big) +\mathcal{Y}_n.
   \end{align}
By properly combining \eqref{QV-combine-2} and \eqref{QV-combine-2a} (to absorb  $\|\partial _{t}^{l}\mathbf{W}\|_{2n-2,1}^{2}$ on the right hand side of \eqref{QV-combine-2} by the same quantity on the left hand side of \eqref{QV-combine-2a}), we obtain
\begin{align}\label{QV-combine-2b}
&\ \frac{\mathrm{d}}{\mathrm{d}t} \mathcal{N}_1 + \sum _{l=0}^{1}\big( \|\partial _{t}^{l}\boldsymbol{\mathcal{Q}}\|_{2n-1} ^{2} + \|\partial _{t}^{l}\mathbf{W}\|_{2n-1}^{2}+\|\partial _{t}^{l}\mathbf{w}\| _{2n}^{2}\big) +\|\boldsymbol\eta\|^2_{2n} +\|\mathbf{u}\|_{2n,1}^{2}  +\|\nabla p\|^2_{2n-2}\nonumber \\
   \lesssim &\ \|\partial _{t} \mathbf{u} \|^2_{2n-1} + \|\boldsymbol{\eta}\|^2_{2n-2,2}  + \sum^{1}_{l=0}\big(\| \partial _{t} ^{l} \mathbf{W}\|^2_{2n-3,2} + \|\partial _{t} ^{l}  \mathbf{w}\|_{0,2n}^{2}\big) +\mathcal{Y}_n,
   \end{align}
   where the energy functional $\mathcal{N}_1$ satisfies
   $$
   \mathcal{N}_1 \sim \big(\mathfrak{E}_{2n-1,0}+\|\boldsymbol\eta\|^2_{2n} + \mathfrak{E}_{2n-2,1}+\|\boldsymbol\eta\|^2_{2n-1,1} \big) \sim \big(\mathfrak{E}_{2n-1,0}+\|\boldsymbol\eta\|^2_{2n}\big).
   $$
   Taking $k=2n-3$ in \eqref{QV-combine} and repeating the preceding argument as in obtaining \eqref{QV-combine-2} give us
\begin{align}\label{QV-combine-3}
&\ \frac{\mathrm{d}}{\mathrm{d}t}\big(\mathfrak{E}_{2n-3,2}+\|\boldsymbol\eta\|^2_{2n-2,2}\big) + \sum _{l=0}^{1}\big( \|\partial _{t}^{l}\boldsymbol{\mathcal{Q}}\|_{2n-3,2} ^{2} + \|\partial _{t}^{l}\mathbf{W}\|_{2n-3,2}^{2}+\|\partial _{t}^{l}\mathbf{w}\| _{2n-2,2}^{2}\big)\nonumber \\
    &\qquad+\|\boldsymbol\eta\|^2_{2n-2,2} +\|\mathbf{u}\|_{2n-2,3}^{2}  +\|\nabla p\|^2_{2n-4,2}\nonumber \\
   \lesssim &\ \|\partial _{t} \mathbf{u} \|^2_{2n-3,2} + \|\boldsymbol{\eta}\|^2_{2n-4,4}  + \sum^{1}_{l=0}\big(\| \partial _{t} ^{l} \mathbf{W}\|^2_{2n-4,3} + \|\partial _{t} ^{l}  \mathbf{w}\|_{0,2n}^{2}\big) +\mathcal{Y}_n.
   \end{align}
   Combining \eqref{QV-combine-2b} and \eqref{QV-combine-3} (to absorb  $\|\boldsymbol\eta\|^2_{2n-2,2}$ and $\|\partial _{t}^{l}\mathbf{W}\|_{2n-3,2}^{2}$ on the right hand side of \eqref{QV-combine-2b} by the same quantities on the left hand side of \eqref{QV-combine-3}), we obtain
   \begin{align}\label{QV-combine-4}
&\ \frac{\mathrm{d}}{\mathrm{d}t} \mathcal{N}_2 + \sum _{l=0}^{1}\big( \|\partial _{t}^{l}\boldsymbol{\mathcal{Q}}\|_{2n-1} ^{2} + \|\partial _{t}^{l}\mathbf{W}\|_{2n-1}^{2}+\|\partial _{t}^{l}\mathbf{w}\| _{2n}^{2}\big) +\|\boldsymbol\eta\|^2_{2n} +\|\mathbf{u}\|_{2n,1}^{2}  +\|\nabla p\|^2_{2n-2}\nonumber \\
   \lesssim &\ \|\partial _{t} \mathbf{u} \|^2_{2n-1} + \|\boldsymbol{\eta}\|^2_{2n-4,4}  + \sum^{1}_{l=0}\big(\| \partial _{t} ^{l} \mathbf{W}\|^2_{2n-4,3} + \|\partial _{t} ^{l}  \mathbf{w}\|_{0,2n}^{2}\big) +\mathcal{Y}_n,
   \end{align}
   where the energy functional $\mathcal{N}_2$ satisfies
   $$
   \mathcal{N}_2 \sim \big(\mathcal{N}_1+\mathfrak{E}_{2n-3,2}+\|\boldsymbol\eta\|^2_{2n-2,2}\big) \sim \big(\mathfrak{E}_{2n-1,0}+\|\boldsymbol\eta\|^2_{2n}\big).
   $$

\vskip .1in
{\bf Step 7.} By repeating the above arguments in the preceding step for $k=2n-4,2n-5,\dots,1$, we can show that
    \begin{align}\label{QV-combine-5}
&\ \frac{\mathrm{d}}{\mathrm{d}t} \mathcal{N}_n + \sum _{l=0}^{1}\big( \|\partial _{t}^{l}\boldsymbol{\mathcal{Q}}\|_{2n-1} ^{2} + \|\partial _{t}^{l}\mathbf{W}\|_{2n-1}^{2}+\|\partial _{t}^{l}\mathbf{w}\| _{2n}^{2}\big) +\|\boldsymbol\eta\|^2_{2n} +\|\mathbf{u}\|_{2n,1}^{2}  +\|\nabla p\|^2_{2n-2}\nonumber \\
   \lesssim &\ \|\partial _{t} \mathbf{u} \|^2_{2n-1} + \|\boldsymbol{\eta}\|^2_{0,2n}  + \sum^{1}_{l=0}\big(\| \partial _{t} ^{l} \mathbf{W}\|^2_{0,2n-1} + \|\partial _{t} ^{l}  \mathbf{w}\|_{0,2n}^{2}\big) +\mathcal{Y}_n,
   \end{align}
   where the energy functional $\mathcal{N}_{n}$ satisfies
$$
\mathcal{N}_{n} \sim \sum_{k=1}^{2n-1} \big(\mathfrak{E}_{k,2n-k-1}+\|\boldsymbol\eta\|^2_{k+1,2n-k-1}\big) \sim \big(\mathfrak{E}_{2n-1,0}+\|\boldsymbol\eta\|^2_{2n}\big).
$$
Taking $k=0$ in \eqref{QV}, we can show that
\begin{align}\label{QV0}
   &\ \frac{\mathrm{d}}{\mathrm{d}t} \mathfrak{E}_{0,2n-1}+ \sum _{l=0}^{1}\big( \|\partial _{t}^{l}\boldsymbol{\mathcal{Q}}\|_{0,2n-1} ^{2} + \|\partial _{t}^{l}\mathbf{W}\|_{0,2n-1}^{2}+\|\partial _{t}^{l}\mathbf{w}\| _{1,2n-1}^{2}\big)
 \nonumber \\
 \lesssim &\ \|\partial_t \mathbf{u}\|^2_{0,2n-1}+\|\mathbf{u}\|^2_{1,2n}+ \sum^{1}_{l=0}\|\partial^l_t \mathbf{w}\|_{0,2n}^{2} + \mathcal{Y}_n ,
\end{align}
provided $ \sigma $ is suitably small. Using \eqref{QV0} to absorb $\|\partial _{t}^{l}\mathbf{W}\|_{0,2n-1}^{2}$ on the right hand side of \eqref{QV-combine-5}, we obtain
    \begin{align}\label{QV-combine-6}
&\ \frac{\mathrm{d}}{\mathrm{d}t} \mathfrak{E}_n + \sum _{l=0}^{1}\big( \|\partial _{t}^{l}\boldsymbol{\mathcal{Q}}\|_{2n-1} ^{2} + \|\partial _{t}^{l}\mathbf{W}\|_{2n-1}^{2}+\|\partial _{t}^{l}\mathbf{w}\| _{2n}^{2}\big) +\|\boldsymbol\eta\|^2_{2n} +\|\mathbf{u}\|_{2n,1}^{2}  +\|\nabla p\|^2_{2n-2}\nonumber \\
   \lesssim &\ \|\partial_{t} \mathbf{u} \|^2_{2n-1} + \|\mathbf{u}\|^2_{1,2n} +\|\boldsymbol{\eta}\|^2_{0,2n}  + \sum^{1}_{l=0} \|\partial _{t} ^{l}  \mathbf{w}\|_{0,2n}^{2} +\mathcal{Y}_n \notag\\
   \lesssim &\ \|\partial_{t} \mathbf{u} \|^2_{2n-1} + \bar{\mathcal{D}}_n^\ast+\bar{\mathcal{D}}_n^\sharp+\mathcal{Y}_n,
   \end{align}
   where we invoked \eqref{Dn*}, \eqref{Esharp} and the energy functional $\mathfrak{E}_{n}$ satisfies
$$
\mathfrak{E}_{n} \sim \sum_{k=0}^{2n-1} \big(\mathfrak{E}_{k,2n-k-1}+\|\boldsymbol\eta\|^2_{k+1,2n-k-1}\big) \sim \big(\mathfrak{E}_{2n-1,0}+\|\boldsymbol\eta\|^2_{2n}\big).
$$
In view of \eqref{E-kn-defi} we see that
 \begin{align*}
 \mathfrak{E}_{n} \sim \big( \|\boldsymbol{\mathcal{Q}}\|_{2n-1}^2+\|\boldsymbol{\mathcal{V}}\|^2_{2n-1}+\|\partial_3\phi\|^2_{2n-1}
+ \|\operatorname{div} \mathbf{w}\|_{2n-1}^{2}+\|\boldsymbol\eta\|^2_{2n}\big),
 \end{align*}
which is \eqref{y1}. Moreover, taking $ k=2n $ in \eqref{u-elli-1*}, we get
\begin{align*}
 \|\mathbf{u}\|_{2n+1 }^2  \lesssim \|\mathbf{W}\|^2_{2n-1 }+\| \partial _{t} \mathbf{W}\|^2_{2n-1}
+\|\mathbf{G}^4\|^2_{2n-1}.
\end{align*}
This along with \eqref{QV-combine-6} further implies that
\begin{align}\label{QV1-final}
&\ \frac{\mathrm{d}}{\mathrm{d}t} \mathfrak{E}_n +\sum^{1}_{l= 0}\big(\|\partial _{t}^{l}\boldsymbol{\mathcal{Q}}\|_{2n-1} ^{2}+\|\partial^l_{t}\mathbf{W}\|^2_{2n-1}+\|\partial_t^{l}\mathbf{w}\|^2_{2n}\big)+\|\boldsymbol\eta\|^2_{2n}+\|\mathbf{u}\|^2_{2n+1}+\|\nabla p\|^2_{2n-2}
\nonumber\\
\lesssim &\ \|\partial _{t}\mathbf u\|_{2n-1}^2+ \bar{\mathcal{D}}_n^\ast+\bar{\mathcal{D}}_n^\sharp+\mathcal{Y}_n.
\end{align}
For the estimate of $\mathcal{Y}_n$ (defined by \eqref{n1}), by Lemma \ref{lem-G-TILDA-G}, Lemma \ref{lem-G-TILD-G-technical} and \eqref{G-4-esti}, we can show that
\begin{gather}\label{esti-Y-2N}
 \mathcal{Y}_{2N}\lesssim  { \mathcal{E}_{N+2}  }(\mathcal{D}_{2N} +  \mathcal{J}_{2N} +\mathcal{F}_{2N}) \quad \text{and} \quad \mathcal{Y}_{N+2}\lesssim   {\mathcal{E}}_{2N}  {\mathcal{D}}_{N+2}.
\end{gather}
This, alongside \eqref{QV1-final}, completes the proof of the lemma.
\end{proof}

The next lemma is devoted to the estimation of $\|\partial^j_t \mathbf{W}\|_{2n-2j-1}$ and $\|\partial^j_{t}\operatorname{div} \mathbf{w}\|_{2n-2j-1}$ with $j=1,2,\dots,n-1$.

\begin{lemma}\label{lem-time-deri-curlw}
Let $ (\boldsymbol{\eta},\mathbf{u},p,\mathbf{w}) $ be the local solution to \eqref{reformulationc} satisfying \eqref{a-priori} for some sufficiently small constant $\delta$. There exists an energy functional
\begin{align}\label{y2}
\bar{\mathfrak{E}}_{n}(t) \sim \sum _{j=1}^{n-1}\big(\|\partial^j_{t}\boldsymbol{\mathcal{Q}}\|^2_{2n-2j-1}+\| \partial _{t}^{j}\boldsymbol{\mathcal{V}}\|^2_{2n-2j-1}+\|\partial^j_{t}\partial_3\phi\|^2_{2n-2j-1}+ \|\partial _{t}^{j} \operatorname{div} \mathbf{w}\|_{2n-2j-1}^{2} \big),
\end{align}
such that for any integer $ N \geq 4 $,
\begin{align}
\bullet \ \ &\ \frac{\mathrm{d}}{\mathrm{d}t}\bar{\mathfrak{E}}_{2N}+\sum^{2N-1}_{j=1} \sum^{1}_{l=0}\big(\|\partial^   {j+l}_t\mathbf{W}\|^2_{4N-2j-1}+\|\partial^   {j+l}_t\mathbf{w}\|^2_{4N-2j}\big)+ \sum^{2N}_{j=1}\| \partial _{t}^{j}\mathbf{u}\| _{4N-2j+1}^{2}\nonumber\\
\lesssim &\ \bar{\mathcal{D}}_{2N}+ \bar{\mathcal{D}}^\ast_{2N} + { \mathcal{E}_{N+2}  }(\mathcal{D}_{2N} +  \mathcal{J}_{2N} +\mathcal{F}_{2N}),\label{E2N}\\
\bullet \ \ &\ \frac{\mathrm{d}}{\mathrm{d}t}\bar{\mathfrak{E}}_{N+2}
 +\sum^{N+1}_{j=1} \sum^{1}_{l=0}\big(\|\partial^   {j+l}_t\mathbf{W}\|^2_{2(N+2)-2j-1}+\|\partial^   {j+l}_t\mathbf{w}\|^2_{2(N+2)-2j}\big)+\sum^{N+2}_{j=1}\| \partial _{t}^{j}\mathbf{u}\| _{2(N+2)-2j+1}^{2} \nonumber\\
\lesssim &\ \bar{\mathcal{D}}_{N+2}+ \bar{\mathcal{D}}^\ast_{N+2} + { \mathcal{E}_{2N}  }\mathcal{D}_{N+2}.\label{EN+2}
\end{align}
 \end{lemma}

\begin{proof}
{\bf Step 1.} Applying $\partial^j_{t}$ to \eqref{pa-3-2eta-eq} and \eqref{eq-w-star}, we have
\begin{subequations}\label{t-first-two-compo-ODE}
 \begin{alignat}{2}
  (\mu+\zeta)\partial^{j+1} _{t}\boldsymbol{\mathcal{Q}}+\boldsymbol{\mathcal{Q}}+ \zeta\partial^{j}_t \boldsymbol{\mathcal{V}} &=\partial^j_{t}\boldsymbol{\mathcal{H}},\label{t-pa-3-2eta-eq}
 \\
   \partial^{j+1}_{t}\boldsymbol{\mathcal{V}} +2 \zeta \partial^{j}_t \boldsymbol{\mathcal{V}} +\zeta\partial^{j+1} _{t}\boldsymbol{\mathcal{Q}} &=\partial^j_{t}\tilde{\boldsymbol{\mathcal{G}}}. \label{t-eq-w-star}
 \end{alignat}
 \end{subequations}
By similar arguments as in deriving \eqref{w-star-eta-star-combined-sumup}, we get for $j=1,\dots,n-1$ and $0\leq k\leq 2n-2j-1$ that
\begin{align}\label{w-combined-sumup}
    \frac{\mathrm{d}}{\mathrm{d}t} \mathcal{R}_1 + \mathcal{R}_2 \lesssim \|\partial^{j+1}_t \mathbf{u}\|^2_{k,2n-2j-k-1}+\|\partial^j_{t}\mathbf{u}\|^2_{k+1,2n-2j-k}+\mathcal{Y} _{n},
    \end{align}
    where $ \mathcal{Y} _{n} $ is defined by \eqref{n1} and the energy functionals satisfy
    $$
    \begin{aligned}
    \mathcal{R}_1 & \sim \|(\partial^j_{t}\boldsymbol{\mathcal{Q}},\partial^j_{t}\boldsymbol{\mathcal{V}})\|^2_{k,2n-2j-k-1},\\
    \mathcal{R}_2 & \sim \|(\partial^{j+1}_{t}\boldsymbol{\mathcal{Q}},\partial^j_{t}\boldsymbol{\mathcal{Q}},\partial^{j+1} _{t}\boldsymbol{\mathcal{V}},\partial^j_{t}\boldsymbol{\mathcal{V}})\|^2_{k,2n-2j-k-1}.
    \end{aligned}
    $$
    Here, for simplicity, we used the short notation for the sum of the square of the norms of the functions. From  \eqref{phi-3-eq} we get
\begin{align*}
\partial_t (\partial^{j}_t\partial _{3}\phi)+2 \zeta (\partial^j_{t}\partial _{3}\phi)= \partial _{t}^{j} \big( \zeta\partial_3 \mathop{\mathrm{div}} _{\ast}\nolimits \mathbf{u} _{\ast} - \zeta \Delta _{\ast}u _{3}+ \mathcal{G}_{3}\big) + \zeta \partial^{j+1}_{t} \mathcal{Q}_{3} - \big( \zeta \Delta _{\ast}(\partial^j_{t}\mathcal{P} \mathbf{u})-\partial^j_{t}\mathcal{P} \boldsymbol{\mathcal{G}}\big)_{3},
\end{align*}
from which we get a similar estimate for $ \partial _{t}^{j}\partial _{3}\phi $ as \eqref{phi*}:
\begin{align}\label{phitj}
&\ \frac{\mathrm{d}}{\mathrm{d}t}\|\partial^j_{t}\partial _{3}\phi\|_{k,2n-2j-k-1}^{2}+ \sum _{l=0}^{1}\|\partial^{j+l}_{t}\partial _{3}\phi\|_{k,2n-2j-k-1}^{2}\nonumber\\
 \lesssim &\ \|\partial^{j+1}_{t}\boldsymbol{\mathcal{Q}}\|_{k,2n-2j-k-1}^{2}+\|\partial^j_{t}\mathbf{u}\|^2_{k+1,2n-2j-k}+\|\partial^j_{t}\boldsymbol{\mathcal{G}}\|^2_{k+1,2n-2j-k-1}.
\end{align}
Recall the equation for $ \mathop{\mathrm{div}}\nolimits \mathbf{w} $ in \eqref{div-eq-w}:
\begin{align*}
\partial_t ^{j+1}\operatorname{div} \mathbf{w} + 2 \zeta \partial_t^j\operatorname{div} \mathbf{w} = \partial _{t}^{j} \operatorname{div} \mathbf{G}^2,\ \ j=1,\dots,n-1,
\end{align*}
from which we derive that
\begin{align}\label{divwtj}
  \frac{\mathrm{d}}{\mathrm{d}t}\|\partial^j_{t}\operatorname{div}\mathbf{w}\|_{k,2n-2j-k-1}^{2}+\sum _{l=1}^{1}\|\partial^{j+l}_{t}\operatorname{div}\mathbf{w}\|_{k,2n-2j-k-1}^{2}
 \lesssim \| \partial^j_{t} \operatorname{div} \mathbf{G}^2\|^2_{k,2n-2j-k-1}.
\end{align}
Also note that by \eqref{phi2-pa-t}, for $l=0,1$,
\begin{align*}
\|\partial_{t} ^{j+l}\nabla_{h}\phi\|_{k,2n-2j-k-1}^{2} &\lesssim \|\partial_{t} ^{j+l}\partial _{3}\phi\|_{k,2n-2j-k-1}^{2}+\sigma \|\partial_{t} ^{j+l}\mathbf{W}\|_{0,2n-2j-1}^{2}+   \| \partial_{t} ^{j+l}\mathbf{w}\|_{0}\|\partial_{t} ^{j+l}\mathbf{w}\|_{1}.
\end{align*}
By repeating the arguments in Step 4 of the proof of Lemma \ref{lem-eta-w},  we can show that there exists an energy functional $ \bar{\mathfrak{E}}_{k,2n-2j-k-1}$ $(j=1,\dots, n-1,\,0\leq k\leq 2n-2j-1)$ such that
\begin{gather*}
\bar{\mathfrak{E}}_{k,2n-2j-k-1} \sim \|(\partial _{t}^{j}\boldsymbol{\mathcal{Q}},\partial _{t}^{j}\boldsymbol{\mathcal{V}},\partial _{t}^{j}\partial _{3}\phi, \partial _{t}^{j} \operatorname{div}\mathbf{w})\|_{k,2n-2j-k-1}^{2},
\end{gather*}
and
\begin{align}\label{M6}
&\ \frac{\mathrm{d}}{\mathrm{d}t} \bar{\mathfrak{E}}_{k,2n-2j-k-1} + \sum^{1}_{l=0}\big( \|\partial _{t}^   {j+l}\mathbf{W}\|_{k,2n-2j-k-1}^2 +\|\partial _{t}^   {j+l}\mathbf{w}\|_{k+1,2n-2j-k-1}^{2}\big) \nonumber \\
  \lesssim &\ \|\partial^{j+1}_t\mathbf{u}\|^2_{k,2n-2j-k-1}+\|\partial^j_{t}\mathbf{u}\|^2_{k+1,2n-2j-k}+ \sigma \sum _{l=0}^{1} \| \partial _{t}^   {j+l} \mathbf{W}\|_{0,2n-2j-1}^{2} \nonumber \\
   &\quad + \sum _{l=0}^{1} \|\partial _{t} ^{j+l} \mathbf{w}\|_{0,2n-2j}^{2} + \sum _{l=0}^{1} \|\partial _{t}^{j+l} \mathbf{w}\|_{0}\|\partial _{t}^{j+l} \mathbf{w}\|_{1} + \mathcal{Y}_n,
\end{align}
where $\sigma>0$ is an arbitrary constant. Here we dropped the norm of $\boldsymbol{\mathcal{Q}}$ from the left hand side of \eqref{M6}, since it does not change the sign of the inequality, and in the sequel it does not affect the energy estimates. Choosing small $\sigma>0$ and releasing a small portion of $\|\partial _{t}^{j+l} \mathbf{w}\|_{1}^2$ by Cauchy's inequality, we update \eqref{M6} as
\begin{align}\label{tim-sumup0}
&\ \frac{\mathrm{d}}{\mathrm{d}t} \bar{\mathfrak{E}}_{k,2n-2j-k-1} + \sum^{1}_{l=0}\big( \|\partial _{t}^   {j+l}\mathbf{W}\|_{k,2n-2j-k-1}^2 +\|\partial _{t}^   {j+l}\mathbf{w}\|_{k+1,2n-2j-k-1}^{2}\big) \nonumber \\
  \lesssim &\ \|\partial^{j+1}_t\mathbf{u}\|^2_{k,2n-2j-k-1}+\|\partial^j_{t}\mathbf{u}\|^2_{k+1,2n-2j-k} + \bar{\mathcal{D}}_{n}^{\ast} + \mathcal{Y}_n,
\end{align}
where we also used \eqref{Dn*} for $\|\partial _{t} ^{j} \mathbf{w}\|_{0,2n-2j}^{2} + \|\partial _{t} ^{j+1} \mathbf{w}\|_{0,2n-2j}^{2}$.

\vskip .1in
{\bf Step 2.} When $ k=0 $ and $ j=1,\dots,n-1 $, we notice by \eqref{Dnt} and \eqref{Dn*} that
\begin{align*}
 \|\partial _{t}^{j+1} \mathbf{u} \|_{0,2n-2j-1}^{2} + \|\partial _{t}^{j}\mathbf{u}\|_{1,2n-2j}^{2} \lesssim \bar{\mathcal{D}}_n+\bar{\mathcal{D}}_{n}^{\ast}.
\end{align*}
We then get from \eqref{tim-sumup0} that
\begin{align}\label{w-eta-t-0-induction}
\frac{\mathrm{d}}{\mathrm{d}t} \bar{\mathfrak{E}}_{0,2n-2j-1} + \sum^{1}_{l=0}\big( \|\partial _{t}^   {j+l}\mathbf{W}\|_{0,2n-2j-1}^2 +\|\partial _{t}^   {j+l}\mathbf{w}\|_{1,2n-2j-1}^{2}\big) \lesssim \bar{\mathcal{D}}_{n}+\bar{\mathcal{D}}_{n}^{\ast}+\mathcal{Y}_n.
\end{align}
When $k=1$, clearly we have
\begin{align}\label{M7}
\|\partial _{t}^{j+1} \mathbf{u} \|_{1,2n-2j-2}^{2} \lesssim \bar{\mathcal{D}}_n +\bar{\mathcal{D}}_{n}^{\ast}
\end{align}
for $ j=1,\dots,n-1 $. By \eqref{M7}, we deduce from \eqref{tim-sumup0} that
\begin{align}\label{M8}
&\ \frac{\mathrm{d}}{\mathrm{d}t} \bar{\mathfrak{E}}_{1,2n-2j-2}
 + \sum^{1}_{l=0}\big(\|\partial _{t}^   {j+l}\mathbf{W}\|_{1,2n-2j-2}^2 +\|\partial _{t}^   {j+l}\mathbf{w}\|_{2,2n-2j-2}^{2}\big)\nonumber \\
   \lesssim &\ \|\partial^j_{t}\mathbf{u}\|^2_{2,2n-2j-2}+ \bar{\mathcal{D}}_{n}+\bar{\mathcal{D}}_{n}^{\ast}+\mathcal{Y}_n.
\end{align}
From \eqref{u-tuoyuan}, we get
\begin{equation}\label{u-tuoyuan1*}
\begin{aligned}
-\zeta\Delta \partial^j_{t} \mathbf{u} &= \partial^{j+1}_t\mathbf{W} + 2\zeta \partial^j_{t} \mathbf{W}+\partial^j_t \mathbf{G}^4,\\
\partial^j_{t}\mathbf{u}|_{\partial\Omega} &=\mathbf{0},
\end{aligned}
\end{equation}
where $ j=1,\dots, n-1 $. This, alongside the classical regularity theory for the elliptic system, implies for $1\leq k\leq 2n-2j$ that
\begin{align}\label{u-elli-1}
\|\partial^j_{t}\mathbf{u}\|_{k+1 , 2n-2j-k}^2
 & \lesssim \|\partial^{j+1}_{t}\mathbf{W}\|_{k-1 , 2n-2j-k}^2+
\|\partial^j _{t}\mathbf{W}\|_{k-1 , 2n-2j-k}^2
+\|\partial^j_t\mathbf{G}^4\|_{k-1,2n-2j-k}^2.
\end{align}
When $k=1$, we obtain
\begin{align}\label{M9}
\|\partial^j_{t}\mathbf{u}\|_{2, 2n-2j-1}^2
 & \lesssim \|\partial^{j+1}_{t}\mathbf{W}\|_{0 , 2n-2j-1}^2+
\|\partial^j _{t}\mathbf{W}\|_{0, 2n-2j-1}^2
+\|\partial^j_t\mathbf{G}^4\|_{0,2n-2j-1}^2.
\end{align}
By properly combining \eqref{M8} and \eqref{M9} (to absorb $\|\partial^j_{t}\mathbf{u}\|^2_{2,2n-2j-2}$ on the right side of \eqref{M8} by $\|\partial^j_{t}\mathbf{u}\|^2_{2,2n-2j-1}$ on the left side of \eqref{M9}), we can show that
\begin{align}\label{w-eta-t-1-induction}
&\ \frac{\mathrm{d}}{\mathrm{d}t} \bar{\mathfrak{E}}_{1,2n-2j-2}
 + \sum^{1}_{l=0}\big(\|\partial _{t}^   {j+l}\mathbf{W}\|_{1,2n-2j-2}^2 +\|\partial _{t}^   {j+l}\mathbf{w}\|_{2,2n-2j-2}^{2}\big)+\|\partial^j_{t}\mathbf{u}\|_{2 , 2n-2j-1}^2 \nonumber \\
    \lesssim &\ \sum^{1}_{l=0}\|\partial^{j+l}_t\mathbf{W}\|^2_{0,2n-2j-1}+\bar{\mathcal{D}}_{n}+\bar{\mathcal{D}}_{n}^{\ast}+\mathcal{Y}_n.
\end{align}
Combining \eqref{w-eta-t-1-induction} with \eqref{w-eta-t-0-induction} implies that there exists an energy functional $ \bar{\mathfrak{E}}_{1,n} $
such that
\begin{align*}
\bar{\mathfrak{E}}_{1,n} \sim \sum _{j=1}^{n-1}(\bar{\mathfrak{E}}_{0,2n-2j-1} +\bar{\mathfrak{E}}_{1,2n-2j-2})
\end{align*}
and
\begin{align}\label{w-eta-t-k1-induction-con}
&\ \frac{\mathrm{d}}{\mathrm{d}t}\bar{\mathfrak{E}}_{1,n}
 + \sum^{n-1}_{j=1}\sum^{1}_{l=0}\big(\|\partial _{t}^{j+l}\mathbf{W}\|_{1,2n-2j-2}^2 +\|\partial _{t}^{j+l}\mathbf{w}\|_{2,2n-2j-2}^{2}\big)+\sum^{n-1}_{j=1}\|\partial^j_{t}\mathbf{u}\|_{2 , 2n-2j-1}^2
  \nonumber \\
  \lesssim &\ \bar{\mathcal{D}}_n+ \bar{\mathcal{D}}_n^\ast+\mathcal{Y}_{n}.
\end{align}

\vskip .1in
{\bf Step 3.} Similarly, when $k\geq2$, we deduce from \eqref{tim-sumup0} and \eqref{u-elli-1} that
 \begin{align}\label{w-eta-tj-induction}
 &\ \frac{\mathrm{d}}{\mathrm{d}t} \bar{\mathfrak{E}}_{k,2n-2j-k-1}
+\sum^{1}_{l=0}\big(\|\partial^   {j+l}_t\mathbf{W}\|^2_{k,2n-2j-k-1}+\|\partial^   {j+l}_t\mathbf{w}\|^2_{k+1,2n-2j-k-1}\big)+ \|\partial^j_{t}\mathbf{u}\|_{k+1 , 2n-2j-k}^2\nonumber\\
\lesssim &\ \sum^1_{l=0}\|\partial_{t}^{j+l} \mathbf{W}\|_{k-1 , 2n-2j-k}^2 +\|\partial^{j+1}_t\mathbf{u}\|^2_{k,2n-2j-k-1} + \bar{\mathcal{D}}_n^\ast + \mathcal{Y}_n
    \end{align}
    for $j=1,\dots,n-2$ and $2\le k \le 2n-2j-1$. Taking $ k=2 $ in \eqref{w-eta-tj-induction} and summing over $ 1\le j\le n-2 $, we get
\begin{align}\label{w-eta-tk2-induction}
&\ \frac{\mathrm{d}}{\mathrm{d}t} \sum_{j=1}^{n-2} \bar{\mathfrak{E}}_{2,2n-2j-3} + \sum_{j=1}^{n-2}\sum^{1}_{l=0}\big(\|\partial _{t}^{j+l}\mathbf{W}\|_{2,2n-2j-3}^2 +\|\partial _{t}^{j+l}\mathbf{w}\|_{3,2n-2j-3}^{2}\big)+ \sum_{j=1}^{n-2}\|\partial^j_{t}\mathbf{u}\|_{3, 2n-2j-2}^2
 \nonumber \\
 \lesssim &\ \sum_{j=1}^{n-2}\sum^1_{l=0}\|\partial_{t}^{j+l} \mathbf{W}\|_{1, 2n-2j-2}^2 + \sum_{j=1}^{n-2}\|\partial^{j+1}_t\mathbf{u}\|^2_{2,2n-2j-3} + \bar{\mathcal{D}}_n^\ast +\mathcal{Y}_n.
\end{align}
By absorbing the first two terms on the right hand side of \eqref{w-eta-tk2-induction} by the corresponding terms on the left hand side of \eqref{w-eta-t-k1-induction-con}, we deduce that there exists an energy functional $ \bar{\mathfrak{E}}_{2,n} $ such that
\begin{gather*}
 \bar{\mathfrak{E}}_{2,n} \sim \Big(\bar{\mathfrak{E}}_{1,n} +\sum_{j=1}^{n-2}\bar{\mathfrak{E}}_{2,2n-2j-3} \Big) \sim \sum _{j=1}^{n-1}(\bar{\mathfrak{E}}_{0,2n-2j-1} +\bar{\mathfrak{E}}_{1,2n-2j-2}) +\sum_{j=1}^{n-2}\bar{\mathfrak{E}}_{2,2n-2j-3}
\end{gather*}
and
\begin{align}\label{M10}
&\ \frac{\mathrm{d}}{\mathrm{d}t} \bar{\mathfrak{E}}_{2,n}+  \sum^{n-1}_{j=1}\sum^{1}_{l=0}\big(\|\partial _{t}^{j+l}\mathbf{W}\|_{1,2n-2j-2}^2 +\|\partial _{t}^{j+l}\mathbf{w}\|_{2,2n-2j-2}^{2}\big)+\sum^{n-1}_{j=1}\|\partial^j_{t} \mathbf{u}\|_{2 , 2n-2j-1}^2\nonumber\\
&\qquad\quad\ +\sum^{n-2}_{j=1}\sum^{1}_{l=0}\big(\|\partial _{t}^{j+l}\mathbf{W}\|_{2,2n-2j-3}^2 +\|\partial _{t}^{j+l}\mathbf{w}\|_{3,2n-2j-3}^{2}\big)+ \sum^{n-2}_{j=1}\|\partial^j_{t} \mathbf{u}\|_{3 , 2n-2j-2}^2                                              \nonumber \\
 \lesssim &\ \bar{\mathcal{D}}_n+ \bar{\mathcal{D}}_n^\ast+\mathcal{Y}_{n}.
\end{align}
Note that by virtue of \eqref{u-elli-1},
\begin{align*}
\|\partial _{t}^{n-1}\mathbf{u}\|_{3}^{2} \lesssim  \|\partial^{n-1}_{t}\mathbf{W}\|_{1 }^2+ \|\partial^{n} _{t}\mathbf{W}\|_{1}^2 +\|\partial^{n-1}_t \mathbf{G}^4\|_{1}^2,
\end{align*}
the left hand side of which corresponds to $ \|\partial^j_{t}\mathbf{u}\|_{3 , 2n-2j-2}^2 $ when $j=n-1$, and the first two terms on the right hand side are contained in the first summation on the left hand side of \eqref{M10}. Hence, we can increase the upper bound of the last summation on the right hand side of \eqref{M10} to $n-1$, i.e.,
\begin{align}\label{M11}
&\ \frac{\mathrm{d}}{\mathrm{d}t} \bar{\mathfrak{E}}_{2,n}+  \sum^{n-1}_{j=1}\sum^{1}_{l=0}\big(\|\partial _{t}^{j+l}\mathbf{W}\|_{1,2n-2j-2}^2 +\|\partial _{t}^{j+l}\mathbf{w}\|_{2,2n-2j-2}^{2}\big)+\sum^{n-1}_{j=1}\|\partial^j_{t} \mathbf{u}\|_{2 , 2n-2j-1}^2\nonumber\\
&\qquad\quad\ +\sum^{n-2}_{j=1}\sum^{1}_{l=0}\big(\|\partial _{t}^{j+l}\mathbf{W}\|_{2,2n-2j-3}^2 +\|\partial _{t}^{j+l}\mathbf{w}\|_{3,2n-2j-3}^{2}\big)+ \sum^{n-1}_{j=1}\|\partial^j_{t} \mathbf{u}\|_{3 , 2n-2j-2}^2 \nonumber \\
 \lesssim &\ \bar{\mathcal{D}}_n+ \bar{\mathcal{D}}_n^\ast+\mathcal{Y}_{n}.
\end{align}

\vskip .1in
{\bf Step 4.} By repeating the arguments in the preceding step for $k$ up to $2n-3$, we can show that there exists an energy functional $\bar{\mathfrak{E}}_n$ satisfying \eqref{y2} and
\begin{align}\label{M12}
&\ \frac{\mathrm{d}}{\mathrm{d}t} \bar{\mathfrak{E}}_{n} +  \sum^{n-1}_{j=1}\sum^{1}_{l=0}\big(\|\partial _{t}^{j+l}\mathbf{W}\|_{2n-2j-1}^2 +\|\partial _{t}^{j+l}\mathbf{w}\|_{2n-2j}^{2}\big)+\sum^{n-1}_{j=1}\|\partial^j_{t} \mathbf{u}\|_{2n-2j+1}^2\nonumber\\
 \lesssim &\ \bar{\mathcal{D}}_n+ \bar{\mathcal{D}}_n^\ast+\mathcal{Y}_{n}.
\end{align}
Then \eqref{E2N} and \eqref{EN+2} follow from \eqref{M12} and \eqref{esti-Y-2N}. This completes the proof of the lemma.
\end{proof}

\section{Synthesis and Finalization}
\label{sub:synthesis}

In this section, we assemble the \emph{a priori} estimates established in the previous sections to close the entire energy scheme in terms of $ \mathcal{E}_{n} $ and $ \mathcal{D}_{n} $ by virtue of the Stokes equations. To this end, we first prove the following result on the control of $ \mathcal{E}_{n} $. The proof is similar to those in \cite{Tan-Wang-SIMA,Feng-hong-zhu-compre}.

\begin{lemma}\label{lemma-energy-recover}
Let $ (\boldsymbol{\eta},\mathbf{u}, p,\mathbf{w}) $ be the local solution to \eqref{reformulationc} satisfying \eqref{a-priori} for some sufficiently small constant $\delta$. Then for any $ n = 2N $ or $n=N+2$ with $N\ge 4$, it holds that
\begin{align}\label{energy-control}
\mathcal{E}_{n}  \lesssim \bar{\mathcal{E}}_{n}+\bar{\mathcal{E}}^{\ast}_{n} + \bar{\mathcal{E}}_{n}^{\sharp} +\mathfrak{E}_n  +\bar{\mathfrak{E}}_{n}\ + (\mathcal{E}_{n} )^{2},
\end{align}
where $\mathcal{E}_{n}$, $\bar{\mathcal{E}}_{n}$, $\bar{\mathcal{E}}^{\ast}_{n}$, $\bar{\mathcal{E}}_{n}^{\sharp}$, $\mathfrak{E}_n$, $\bar{\mathfrak{E}}_{n}$ are defined by \eqref{p_energy_defc}, \eqref{Ent}, \eqref{En*}, \eqref{Esharp}, \eqref{y1}, \eqref{y2}, respectively.
\end{lemma}

\begin{proof}
Let $n$ be either $2N$ or $N+2$.

\vskip .1in
{\bf Step 1.} Recalling \eqref{phi2-pa-t}, we have
\begin{align}\label{aa}
\|\nabla_h \partial _{t}^{j}\phi\|^2_{k,\lambda}\lesssim\|\partial_3 \partial _{t}^{j}\phi\|^2_{k,\lambda}+ \sigma\| \partial _{t}^{j} \mathbf{W}\|^2_{0,k+\lambda}+\| \partial _{t}^{j} \mathbf{w}\|_{0}\|\partial _{t}^{j} \mathbf{w}\|_{1}
\end{align}
for any $ \sigma>0 $, where $ 1 \leq j \leq n-1 $, $ 0 \leq k \leq 2n-2j-1 $ and $ 0 \leq \lambda \leq 2n-2j-k-1 $. Then it follows from \eqref{aa} and the definition of the anisotropic Sobolev norm, \eqref{ani-Sobolev}, that
\begin{align}\label{a}
\| \nabla  \partial _{t}^{j}\phi\|^2_{k,\lambda}\lesssim\|\partial_3 \partial _{t}^{j}\phi\|^2_{k,\lambda}+ \sigma\| \partial _{t}^{j} \mathbf{W}\|^2_{0,k+\lambda}+\| \partial _{t}^{j} \mathbf{w}\|_{0,2n-2j-1}\|\partial _{t}^{j} \mathbf{w}\|_{1,2n-2j-1}.
\end{align}
Also notice by \eqref{x4} and \eqref{x5},
\begin{align}\label{b}
 \|\partial _{t}^{j}\mathbf{W}\|^2_{k, 2n-2j-k-1} \lesssim \|\partial _{t}^{j}\boldsymbol{\boldsymbol{\mathcal{V}}}\|^2_{k,2n-2j-k-1}+ \|\partial _{t}^{j}\nabla \phi\|^2_{k,2n-2j-k-1},
\end{align}
and by \eqref{div-curl-esti},
\begin{align}\label{c}
\|\partial _{t}^{j}\mathbf{w}\|^2_{k+1, 2n-2j-k-1} \lesssim \|\partial _{t}^{j}\mathbf{W}\|^2_{k, 2n-2j-k-1} +\|\partial _{t}^{j}\mathop{\mathrm{div}}\nolimits \mathbf{w}\|^2_{k, 2n-2j-k-1}+ \| \partial _{t}^{j} \mathbf{w}\|^2_{0,2n-2j}.
\end{align}
Taking $\lambda=2n-2j-k-1$ in \eqref{a}, combining the result with \eqref{b}, and choosing sufficiently small $\sigma>0$, we can show that
\begin{align}\label{d}
&\,\|\partial^j_{t}\partial _{3}\phi\|_{k,2n-2j-k-1}^{2}\nonumber\\
\gtrsim&\, \|\partial^j_{t}\nabla\phi\|_{k,2n-2j-k-1}^{2} - \sigma\|\partial^j_t \mathbf{W}\|_{0,2n-2j-1}^2 - \varepsilon\|\partial^j_{t}\mathbf{w}\|_{1,2n-2j-1}^2 - C_{ \sigma,\varepsilon}\|\partial^j_t\mathbf{w}\|_{0,2n-2j-1}^2
\nonumber\\
\gtrsim&\, \|\partial^j_t\mathbf{W}\|^2_{k,2n-2j-k-1} - \|\partial _{t}^{j}\boldsymbol{\mathcal{V}}\|^2_{k,2n-2j-k-1} - \varepsilon\|\partial^j_{t}\mathbf{w}\|_{1,2n-2j-1}^2-C_{ \sigma,\varepsilon}\|\partial^j_t\mathbf{w}\|_{0,2n-2j-1}^2,
\end{align}
where $\varepsilon>0$ is a constant to be determined. After arranging terms, we obtain
\begin{align}\label{e}
&\,\|\partial _{t}^{j}\boldsymbol{\mathcal{V}}\|^2_{k,2n-2j-k-1} + \|\partial^j_{t}\partial _{3}\phi\|_{k,2n-2j-k-1}^{2} \nonumber\\
\gtrsim&\, \|\partial^j_t\mathbf{W}\|^2_{k,2n-2j-k-1} - \varepsilon\|\partial^j_{t}\mathbf{w}\|_{1,2n-2j-1}^2-C_{ \sigma,\varepsilon}\|\partial^j_t\mathbf{w}\|_{0,2n-2j-1}^2.
\end{align}
Using \eqref{c}, we derive from \eqref{e} that
\begin{align}\label{f}
&\,\|\partial _{t}^{j}\boldsymbol{\mathcal{V}}\|^2_{k,2n-2j-k-1} + \|\partial^j_{t}\partial _{3}\phi\|_{k,2n-2j-k-1}^{2} + \|\partial _{t}^{j}\mathop{\mathrm{div}}\nolimits \mathbf{w}\|^2_{k, 2n-2j-k-1} \nonumber\\
\gtrsim&\, \|\partial _{t}^{j}\mathbf{w}\|^2_{k+1, 2n-2j-k-1} - \| \partial _{t}^{j} \mathbf{w}\|^2_{0,2n-2j}  - \varepsilon\|\partial^j_{t}\mathbf{w}\|_{1,2n-2j-1}^2-C_{ \sigma,\varepsilon}\|\partial^j_t\mathbf{w}\|_{0,2n-2j-1}^2.
\end{align}
Since by definition,
\begin{align*}
- \varepsilon\|\partial^j_{t}\mathbf{w}\|_{1,2n-2j-1}^2 \ge -\varepsilon \|\partial _{t}^{j}\mathbf{w}\|^2_{k+1, 2n-2j-k-1},
\end{align*}
and\begin{align*}
- C_{ \sigma,\varepsilon} \|\partial^j_t\mathbf{w}\|_{0,2n-2j-1}^2 \ge - C_{ \sigma,\varepsilon} \| \partial _{t}^{j} \mathbf{w}\|^2_{0,2n-2j},
\end{align*}
by choosing sufficiently small $\varepsilon > 0$, we get from \eqref{f} that
\begin{align}\label{fa}
&\,\|\partial _{t}^{j}\mathbf{w}\|^2_{k+1, 2n-2j-k-1} \notag\\
\lesssim &\, \|\partial _{t}^{j}\boldsymbol{\mathcal{V}}\|^2_{k,2n-2j-k-1} + \|\partial^j_{t}\partial _{3}\phi\|_{k,2n-2j-k-1}^{2} + \|\partial _{t}^{j}\mathop{\mathrm{div}}\nolimits \mathbf{w}\|^2_{k, 2n-2j-k-1} + \|\partial^j_t\mathbf{w}\|_{0,2n-2j}^2
\end{align}
for $1\le  j \le n-1 $ and $ 0 \leq k \leq 2n-2j-1 $. It follows from \eqref{fa}, \eqref{y2} and \eqref{En*} that
 \begin{align}\label{w0}
 \sum _{j=1}^{n-1}\|\partial _{t}^{j}\mathbf{w}\|^2_{2n-2j} & \lesssim \sum _{j=1}^{n-1}\big( \|\partial _{t}^{j}\boldsymbol{\mathcal{V}}\|^2_{2n-2j-1} + \|\partial^j_{t}\partial _{3}\phi\|_{2n-2j-1}^{2} + \|\partial _{t}^{j}\mathop{\mathrm{div}}\nolimits \mathbf{w}\|^2_{2n-2j-1} + \|\partial^j_t\mathbf{w}\|_{0,2n-2j}^2 \big)
  \nonumber \\
 & \lesssim \bar{\mathcal{E}}^{\ast}_{n}+\bar{\mathfrak{E}}_n.
 \end{align}
 When $j=0$,  by \eqref{div-curl-esti}, we have
 \begin{align}\label{w1}
 \|\mathbf{w}\|_{2n}^2 \lesssim \|\mathop{\mathrm{div}}\nolimits  \mathbf{w}\|_{2n-1}^2 + \|\mathbf{W}\|_{2n-1}^2 + \|\mathbf{w}\|_{0,2n}^2.
 \end{align}
 According to \eqref{x4}, \eqref{x5} and \eqref{phi2}, it holds that
 \begin{align}
 \|\mathbf{W}\|_{2n-1}^2 \lesssim  \|\boldsymbol{\mathcal{V}}\|^2_{2n-1}+ \|\nabla\phi\|_{2n-1}^2 &\lesssim \|\boldsymbol{\mathcal{V}}\|^2_{2n-1}+\|\partial_3\phi\|_{2n-1}^2 + \sigma\|\mathbf{W}\|_{0,2n-1}^2 +  \|\mathbf{w}\|_0\|\mathbf{w}\|_1 \notag\\
 &\lesssim \|\boldsymbol{\mathcal{V}}\|^2_{2n-1}+\|\partial_3\phi\|_{2n-1}^2 + \sigma\|\mathbf{W}\|_{2n-1}^2 +  \|\mathbf{w}\|_0\|\mathbf{w}\|_1.
 \end{align}
 Choosing $\sigma>0$ sufficiently small, we have
  \begin{align}\label{w2}
 \|\mathbf{W}\|_{2n-1}^2 &\lesssim \|\boldsymbol{\mathcal{V}}\|^2_{2n-1}+\|\partial_3\phi\|_{2n-1}^2 + \|\mathbf{w}\|_0\|\mathbf{w}\|_1\notag\\
 &\lesssim \|\boldsymbol{\mathcal{V}}\|^2_{2n-1}+\|\partial_3\phi\|_{2n-1}^2 + C_\varepsilon\|\mathbf{w}\|_0^2 + \varepsilon\|\mathbf{w}\|_{2n}^2.
 \end{align}
 Choosing $\varepsilon>0$ sufficiently small, and then substituting \eqref{w2} into \eqref{w1} give us
  \begin{align}\label{w3}
 \|\mathbf{w}\|_{2n}^2 \lesssim \|\mathop{\mathrm{div}}\nolimits  \mathbf{w}\|_{2n-1}^2 + \|\boldsymbol{\mathcal{V}}\|^2_{2n-1}+\|\partial_3\phi\|_{2n-1}^2 + \|\mathbf{w}\|_0^2 + \|\mathbf{w}\|_{0,2n}^2.
 \end{align}
 Note that the first three terms on the right of \eqref{w3} are contained in $\mathfrak{E}_n$, $\|\mathbf{w}\|_0^2$ is contained in $\bar{\mathcal{E}}_n$ and $\|\mathbf{w}\|_{0,2n}^2$ is contained in $\bar{\mathcal{E}}^{\ast}_{n}$. Hence, we have
   \begin{align}\label{w4}
 \|\mathbf{w}\|_{2n}^2 \lesssim \mathfrak{E}_n + \bar{\mathcal{E}}_n + \bar{\mathcal{E}}^{\ast}_{n}.
 \end{align}
 Lastly, when $j=n$, we know by definition, $\|\partial_t^n \mathbf{w}\|_0^2$ is contained in $\bar{\mathcal{E}}_n$. Combining \eqref{w0} and \eqref{w4}, we end up with
  \begin{align}\label{w5}
 \sum _{j=0}^{n}\|\partial _{t}^{j}\mathbf{w}\|^2_{2n-2j} \lesssim \bar{\mathcal{E}}_n + \bar{\mathcal{E}}^{\ast}_{n}+ \mathfrak{E}_n + \bar{\mathfrak{E}}_n.
 \end{align}

\vskip .1in
 {\bf Step 2.} To simplify the presentation, we let
\begin{equation}\label{n112}
 \mathcal{X}_n =   \| \bar{\nab}^{2n-2}_0  \mathbf{G}^1\|^2_{0} + \| \bar{\nab}^{2n-2}_0  G^3\|^2_{1}.
\end{equation}
Applying $\partial _{t}^{j}$, $j=1,\dots,n-1$, to the second and  the fourth equations in \eqref{perturbc}, we get
  \begin{equation}\label{stokesp1}
\begin{aligned}
-\mu\Delta \partial^{j}_t \mathbf{u}+\nabla \partial^{j}_t p &=-\partial^{j+1}_t \mathbf{u}+ \p_{3}^2 \partial^j_{t} \boldsymbol{\eta} + \zeta\partial^j_{t}\mathbf{W}+\partial^j_{t} \mathbf{G}^1,\\
\diverge \partial^{j}_t \mathbf{u} &=\partial^{j}_t G^3,\\
\dt^{j} \mathbf{u}|_{\partial\Omega} & =\mathbf{0}.
\end{aligned}
\end{equation}
Employing Lemma \ref{i_linear_elliptic2} to the problem \eqref{stokesp1}, we obtain for $ j=0,1,\dots,n-1 $ that
\begin{align}\label{fes3}
 &\ \|\partial _{t}^{j} \mathbf{u}\|_{2n-2j}^2 +\|\nabla\partial^j_{t}p\|^2_{2n-2j-2} \nonumber \\
  \lesssim &\
 \|\partial _{t}^{j+1} \mathbf{u}\|_{2n-2j-2 }^2+ \|\partial ^{2}_3\partial _{t}^{j}\boldsymbol{\eta}\|_{2n-2j-2 }^2+ \|\partial _{t}^{j}\mathbf{W}\|^2_{2n-2j-2} + \|\partial _{t}^{j}\mathbf{G}^1\|_{2n-2j-2}^2\notag\\
 &\quad +\|\partial^j_{t}G^3\|^2_{2n-2j-1}+\|\partial^j_{t}\mathbf{u}\|^2_{0}\nonumber
  \\
 \lesssim &\ \|\partial _{t}^{j+1}\mathbf{u}\|_{2n-2j-2 }^2 +
  \|\partial _{t}^{j}\boldsymbol{\eta}\|_{2n-2j  }^2 + \|\partial _{t}^{j}\mathbf{W}\|^2_{2n-2j-2} +\bar{\mathcal{E}}_n+ \mathcal{X}_n.
\end{align}
 By an induction on \eqref{fes3} and the facts $\partial_t \boldsymbol{\eta}=\mathbf{u}$ and $\|\boldsymbol\eta\|_{2n}^2 \lesssim \mathfrak{E}_n$, we can show that
 \begin{align}\label{fes4}
&\ \sum_{j=0}^{n} \|\partial _{t}^{j}\mathbf{u}\|_{2n-2j }^2+\sum_{j=0}^{n-1}\|\nabla\partial^j_{t}p\|^2_{2n-2j-2}
 \nonumber \\
 \lesssim &\ \|\partial _{t}^{n}\mathbf{u}\|^2_{0}  +\sum_{j=0}^{n-1} \|\partial _{t}^{j}\boldsymbol{\eta}\|_{2n-2j }^2+\sum_{j=0}^{n-1} \|\partial _{t}^{j}\mathbf{W}\|_{2n-2j-2}^2 +\bar{\mathcal{E}}_n +\mathcal{X}_n\nonumber
\\
\lesssim &\ \|\boldsymbol\eta\|_{2n  }^2 +\sum_{j=1}^{n-1} \|\partial _{t}^{j-1}\mathbf{u}\|_{2n-2j }^2    +\sum_{j=0}^{n-1} \|\partial _{t}^{j}\mathbf{W}\|_{2n-2j-2}^2+ \bar{\mathcal{E}}_n + \mathcal{X}_n\nonumber
\\
\lesssim &\ \sum_{j=0}^{n-2} \|\partial _{t}^{j}\mathbf{u}\|_{2n-2j-2}^2 +\bar{\mathcal{E}}_n + \bar{\mathcal{E}}_n^\ast +\mathfrak{E}_n+\bar{\mathfrak{E}}_n + \mathcal{X}_n,
\end{align}
where we used \eqref{w5} for the estimate of $\|\partial _{t}^{j}\mathbf{W}\|_{2n-2j-2}^2=\|\partial _{t}^{j}\nabla\times \mathbf{w}\|_{2n-2j-2}^2$. By interpolation and Young's inequalities, we get from \eqref{fes4} that
\begin{align*}
  \sum_{j=0}^{n} \|\partial _{t}^{j}\mathbf{u}\|_{2n-2j }^2 +\sum_{j=0}^{n-1}\|\nabla\partial^j_{t}p\|^2_{2n-2j-2}
 &\lesssim  \sum_{j=0}^{n-2} \|\partial _{t}^{j}\mathbf{u}\|_{0}^2+ \bar{\mathcal{E}}_n + \bar{\mathcal{E}}_n^\ast +\mathfrak{E}_n+ \bar{\mathfrak{E}}_n +\mathcal{X}_n \notag\\
 &\lesssim   \bar{\mathcal{E}}_n  + \bar{\mathcal{E}}_n^\ast +\mathfrak{E}_n+ \bar{\mathfrak{E}}_n +\mathcal{X}_n.
 \end{align*}
 This, combined with \eqref{Esharp}, \eqref{y1}, \eqref{w5} and the estimate $ \mathcal{X}_n \lesssim (\mathcal{E}_{n} )^2$ from \eqref{p_G_e_0c}, gives \eqref{energy-control}.
\end{proof}

Now we are ready to recover the energy estimates in terms of $ \mathcal{E}_{n} $ and $ \mathcal{D}_{n} $.

\begin{lemma}\label{lem-combined-esti}
Let $ (\boldsymbol{\eta},\mathbf{u}, p,\mathbf{w}) $ be the local solution to \eqref{reformulationc} satisfying \eqref{a-priori} for some sufficiently small constant $\delta$. Then for $n=N+2$ or $n=2N$ with $N\ge4$, there exists an energy functional $\tilde{ \mathcal{E}}_{n}\sim \mathcal{E}_{n}$ such that
\begin{align}
\bullet\quad &\,\frac{\mathrm{d}}{\mathrm{d}t}\tilde{ \mathcal{E}}_{2N}+ {\mathcal{D}}_{2N}  \lesssim \sqrt{ \mathcal{E}_{N+2}}({\mathcal{J}}_{2N}+\mathcal{F}_{2N}),\label{sys2nc}\\
\bullet\quad &\,\frac{\mathrm{d}}{\mathrm{d}t}\tilde{ \mathcal{E}}_{N+2}+ {\mathcal{D}}_{N+2}  \lesssim 0,\label{sysn+2c}
\end{align}
provided $ \delta $ is sufficiently small, where $\mathcal{D}_n$ is defined by \eqref{p_dissipation_defc} and $\mathcal{J}_{2N}$, $\mathcal{F}_{2N}$ are defined by \eqref{FJ}.
\end{lemma}

\begin{proof}
{\bf Step 1.} Applying Lemma \ref{i_linear_elliptic2} with $r=2n-2j+1\ge 3$ to \eqref{stokesp1} and using \eqref{n1} along with \eqref{Dnt}, we can show that
\begin{align}\label{ffes1}
 &\ \|\partial_t^{j}  \mathbf{u}\|_{2n-2j+1}^2 + \|\nabla \partial_t^{j}  p\| _{2n-2j-1}^2\nonumber
 \\
  \lesssim &\ \|\partial_{t}^{j+1} \mathbf{u}\|_{2n-2j-1 }^2+ \|\partial_{3}^2 \partial_t^{j }\boldsymbol{\eta}\|_{2n-2j-1 }^2+\|\partial^j_{t}\mathbf{W}\|^2_{2n-2j-1} + \|\partial_{t}^j \mathbf{G}^1  \|_{2n-2j-1}^2 + \|\partial_t^{j}  G^3\|_{2n-2j}^2+ \|\partial_t^{j}  \mathbf{u} \|_{0}^2\nonumber
\\
\lesssim &\ \|\partial_t^{j+1} \mathbf{u} \|_{2n-2j-1 }^2 + \|\partial_t^{j }\boldsymbol{\eta}\|_{2n-2 j +1 }^2+\|\partial^j_{t}\mathbf{W}\|^2_{2n-2j-1} +\bar{\mathcal{D}}_n +  \mathcal{Y}_n .
\end{align}
Since $\dt \boldsymbol{\eta}=\mathbf{u}$, by an induction on \eqref{ffes1} and \eqref{Dnt}, we get
\begin{align}\label{eses11}
 &\ \sum_{j=1}^{n} \|\partial_t^{j} \mathbf{u} \|_{2n-2j+1}^2+  \sum_{j=1}^{n-1}\|\nabla \partial_t^{j}  p\|_{2n-2j-1}^2 \lesssim \sum_{j=1}^{n-1} \big( \|\partial_t^{j }\boldsymbol{\eta}\| _{2n-2 j +1 }^2 + \|\partial^j_{t}\mathbf{W}\|^2_{2n-2j-1}\big) +\bar{\mathcal{D}}_n+\mathcal{Y}_n \nonumber
\\
 \lesssim  &\ \sum_{j=0}^{n-2}\|\partial_t^{j } \mathbf{u}\|_{2n-2j-1 }^2+\sum^{n-1}_{j=1}\|\partial^j_{t}\mathbf{W}\|^2_{2n-2j-1}   +\bar{\mathcal{D}}_n + \mathcal{Y}_n.
\end{align}
For reader's convenience, let us recall \eqref{QV1-final} and \eqref{M12}:
\begin{align}
&\ \frac{\mathrm{d}}{\mathrm{d}t} \mathfrak{E}_n +\sum^{1}_{l= 0}\big(\|\partial^l_{t}\mathbf{W}\|^2_{2n-1}+\|\partial_t^{l}\mathbf{w}\|^2_{2n}+\|\partial _{t}^{l}\boldsymbol{\boldsymbol{\mathcal{Q}}}\|_{2n-1} ^{2}\big)+\|\boldsymbol{\eta}\|^2_{2n}+\|\mathbf{u}\|^2_{2n+1}+\|\nabla p\|^2_{2n-2}
\nonumber\\
\lesssim &\ \|\partial _{t}\mathbf{u}\|_{2n-1}^2 + \bar{\mathcal{D}}_n^\ast+\bar{\mathcal{D}}_n^\sharp+\mathcal{Y}_n ,\label{R1}
\end{align}
and
\begin{align}
&\ \frac{\mathrm{d}}{\mathrm{d}t} \bar{\mathfrak{E}}_{n}
 +\sum^{n-1}_{j=1}  \sum^{1}_{l=0}\big(\|\partial^   {j+l}_t\mathbf{W}\|^2_{2n-2j-1}+\|\partial^   {j+l}_t\mathbf{w}\|^2_{2n-2j}\big)+ \sum^{n-1}_{j=1}\| \partial _{t}^{j}\mathbf{u}\| _{2n-2j+1}^{2}\nonumber\\
 \lesssim &\ \bar{\mathcal{D}}_n+ \bar{\mathcal{D}}_n^\ast+\mathcal{Y}_{n}.\label{R2}
\end{align}
By manipulating \eqref{eses11} and \eqref{R2}, we can show that
\begin{align}
&\ \frac{\mathrm{d}}{\mathrm{d}t} \bar{\mathfrak{E}}_{n}
 +\sum_{j=1}^{n} \|\partial_t^{j} \mathbf{u} \|_{2n-2j+1}^2 + \sum_{j=1}^{n-1}\|\nabla \partial_t^{j}  p\|_{2n-2j-1}^2 + \sum^{n-1}_{j=1} \sum^{1}_{l=0}\big(\|\partial^   {j+l}_t\mathbf{W}\|^2_{2n-2j-1}+\|\partial^   {j+l}_t\mathbf{w}\|^2_{2n-2j}\big)\nonumber\\
 \lesssim &\ \sum_{j=0}^{n-2}\|\partial_t^{j } \mathbf{u}\|_{2n-2j-1 }^2 + \bar{\mathcal{D}}_n+ \bar{\mathcal{D}}_n^\ast+\mathcal{Y}_{n}.\label{R3}
\end{align}
Taking the sum of \eqref{R1} and \eqref{R3} and letting $ \mathfrak{E}_{n}^{\ast} := \mathfrak{E}_{n}+\bar{\mathfrak{E}}_{n} $, we obtain
\begin{align}\label{fes5}
 &\frac{\mathrm{d}}{\mathrm{d}t}\mathfrak{E}_{n}^{\ast} + \sum_{j=0}^{n} \|\partial_t^{j} \mathbf{u} \|_{2n-2j+1}^2 + \sum_{j=1}^{n-1}\|\nabla \partial_t^{j}  p\|_{2n-2j-1}^2 + \sum^{n-1}_{j=0} \sum^{1}_{l=0}\big(\|\partial^   {j+l}_t\mathbf{W}\|^2_{2n-2j-1}+\|\partial^   {j+l}_t\mathbf{w}\|^2_{2n-2j}\big) \notag\\
 &\qquad\ +\sum^{1}_{l= 0} \|\partial _{t}^{l}\boldsymbol{\boldsymbol{\mathcal{Q}}}\|_{2n-1} ^{2} +\|\boldsymbol{\eta}\|^2_{2n} +\|\nabla p\|^2_{2n-2} \notag\\
 \lesssim & \sum_{j=0}^{n-2}\|\partial_t^{j } \mathbf{u}\|_{2n-2j-1 }^2 + \bar{\mathcal{D}}_n+ \bar{\mathcal{D}}_n^\ast + \bar{\mathcal{D}}_n^\sharp+\mathcal{Y}_n.
\end{align}
Recalling \eqref{p_dissipation_defc} and $\boldsymbol{\boldsymbol{\mathcal{Q}}}=\mathcal{P} (\partial_3^2\boldsymbol{\eta})$, we see that the dissipation on the right of \eqref{fes5} contains the quantities in $\mathcal{D}_n$, except $\|\partial_t^n \mathbf{w}\|_0^2$, $\|\partial_t^{n+1} \mathbf{w}\|_0^2$ and $\|\partial_3\boldsymbol{\eta}\|_{0,2n}^2$. By adding these terms to the left of \eqref{fes5} and noticing that $\|\partial_t^n \mathbf{w}\|_0^2$ and $\|\partial_t^{n+1} \mathbf{w}\|_0^2$ are contained in $\bar{\mathcal{D}}_n$ (see \eqref{Dnt}) and $\|\partial_3\boldsymbol{\eta}\|_{0,2n}^2$ is contained in $\bar{\mathcal{D}}_n^\sharp$ (see \eqref{Esharp}), we obtain
\begin{align}\label{R4}
 \frac{\mathrm{d}}{\mathrm{d}t}\mathfrak{E}_{n}^{\ast} + \mathcal{D}_n \lesssim \sum_{j=0}^{n-2}\|\partial_t^{j } \mathbf{u}\|_{2n-2j-1 }^2 + \bar{\mathcal{D}}_n+ \bar{\mathcal{D}}_n^\ast + \bar{\mathcal{D}}_n^\sharp+\mathcal{Y}_n.
\end{align}
Since $\mathcal{D}_n$ contains $\|\partial_t^{j} \mathbf{u} \|_{2n-2j+1}^2$, which has two more spatial derivatives than $\|\partial_t^{j} \mathbf{u} \|_{2n-2j-1}^2$, by interpolation and Young's inequalities, we can improve \eqref{fes5} to be
\begin{gather}\label{fes6}
 \frac{\mathrm{d}}{\mathrm{d}t}\mathfrak{E}_n ^{\ast}  + \mathcal{D}_n
 \lesssim \sum_{j=0}^{n-2}\|\partial_t^{j } \mathbf{u}\|_0^2+\bar{\mathcal{D}}_{n}+\bar{\mathcal{D}}_{n}^{\ast}+\bar{\mathcal{D}}_{n}^{\sharp} +\mathcal{Y}_{n} \lesssim \bar{\mathcal{D}}_{n}+\bar{\mathcal{D}}_{n}^{\ast}+\bar{\mathcal{D}}_{n}^{\sharp}+\mathcal{Y}_{n},
\end{gather}
where we used \eqref{Dnt} when deriving the last relationship.

\vskip .1in
{\bf Step 2.} For simplicity, we denote $\mathcal{Z}_{n} $ for $n=2N$ or $n=N+2$ by
\begin{equation}\label{R5}
\mathcal{Z}_{2N}:=\sqrt{ \mathcal{E}_{N+2}} (\mathcal{D}_{2N}+\mathcal{J}_{2N}+\mathcal{F}_{2N}) \quad \text{ and } \quad \mathcal{Z}_{N+2}:=\sqrt{\mathcal{E}_{2N} }\mathcal{D}_{N+2}.
\end{equation}
Note that since $N\ge4$,
$$
\sqrt{\mathcal{E}_{N+2}} \mathcal{D}_n = \left\{\begin{aligned}
&\sqrt{\mathcal{E}_{N+2}} \mathcal{D}_{2N} \\
&\sqrt{\mathcal{E}_{N+2}} \mathcal{D}_{N+2}
\end{aligned}
\right. \quad
\le \quad \left. \begin{aligned}
& \mathcal{Z}_{2N}, &n&=2N\\
& \mathcal{Z}_{N+2}, &n&=2+N
\end{aligned}
\right\} = \mathcal{Z}_n.
$$
Then we write \eqref{barEN-differ-esti} as
\begin{align}\label{R6}
 \frac{\mathrm{d}}{\mathrm{d}t}\Big(\bar{\mathcal{E}}_{n}+\int_{\Omega}\nabla\partial^{n-1}_tp \cdot \mathbf{Q}^{3,n}\mathrm{d}\mathbf{x}\Big)+\bar{\mathcal{D}}_{n} \lesssim \mathcal{Z}_{n}.
\end{align}
Using \eqref{R5}, we write \eqref{p_u_e_00c}--\eqref{p_u_e_00c1} and \eqref{p_u_e_00'132c}--\eqref{p_u_e_00'132c2} in the unified fashion as
\begin{align}\label{R7}
\frac{\mathrm{d}}{\mathrm{d}t}\bar{\mathcal{E}}_{n}^{\ast}+\bar{\mathcal{D}}_{n}^{\ast} \lesssim  \mathcal{Z}_n,
\end{align}
and
\begin{align}\label{R8}
\frac{\mathrm{d}}{\mathrm{d}t} \widetilde{\bar{\mathcal{E}}^\sharp_{n}} +\bar{\mathcal{D}}_{n}^\sharp \lesssim \mathcal{Z}_n +\bar{\mathcal{D}}_{n}^{\ast},
\end{align}
where
\begin{align}\label{R9}
\widetilde{\bar{\mathcal{E}}^\sharp_{n}} := \bar{\mathcal{E}}^\sharp_{n}+\sum_{\alpha\in \mathbb{Z}_+^2, \, |\alpha| \le 2n}\int_\Omega \big( 2\partial^\alpha \mathbf{u} \cdot \partial^\alpha \boldsymbol{\eta} + \partial^{\alpha}\mathbf{w}\cdot(\nabla\times\partial^{\alpha}\boldsymbol{\eta})\big)\mathrm{d}\mathbf{x}.
\end{align}
For any $ \epsilon >0 $, multiplying \eqref{R8} by $\epsilon$, \eqref{fes6} by $\epsilon^2$, then taking the summation of the results with \eqref{R6} and \eqref{R7}, we obtain
\begin{align}\label{R10}
\frac{\mathrm{d}}{\mathrm{d}t}\tilde{\mathcal{E}}_{n} +\bar{\mathcal{D}}_{n}+\bar{\mathcal{D}}_{n}^{\ast}+\epsilon \bar{\mathcal{D}}_{n}^{\sharp}+\epsilon ^{2}\mathcal{D}_{n} \lesssim \mathcal{Z}_{n}+\epsilon\bar{\mathcal{D}}_{n}^{\ast}+\epsilon ^{2}(\bar{\mathcal{D}}_{n}+\bar{\mathcal{D}}_{n}^{\ast}+\bar{\mathcal{D}}_{n}^{\sharp} + \mathcal{Y}_n),
\end{align}
where
\begin{align*}
\tilde{\mathcal{E}}_{n} :=  \bar{\mathcal{E}}_{n} + \int_{\Omega}\nabla\partial^{n-1}_tp\cdot \mathbf{Q}^{3,n}\mathrm{d}\mathbf{x} +\bar{\mathcal{E}}_{n}^{\ast}+\epsilon\widetilde{\bar{\mathcal{E}}^\sharp_{n}} + \epsilon ^{2} \mathfrak{E}_n ^{\ast}.
\end{align*}
Note that according to \eqref{esti-Y-2N} and \eqref{R5}, under the {\it a priori} assumption \eqref{a-priori}, it holds that $\mathcal{Y}_n \lesssim \mathcal{Z}_n$. By choosing $ \epsilon>0 $ small enough, we get from \eqref{R10} that
\begin{align}\label{combined-diff}
\frac{\mathrm{d}}{\mathrm{d}t}\tilde{\mathcal{E}}_{n} + \bar{\mathcal{D}}_{n}+\bar{\mathcal{D}}_{n}^{\ast}+\epsilon \bar{\mathcal{D}}_{n}^{\sharp}+\epsilon^{2}\mathcal{D}_{n}\lesssim  \mathcal{Z}_{n}.
\end{align}
Note that according to H\"older's inequality, \eqref{p_energy_defc} and \eqref{p_F_e_02}, we have
\begin{equation*}
\int_\Omega  \nabla \dt^{n-1} p\cdot \mathbf{Q}^{3,n}\mathrm{d}\mathbf{x}\ls \sqrt{\mathcal{E}_n}\sqrt{\mathcal{E}_{N+2}\mathcal{E}_n}=\sqrt{\mathcal{E}_{N+2}}\mathcal{E}_n,
\end{equation*}
which implies
\begin{align}\label{R11}
\tilde{\mathcal{E}}_{n} \gtrsim  \bar{\mathcal{E}}_{n} + \bar{\mathcal{E}}_{n}^{\ast} +\epsilon\widetilde{\bar{\mathcal{E}}^\sharp_{n}} + \epsilon ^{2} \mathfrak{E}_n ^{\ast} - \sqrt{\mathcal{E}_{N+2}}\mathcal{E}_n.
\end{align}
In addition, by H\"older's and Poincar\'e's (applied to $\eta$) inequalities, we have
\begin{align}\label{R12}
\big| \sum_{\alpha\in \mathbb{Z}_+^2, \, |\alpha| \le 2n}\int_\Omega \big( 2\partial^\alpha \mathbf{u} \cdot \partial^\alpha \boldsymbol{\eta} &+ \partial^{\alpha}\mathbf{w}\cdot(\nabla\times\partial^{\alpha}\boldsymbol{\eta})\big)\mathrm{d}\mathbf{x} \big| \le 2\|\mathbf{u}\|_{0,2n}\|\boldsymbol{\eta}\|_{0,2n} + \|\mathbf{w}\|_{0,2n} \|\nabla\boldsymbol{\eta}\|_{0,2n} \notag\\
&\le \frac{\mu}{2}\|\nabla\boldsymbol{\eta}\|_{0,2n}^2 + C_\mu\big(\|\mathbf{u}\|_{0,2n}^2+\|\mathbf{w}\|_{0,2n}^2\big) \le \frac{1}{2} \bar{\mathcal{E}}^\sharp_{n} + C_\mu \bar{\mathcal{E}}_n^*,
\end{align}
where the definitions \eqref{Esharp} and \eqref{En*} are applied. By \eqref{R12}, we get from \eqref{R9}:
\begin{align}\label{R13}
\widetilde{\bar{\mathcal{E}}^\sharp_{n}} \ge \frac{1}{2} \bar{\mathcal{E}}^\sharp_{n} - C_\mu \bar{\mathcal{E}}_n^*.
\end{align}
Substituting \eqref{R13} into \eqref{R11}, fixing a sufficiently small $\epsilon$ and recalling $\mathfrak{E}_{n}^{\ast} = \mathfrak{E}_{n}+\bar{\mathfrak{E}}_{n} $, we obtain
\begin{align}\label{R14}
\tilde{\mathcal{E}}_{n} \gtrsim  \bar{\mathcal{E}}_{n} + \bar{\mathcal{E}}_{n}^{\ast} + \bar{\mathcal{E}}^\sharp_{n} + \mathfrak{E}_{n}+\bar{\mathfrak{E}}_{n}  - \sqrt{\mathcal{E}_{N+2}}\mathcal{E}_n.
\end{align}
Combining \eqref{energy-control} and \eqref{R14}, we have
\begin{align}\label{R15}
\mathcal{E}_{n}  \lesssim \tilde{\mathcal{E}}_{n} + \sqrt{\mathcal{E}_{N+2}}\mathcal{E}_n+ (\mathcal{E}_{n} )^{2}.
\end{align}
Since $\mathcal{E}_{N+2} \le \mathscr{G}_{2N}$, under the {\it a priori} assumption \eqref{a-priori}, we obtain
\begin{align}\label{R17}
\mathcal{E}_{n}  \lesssim \tilde{\mathcal{E}}_{n} + (\mathcal{E}_{n} )^{2},
\end{align}
which, along with the smallness of $\delta$, implies $\mathcal{E}_{n} \lesssim  \tilde{\mathcal{E}}_{n} $. Since apparently $ \tilde{\mathcal{E}}_{n}\lesssim \mathcal{E}_{n} $, we have $\tilde{\mathcal{E}}_{n} \sim \mathcal{E}_{n}$. Since $\mathcal{Z}_{2N} = \sqrt{ \mathcal{E}_{N+2}} \mathcal{D}_{2N}+\sqrt{ \mathcal{E}_{N+2}} (\mathcal{J}_{2N}+\mathcal{F}_{2N})$, $\mathcal{Z}_{N+2} = \sqrt{\mathcal{E}_{2N}} \mathcal{D}_{N+2}$ and $\mathcal{E}_{N+2} \le \mathcal{E}_{2N} \le \mathscr{G}_{2N}$, we get \eqref{sys2nc} and \eqref{sysn+2c} from \eqref{combined-diff} when $ \delta $ is sufficiently small. The proof is complete.
\end{proof}

The next two lemmas are the last part of the entire proof of Theorem \ref{thm-main-result}, which provide the estimates of $ \mathcal{F}_{2N} $, $\mathcal{J}_{2N}$ and $\mathscr{G}_{2N}$. The proofs of these lemmas are in the spirit of \cite{Tan-Wang-SIMA}. However, since the problem under consideration is more complicated than the previous case, we still present the technical details for self-containment.

\begin{lemma}\label{Lemm-E-2N}
Let $ (\boldsymbol{\eta},\mathbf{u}, p,\mathbf{w}) $ be the local solution to \eqref{reformulationc} satisfying \eqref{a-priori} for some sufficiently small constant $\delta$. There exists a generic constant $ 0<\delta\ll 1 $ such that if $ \mathscr{G}_{2N}(T)\leq \delta $, then
\begin{align}\label{con-E-N-2-lem}
\mathcal{F}_{2N}(t)\lesssim \mathcal{F}_{2N}(0)+\sup_{0\leq \tau\leq t}\mathcal{E}_{2N}(\tau)+\int^t_{0}\mathcal{D}_{2N}(\tau)\mathrm{d}\tau,
\end{align}
and for any $\theta>0$,
\begin{align}\label{grow2}
\int^t_{0}\frac{\mathcal{F}_{2N}+\mathcal{J}_{2N}}{(1+\tau)^{1+\theta}}\mathrm{d} \tau\lesssim\mathcal{F}_{2N}(0)+\sup_{0\leq \tau\leq t}\mathcal{E}_{2N}(\tau)+\int^t_{0}\mathcal{D}_{2N}(\tau)\mathrm{d}\tau.
\end{align}
\end{lemma}
\begin{proof}
By the same arguments as in the derivation of \eqref{3.115} (basically, begin with replacing $k+1$ by $k+2$ in \eqref{est-eta}), we can show that for $k=1,\dots,2n-1$, there exists an energy functional $\tilde{\mathcal{F}}_{2N} \sim \mathcal{F}_{2N}$ such that
\begin{align}\label{hti-F-esit}
&\ \frac{\mathrm{d}}{\mathrm{d}t}\tilde{\mathcal{F}}_{2N}
+  \mathcal{F}_{2N} +\mathcal{J}_{2N} \notag\\
 \lesssim &\ \|\boldsymbol{\eta}\|_{1,4N}^2+ \|\partial_t \mathbf{u}\|_{4N-1}^2+\|\mathbf{W}\|^2_{4N-1}
 + \|\mathbf{G}^1\|_{4N-1}^2+\|G^3\|_{4N}^2+\|\operatorname{div}\boldsymbol{\eta} \|_{4N}^2,
\end{align}
where, thanks to \eqref{p_G_e_001c} and \eqref{pgn1}, we have
\begin{equation}
\|\mathbf{G}^1\|_{4N-1}^2+\|G^3\|_{4N}^2+\|\operatorname{div}\boldsymbol{\eta} \|_{4N}^2\lesssim  { \mathcal{E}_{N+2}  }(\mathcal{D}_{2N} +  \mathcal{J}_{2N} +\mathcal{F}_{2N}).
\end{equation}
Then we get from \eqref{hti-F-esit} that
\begin{equation}\label{109}
 \frac{\mathrm{d}}{\mathrm{d}t}\tilde{\mathcal{F}}_{2N}
+  \mathcal{F}_{2N} +\mathcal{J}_{2N}  \lesssim  \mathcal{E}_{2N}+ \mathcal{D}_{2N}
 + { \mathcal{E}_{N+2}  } \mathcal{D}_{2N}\lesssim\mathcal{E}_{2N}+ \mathcal{D}_{2N},
\end{equation}
where we also used the fact that $\mathcal{E}_{N+2}(T) \le \mathscr{G}_{2N}(T) \leq \delta$ and $\delta$ is sufficiently small. Applying Gro\"nwall's inequality to \eqref{109}, we obtain
\begin{align}\label{1001}
  \f  & \ls  \f(0)     {\mathop{\mathrm{e}}}^{-t}+\int_0^t {\mathop{\mathrm{e}}}^{-(t- \tau)}\big(  \se{2N}( \tau)+ \sd{2N}( \tau) \big)\mathrm{d} \tau\nonumber
  \\& \ls \f(0) {\mathop{\mathrm{e}}}^{-t}+\sup_{0\le \tau\le t}\mathcal{E}_{2N}(\tau) \int_0^t {\mathop{\mathrm{e}}}^{-(t- \tau)}\mathrm{d}  \tau+  \int_0^t \sd{2N}( \tau)  \mathrm{d}  \tau,
\end{align}
which gives \eqref{con-E-N-2-lem}. Next, multiplying \eqref{109} by $(1+t)^{-1-\theta}$ for any $\theta>0$, we get
\begin{equation}\label{1110}
\frac{\mathrm{d}}{\mathrm{d}t} \Big( \frac{\tilde{\mathcal{F}}_{2N}}{(1+t)^{1+\theta}}\Big)+(1+\theta)\frac{\tilde{\mathcal{F}}_{2N}}{(1+t)^{2+\theta}}
+  \frac{\f+\mathcal{J}_{2N}}{(1+t)^{1+\theta}} \ls   \frac{\se{2N}}{(1+t)^{1+\theta}}+  \frac{\sd{2N}}{(1+t)^{1+\theta}} .
\end{equation}
 Integrating \eqref{1110} in time yields \eqref{grow2}. This completes the proof of the lemma.
\end{proof}

The last lemma closes the entire \emph{a priori} estimates of the local solution.

\begin{lemma} \label{Dgle}
Let $ (\boldsymbol{\eta},\mathbf{u}, p,\mathbf{w}) $ be the local solution to \eqref{reformulationc} satisfying \eqref{a-priori} for some sufficiently small constant $\delta$. There exists a generic constant $0<\delta \ll1$ such that if $ \mathscr{G}_{2N}(T)\leq \delta $, then
\begin{align}
\bullet&\quad \mathcal{E}_{2N} (t)+\int_0^{t}\mathcal{D}_{2N}(\tau)\mathrm{d}\tau \lesssim
\mathcal{E}_{2N} (0) + \mathcal{F}_{2N}(0), \quad \forall\ 0\leq t\leq T, \label{Dg}\\
\bullet&\quad (1+t)^{2N-4} \mathcal{E}_{N+2} (t)\lesssim
\mathcal{E}_{2N} (0),  \quad \forall\ 0\le t\le T. \label{n+2c}
\end{align}
\end{lemma}

\begin{proof}
Integrating \eqref{sys2nc} over $(0,t)$, we find that
\begin{equation}
  {\mathcal{E}}_{2N}(t)
+ \int_0^t{\mathcal{D}}_{2N}(\tau) \mathrm{d}\tau
\ls {\mathcal{E}}_{2N}(0)+\int_0^t \sqrt{ \se{N+2}  }(  \mathcal{J}_{2N} +\mathcal{F}_{2N})\mathrm{d}\tau.
\end{equation}
By the assumption $ \mathscr{G}_{2N}(T)\leq \delta $ and \eqref{grow2}, we can show that
\begin{align}
  {\mathcal{E}}_{2N}(t)
+ \int_0^t{\mathcal{D}}_{2N}(\tau) \mathrm{d}\tau
&\lesssim {\mathcal{E}}_{2N}(0)+ \int_0^t \sqrt{\delta}(1+ \tau)^{-N+2} \big(\mathcal{J}_{2N} (\tau) + \mathcal{F}_{2N}(\tau)\big)\mathrm{d} \tau
\nonumber\\
& \lesssim {\mathcal{E}}_{2N}(0)+ \sqrt{ \delta } \Big(\mathcal{F}_{2N}(0)+  \sup_{0\leq \tau\leq t}\mathcal{E}_{2N}(\tau)+ \int_0^t \mathcal{D}_{2N}(\tau) \mathrm{d}\tau\Big),
\end{align}
where we have chosen $\theta>0$ such that $N-2\ge 1+\theta$. We remark that this is reason that we require $N\ge 4$ in our main results. Since $\delta$ is sufficiently small, the preceding estimate gives \eqref{Dg}.

\vskip .1in
Recalling \eqref{sysn+2c}, we have
\begin{align}\label{7-55-revised}
 \frac{\mathrm{d}}{\mathrm{d}t}\tilde{ \mathcal{E}}_{N+2}+ {\mathcal{D}}_{N+2}  \le 0.
\end{align}
Recalling the definitions of ${\mathcal{D}}_{N+2} $ and ${\mathcal{E}}_{N+2} $ and Lemma \ref{lem-anisotropi-1}, we know that every term in ${\mathcal{E}}_{N+2} $ can be controlled by the terms in ${\mathcal{D}}_{N+2} $ except $ \|\nabla_\ast^{2(N+2)+1}\boldsymbol\eta \|_{0}$. However, by interpolation, we have
\begin{align} \label{intep0c}
\|\nabla_\ast^{2(N+2)+1}\boldsymbol\eta \|_{0}^{2} \lesssim  \|\nabla_\ast^{2(N+2)}\boldsymbol\eta \|_{0} ^{2\theta} \|\nabla_\ast^{4N+1}\boldsymbol\eta \|_{0}^{2(1-\theta)} \lesssim ( {\mathcal{D}_{N+2} })^\theta({\mathcal{E}_{2N} })^{1-\theta},\text{
where }\theta=\frac{2N-4}{2N-3}.
\end{align}
Hence, we can show that
\begin{equation} \label{intepc}
{\mathcal{E}}_{N+2} \le({\mathcal{D}}_{N+2} )^\theta({\mathcal{E}}_{2N} )^{1-\theta}.
\end{equation}
Now by Lemma \ref{Lemm-E-2N}, we have
\begin{equation}
\sup_{0\le \tau\le t}\mathcal{E}_{2N} (\tau)\lesssim
\mathcal{E}_{2N} (0) :=\mathcal{M}_0.
\end{equation}
This, along with \eqref{intepc}, implies that
\begin{equation} \label{u2c}
\tilde{\mathcal{E}}_{N+2} \lesssim\mathcal{E}_{N+2} \lesssim (\mathcal{D}_{N+2} )^\theta \mathcal{M}_0^{1-\theta}.
\end{equation}
By \eqref{7-55-revised} and \eqref{u2c}, there exists some constant $C>0$
such that
\begin{equation}
\frac{\mathrm{d}}{\mathrm{d}t} \tilde{\mathcal{E}}_{N+2} +\frac{C}{\mathcal{M}_0^\gamma}
(\tilde{\mathcal{E}}_{N+2} )^{1+\mathfrak{m}}\le 0,\ \text{ where } \mathfrak{m} =
\frac{1}{\theta}-1 = \frac{1}{2N-4}.
\end{equation}
Solving this differential inequality, we obtain
\begin{equation} \label{u3c}
\mathcal{E}_{N+2} (t) \lesssim \tilde{\mathcal{E}}_{N+2} (t) \lesssim \frac{\mathcal{M}_0}{\big\{\mathcal{M}_0^\mathfrak{m}
+ \mathfrak{m} C[ \mathcal{E}_{N+2} (0)]^\mathfrak{m}\, t\big\}^{1/\mathfrak{m}} }
{\mathcal{E}}_{N+2} (0).
\end{equation}
With ${\mathcal{E}}_{N+2} (0)\lesssim\mathcal{M}_0 $ and
the fact $1/\mathfrak{m}=2N-4>1$, we obtain from \eqref{u3c} that
\begin{equation}
{\mathcal{E}}_{N+2} (t)\lesssim
\frac{\mathcal{M}_0}{(1+\mathfrak{m} C t)^{1/\mathfrak{m}} }\lesssim
\frac{\mathcal{M}_0}{(1+t)^{2N-4} },
\end{equation}
which implies \eqref{n+2c}. This completes the proof of the lemma.
\end{proof}

In view of Lemma \ref{Lemm-E-2N} and lemma \ref{Dgle} we see that all the five components in the expression of $\mathscr{G}_{2N}(T)$ are bounded by a constant multiple of $\mathcal{E}_{2N}(0) + \mathcal{F}_{2N}(0)$. Hence, the \emph{a priori} assumption \eqref{a-priori} can be fulfilled by choosing $\mathcal{E}_{2N}(0) + \mathcal{F}_{2N}(0)$ to be sufficiently small, and the global estimate \eqref{G-2N-esti-thm} follows. Based on the local well-posedness and \eqref{G-2N-esti-thm}, the proof of Theorem \ref{thm-main-result} is complete.

\section*{Acknowledgements}
Z. Feng was supported in part by the National Natural Science Foundation of China $\#$12101095, the Science and Technology Research Program of Chongqing Municipal Education Commission $\#$KJQN202100517,
the Research Project of Chongqing Education Commission  $\#$CXQT21014, the
 Natural Science Foundation of Chongqing $\#$cstc2021jcyj-msxmX0224, $\#$
cstb2022nscq-msx2878 and the grant of Chongqing Young Experts' Workshop. The
research of G. Hong was partially supported by the National Natural Science Foundation of China $\#$12522111 and $\#$12201221, the Guangdong Provincial Pearl River Talents Program $\#$2023QN10X436, the Guangdong Basic and Applied Basic Research Foundation $\#$2024A1515012306, the Guangzhou Municipal Science and Technology Project $\#$2024A04J3788, and the Fundamental Research Funds for the Central Universities $\#$2025ZYGXZR040. J. Wu was partially supported by the National Science Foundation of USA $\#$2104682, $\#$2309748 and
the AT\&T Foundation at Oklahoma State University.

\section*{Data availability statement} Data sharing is not applicable to this article as no datasets were generated or analyzed during the current study.

\section*{Conflict of interest statement} On behalf of all authors, the corresponding author states that there is no conflict of interest.


\end{document}